\documentclass[11pt,leqno,twoside]{amsart}

\usepackage{enumerate}

\usepackage{amssymb}
\usepackage[english]{babel}
\usepackage{amssymb,amsthm,amsmath,eucal,mathrsfs}
\usepackage{bm}

\usepackage{amsmath}
\usepackage{amsfonts}
\usepackage{amsmath}
\usepackage{amssymb}
\usepackage{graphicx}
\usepackage{lscape}
\usepackage{amstext}
\usepackage{amsthm}
\usepackage{color}
\usepackage{float}
\usepackage{mathrsfs}
\usepackage{epsfig}
\usepackage{url}
\usepackage{fancyhdr}
\usepackage{pspicture}
\usepackage{graphicx}

\usepackage{cite}

\usepackage{amsmath,verbatim}

\usepackage{amsthm}

\usepackage{amssymb}

\usepackage{amsfonts}

\usepackage{dsfont}

\usepackage{hyperref}

\newtheorem{theorem}{Theorem}[section]

\newtheorem{lemma}[theorem]{Lemma}

\newtheorem{proposition}[theorem]{Proposition}

\newtheorem{corollary}[theorem]{Corollary}

\newtheorem{assumption}{Assumption}
\newtheorem*{notation}{Notation}

\theoremstyle{definition}

\newtheorem{definition}[theorem]{Definition}

\newtheorem{remark}[theorem]{Remark}

\numberwithin{equation}{section}

\renewcommand{\Im}{{\ensuremath{\mathrm{Im\,}}}} 

\renewcommand{\Re}{{\ensuremath{\mathrm{Re\,}}}} 

\renewcommand{\div}{\mathrm{div}\,}    

\renewcommand{\O}{\mathcal{O}\,}

\newcommand{\PP}{\mathbb{P}}

\newcommand{\TT}{\boldsymbol{T}}

\newcommand\restr[2]{{
  \left.\kern-\nulldelimiterspace 
  #1 
  \vphantom{\big|} 
  \right|_{#2} 
  }}

\usepackage{eucal}

\newcommand{\mk}[1]{{\color{black}{#1}}}

\title[Electromagnetic Scattering by a Dimer]{
Electromagnetic waves generated by\\ a hybrid dieletric-plasmonic dimer
}

\author[Cao, Ghandriche and Sini]{Xinlin Cao $^*$ Ahcene Ghandriche $^{**}$ and Mourad Sini$^{\ddag}$}
\thanks{$^*$ Department of Applied Mathematics, The Hong Kong Polytechnic University, Hong Kong SAR. Email: xinlin.cao@polyu.edu.hk} 
\thanks{$^{**}$ Nanjing Center for Applied Mathematics, Nanjing, 211135, People's Republic of China. Email: gh.hsen@njcam.org.cn}
\thanks{$^{\ddag}$ RICAM, Austrian Academy of Sciences, Altenbergerstrasse 69, A-4040, Linz, Austria. Email: mourad.sini@oeaw.ac.at. This author is partially supported by the Austrian Science Fund (FWF): P 30756-NBL and P 32660}

\begin{document}


\allowdisplaybreaks

\begin{abstract}

We know that the electric field generated by a plasmonic nano-particle (with negative permittivity) is given as a polarization of the incident electric field. Similarly, the electric field produced by a dielectric nano-particle (with positive but high permittivity) is given as a polarization of the incident magnetic field. In this work, we demonstrate that a hybrid dimer—composed of two closely coupled nano-particles, one plasmonic and the other dielectric—can polarize both the incident electric and magnetic fields. Consequently, such hybrid dimers have the potential to modify both the electric permittivity and magnetic permeability of the surrounding medium. However, this dual modification occurs only when the two nano-particles share common resonant frequencies. We derive the asymptotic expansion of the fields generated by these hybrid dimers in the subwavelength regime for incident frequencies near their shared resonant frequencies.

\end{abstract}


\maketitle

\tableofcontents

\section{Introduction and the main result}\label{SecI}

\subsection{Motivation}

The interaction of light with nanoscale materials has revolutionized the understanding and manipulation of electromagnetic fields at subwavelength scales, paving the way for transformative applications in photonics, sensing, and metamaterials, see \cite{J-J:2016, Bohren-Huffmann,FZS, N-H-Book}. Among the various nanostructures, hybrid dimers—composed of plasmonic and dielectric nano-particles represent a significant advancement due to their ability to interact with both the electric and magnetic components of incident electromagnetic fields. Indeed, plasmonic nano-particles, characterized by their negative permittivity at optical frequencies, generate intense localized electric fields through surface plasmon resonances, effectively polarizing the incident electric field, see \cite{CGS-Optic, maier2007plasmonics, CP, GS-OA, FZS, baffou2009, hao2004electromagnetic}. In contrast, dielectric nano-particles with high positive permittivity are known for their low-loss interaction with the magnetic component of light, driven by displacement currents that produce magnetic dipole resonances, see \cite{Ammari-Li-Zou-2, PD-L-K, K-M-B-K-L:2016, TS2018}. When combined in a hybrid dimer, these nano-particles exhibit coupled electromagnetic responses, allowing simultaneous polarization of both the electric and magnetic fields. This unique dual interaction makes hybrid dimers a promising platform for engineering media with tailored electric permittivity and magnetic permeability \cite{ammari2018mathematical, Kiselev-el-al, Wang-et-al}.
\bigskip

\noindent A critical feature of hybrid dimers is their ability to achieve resonant frequency alignment between the plasmonic and dielectric components. This resonance matching enhances their electromagnetic coupling, enabling the formation of hybrid modes with strong field localization and enhancement. Recent studies have demonstrated that these hybrid modes can lead to significant field amplification in the gap between nano-particles, often referred to as "hot spots," which are central to applications such as surface-enhanced Raman spectroscopy (SERS), nanoscale biosensing, and nonlinear optics \cite{ammari2018mathematical, Wang-et-al, Liu-et-al}.
\bigskip

\noindent In this work, we explore the theoretical and computational aspects of hybrid dimer interactions in the subwavelength regime. By deriving the asymptotic expansion of the fields generated near the common resonant frequencies of plasmonic and dielectric nano-particles, we aim to elucidate the underlying mechanisms driving the dual polarization effects. This approach provides a robust framework for understanding how hybrid dimers can modify the electromagnetic properties of the surrounding medium, offering new insights into their role in advanced photonic devices.
 Below, we list few advantages of using hybrid dimers over homogeneous dimers of the form plasmonic-plasmonic or dielectric-dielectric dimers.
\begin{enumerate}
\item[] 
\item \textit{Field Polarization Mechanisms by hybrid Dimers}.
These hybrid dimers combine the complementary nature of plasmonic (electric field-driven) and dielectric (magnetic field-driven) responses, creating dual polarization effects. The combination allows for both the electric permittivity $\epsilon$ and magnetic permeability $\mu$ modulation, which is unique to hybrid structures. In a next work, we will analyze with more details, the case when we have a cluster of such dimers, regularly arranged in a given bounded domain, and show how the generated effective medium is a modulation, by averaging, of both the permittivity and permeability offering a way how to design single or double negative electromagnetic media.  
\item[]
\item \textit{Field Localization and Enhancement}. Hybrid dimers exhibit enhanced field localization and polarization in the gap region, benefiting from the synergistic effects of plasmonic and dielectric components. Indeed, in the case of homogeneous dimers, as plasmonic-plasmonic, strong electric field enhancements are localized in the gap, but losses due to ohmic heating can dampen efficiency. Regarding dielectric-dielectric dimers, we have relatively moderate enhancement but with less ohmic loss in comparison to plasmonic counterparts. Therefore, using heterogeneous dimers might improve both the enhancement, as for plasmonics, and reduce the ohmic loss, as for dielectrics.
\item[]
\item \textit{Potential applications}. 
The ability of hybrid dimers to modulate both $\epsilon$ and $\mu$ enables applications in designing metamaterials with tunable refractive indices, broadband absorbers, and devices requiring simultaneous electric and magnetic field control. Homogeneous plasmonic-plasmonic dimers are widely used in SERS, photothermal therapy, and plasmonic sensing due to their strong electric field enhancements, while dielectric-dielectric dimers are ideal for applications requiring low loss, such as photonic waveguides and resonators with high-Q factors. Therefore, hybrid dimers can offers possibilities to be used in both the mentioned applications as they share the qualities and avoid their disadvantages, to some extent. 

\end{enumerate}

\subsection{Main results}

Let $D_1$ and $D_2$ be two bounded and $C^2$-regular domains in $\mathbb{R}^3$ and model $D:=D_1\cup D_2$  
 to stand for a dimer composed of two nano-particles $D_1$ and $D_2$. We assume that $D_1$ is a dielectric nano-particle, namely  its permittivity and permeability constants enjoy the following properties:
\begin{equation*}
\epsilon_r^{(1)} \, := \, \frac{\epsilon^{(1)}}{\epsilon_0} \quad \text{with} \quad Re\left( \epsilon^{(1)} \right) \gg 1 \quad \text{and} \quad \epsilon_r^{(1)} \,= \, 1 \quad \text{outside} \quad D_1,
\end{equation*}
 while 
 $\mu_{r}^{(1)} \, := \, \dfrac{\mu^{(1)}}{\mu_{0}} \, = \, 1$ in the whole space $\mathbb{R}^3$. The nano-particle $D_2$ is taken to be a plasmonic one, i.e. with a moderately contrasting relative permittivity $\epsilon_r^{(2)} \, := \, \dfrac{\epsilon^{(2)}}{\epsilon_0}$, enjoying negative real values,
\begin{equation*}
\Re(\epsilon_{r}^{(2)}) \, < \, 0 \quad \text{and} \quad  \epsilon_{r}^{(2)} \, \sim \, 1 \quad  \text{with} \quad \epsilon_{r}^{(2)} \, = \, 1 \quad \text{outside} \quad D_2,  
\end{equation*}
 and a  permeability satisfying  $\mu_r^{(2)} \, := \, \dfrac{\mu^{(2)}}{\mu_0}=1$ in the whole space $\mathbb{R}^3$. More details will be given later on the related quantities. The electromagnetic wave propagation, in the time-harmonic regime, with the presence of the dimer $D$ satisfies 

\begin{equation}\label{U}
\left \{
\begin{array}{llrr}
Curl(E) - i \, k \, \mu_r\,  H = 0, \, \quad & \mbox{in } \mathbb{R}^3, \\
&  \\
Curl(H) + i \, k \, \varepsilon_{r} \, E = 0, \, \quad & \mbox{in } \mathbb{R}^3,
\end{array} 
\right.
\end{equation}
\medskip
\newline 
where the total field $(E, H)$ is of the form $(E := E^{Inc} +E^s, H := H^{Inc} +H^s)$ and the incident plane wave $(E^{Inc},H^{Inc})$ is of the form
\begin{equation}\label{EincHinc}
E^{Inc}(x,\theta, \mathsf{p}) \, = \, \mathsf{p}  \, e^{i \, k \; \theta \cdot x } \quad \text{and} \quad
H^{Inc}(x,\theta, \mathsf{p}) \, = \, \left( \theta \times \mathsf{p} \right) \, e^{i \, k \; \theta \cdot x},
\end{equation}
with $\theta, \mathsf{p} \in \mathbb{S}^{2}$, $\mathbb{S}^{2}$ being the unit sphere, such that $\theta \cdot \mathsf{p} \, = \, 0$, as the direction of incidence and polarization respectively, and the scattered field $(E^s, H^s)$ satisfies the Silver-M\"{u}ller radiation condition (SMRC) at infinity
\begin{equation}\label{Radiation-conditions}
	\sqrt{\mu_0\epsilon_0^{-1}}H^{s}(x)\times \frac{x}{|x|}-E^s(x) \, = \, \mathcal{O}\left( \frac{1}{|x|^2} \right).
\end{equation}

This problem is well-posed in appropriate Sobolev spaces, see \cite{colton2019inverse} and \cite{Mitrea}, and we have the following behaviors
\begin{equation}\label{def-far}
	E^s(x) \, = \, \frac{e^{i k|x|}}{|x|}\left(E^\infty(\hat{x}) \, + \, \mathcal{O}\left( \frac{1}{|x|} \right)\right), \quad \mbox{as}\quad |x|\rightarrow \infty,
\end{equation}
and
\begin{equation}\notag
	H^s(x) \, = \, \frac{e^{i k|x|}}{|x|}\left(H^\infty(\hat{x}) \, + \, \mathcal{O}\left( \frac{1}{|x|} \right)\right), \quad \mbox{as}\quad |x|\rightarrow \infty,
\end{equation}
where $(E^\infty(\hat{x}), H^\infty(\hat{x}))$ is the corresponding electromagnetic far-field pattern of \eqref{U} in the propagation direction $\hat{x}:=\frac{x}{|x|}$.
\bigskip
\newline 
Next, we present the necessary assumptions on the model \eqref{U} to derive the main results.
\begin{assumption}\label{Assumption-I-II-III-IV}
\begin{enumerate}
    \item[]
    \item[] 
    \item \textit{Assumption on the dimer}. Suppose that each component $D_m$ of $D$ can be represented by $D_m=a {B}_m+{z}_m$ with the parameter $a>0$ and the location ${z}_m$, for $m=1,2$. 
Denote 
\begin{equation*}
	a:=\max\{\mathrm{diam}(D_1), \mathrm{diam}(D_2)\} \quad \text{and} \quad d:= \mathrm{dist}(D_1, D_2).
\end{equation*}
We take 
	\begin{equation}\label{d-a}
d \, = \, \alpha_0 \, a^{t},
	\end{equation} 
	with $t$ such that $0 \, < \, t \, < \, 1$, and  $\alpha_0$ is a positive constant independent on $a$. 
   \item[]
   \item \textit{Assumptions on the permittivity and permeability of each particle.}
Regarding the permittivity, we assume that
\begin{equation}\label{def-eta}
	\eta(x) \, := \, \begin{cases}
			\eta_{1} \, := \, \epsilon_{r}^{(1)} \, - \, 1 \, = \, \eta_0 \; a^{-2} & \text{if $x \in D_{1}$}\\
              & \\
                \eta_{2} \, := \, \epsilon_{r}^{(2)} \, - \, 1 \, \sim 1 \quad \text{with} \quad \Re \left( \epsilon_{r}^{(2)} \right) \, < \, 0 & \text{if $x \in D_{2}$}
		 \end{cases},
\end{equation}
where $\eta_0$ is a constant in the complex plane independent on the parameter $a$, such that $\Re\left( \eta_0 \right) \, \in \, \mathbb{R}^{+}$. Moreover, regarding the permeability $\mu_r^{(m)}$, we assume that  $\mu_r^{(m)} \, = \, 1,$ for $ m = 1, 2$. 
   \item[] 
   \item \textit{Assumption on the shape of $B_m$.}
   \begin{enumerate}
       \item[]
       \item Regarding the shape $B_{1}$.
   Since
   \begin{equation}\label{Eq0404}
	\mathbb{H}_{0}\left( \div = 0 \right)(B_{1}) \equiv Curl \left( \mathbb{H}_{0}\left( Curl \right) \cap \mathbb{H}\left( \div = 0 \right) \right)(B_{1}),
 \end{equation}
 where
   \begin{eqnarray*}
        \mathbb{H}_{0}\left( Curl \right) \cap \mathbb{H}\left( \div = 0 \right)(B_{1}) \, &:=& \,  \Bigg\{ E \in \mathbb{L}^{2}(B_{1}), \;\text{such that} \; Curl\left( E \right) \in \mathbb{L}^{2}(B_{1}), \\&& \qquad \div\left( E \right) \, = \, 0, \text{in} \, B_{1}, \; \text{and} \; \nu \times E \, = \, 0 \; \text{on} \, \partial B_{1} \Bigg\}, 
   \end{eqnarray*}
	see for instance \cite[Theorem 3.17]{amrouche1998vector}, \mk{then there exists $\phi_{n,m,B_{1}} \in  \mathbb{H}_{0}\left( Curl \right) \cap \mathbb{H}\left( \div = 0 \right)(B_{1})$, such that }  
\begin{equation}\label{pre-cond}
e_{n,m, B_{1}}^{(1)} \, = \, \nabla \times \phi_{n,m, B_{1}}  \,\, \mbox{ with } \,\, \nu\times \phi_{n,m, B_{1}} \, = \, 0 \,\, \text{and} \,\, \nabla \cdot \phi_{n,m, B_{1}} \, = \, 0,
\end{equation}	
where $e_{n,m, B_{1}}^{(1)} \in \mathbb{H}_{0}\left( \div = 0 \right)(B_{1})$ is an eigenfunction, corresponding to the eigenvalue $\lambda_{n}^{(1)}(B_{1})$, related to the Newtonian operator $N_{B_{1}}(\cdot)$ defined, from $\mathbb{L}^{2}(B_{1})$ to $\mathbb{L}^{2}(B_{1})$, by\footnote{The Newtonian operator $N_{B_{1}}(\cdot)$ is bounded from $\mathbb{L}^{2}(B_{1})$ to $\mathbb{H}^{2}(B_{1})$.} 
\begin{equation*}
    N_{B_{1}}\left( E \right)(x) \, := \, \int_{B_{1}} \frac{1}{4 \, \pi} \, \frac{1}{\left\vert x - y \right\vert} \, E(y) \, dy, 
\end{equation*}
i.e., 
\begin{equation*}
    N_{B_{1}}\left( e_{n,m, B_{1}}^{(1)} \right) \; = \; \lambda_{n}^{(1)}(B_{1}) \; e_{n,m, B_{1}}^{(1)}, \quad \text{in} \; B_{1}.
\end{equation*}
We assume\footnote{To reduce the length of the notation in the sequel, we eliminate the need to depend on the vector $\phi_{n,m, B_{1}}$ with respect to multiplicity index $m$.} that for $B_1$, 
\begin{equation}\label{def-phi-n0}
\int_{B_1} \phi_{n_{0}, B_{1}}(y) \, dy  \, \otimes \, \int_{B_1} \phi_{n_{0}, B_{1}}(y) \, dy \, \neq \, 0 \quad \text{for certain} \; n_{0} \in \mathbb{N},
\end{equation}
and 
\begin{equation*}
    \sum_{n} \, \frac{1}{\lambda_{n}^{(3)}(B_{1})} \, \int_{B_{1}} e_{n,B_{1}}^{(3)}(x) \, dx \, \otimes \, \int_{B_{1}} e_{n,B_{1}}^{(3)}(x) \, dx \, \neq \, 0, 
\end{equation*}
where $e_{n,B_{1}}^{(3)} \in \nabla \mathcal{H}armonic(B_{1})$ is an eigenfunction, corresponding to the eigenvalue $\lambda_{n}^{(3)}(B_{1})$, related to the Magnetization operator $\nabla M_{B_{1}}(\cdot)$ defined, from $\mathbb{L}^{2}(B_{1})$ to $\mathbb{L}^{2}(B_{1})$, by 
\begin{equation*}
    \nabla M_{B_{1}}(E)(x) \, := \, \underset{x}{\nabla} \int_{B_{1}} \underset{y}{\nabla}\left( \frac{1}{4 \, \pi \, \left\vert x \, - \, y \right\vert} \right) \cdot E(y) \, dy, 
\end{equation*}
i.e., 
\begin{equation}\label{B1EigFct}
    \nabla M_{B_{1}}\left( e_{n}^{(3)} \right) \, = \, \lambda_{n}^{(3)}(B_{1}) \, e_{n}^{(3)}, \quad \text{in} \,\; B_{1}. 
\end{equation}
\item[]
       \item Regarding the shape $B_{2}$.
       For $B_2$, we assume that 
\begin{equation*}
	\sum_n \int_{B_2}  \phi_{n,B_{2}}(y) \, dy \otimes  \int_{B_2} \phi_{n,B_{2}}(y) \, dy \neq 0,
\end{equation*}
where $\phi_{n,B_{2}}(\cdot)$ satisfy $(\ref{Eq0404})$ in $B_{2}$, and 
\begin{equation*}
     \int_{B_{2}} e_{n_{\star},B_{2}}^{(3)}(x) \, dx \, \otimes \, \int_{B_{2}} e_{n_{\star},B_{2}}^{(3)}(x) \, dx \, \neq \, 0, \quad \text{for certain} \; n_{\star} \in \mathbb{N}, 
\end{equation*}
where $e_{n,B_{2}}^{(3)}(\cdot) \in \nabla \mathcal{H}armonic(B_{2})$ is an eigenfunction, corresponding to the eigenvalue $\lambda_{n}^{(3)}(B_{2})$, related to the Magnetization operator $\nabla M_{B_{2}}\left( \cdot \right)$. 
   \end{enumerate}
   \medskip
   To gain more information about the used spaces $\mathbb{H}_{0}\left( div \, = \, 0\right)(B_{m}), \, \nabla \mathcal{H}armonic(B_{m})$, and the eigensystems that relate to the Newtonian operator $N_{B_{m}}(\cdot)$ and the Magnetization operator $\nabla M_{B_{m}}(\cdot)$, with $m = 1, 2$, it is recommended that the readers refer to Remark \ref{Remark23}.
\item[]
\item \textit{Assumption on the used incident frequency $k$.} Define the vector Magnetization operator $\nabla {M}^0_{B_{2}}(\cdot)$ as 
\eqref{N-M opera}. Under the Helmholtz decomposition of $\mathbb{L}^2$-space given by \eqref{L2-decomposition},
denote $(\lambda_{n}^{(3)}(B_{2}), e_n^{(3)})$ as the corresponding eigen-system of $\nabla M^0_{B_{2}}(\cdot)$ over the subspace $\nabla \mathcal{H}armonic$. There exist complex constants $c_0$ and $ d_0$, with $ Re\left( c_0 \right), Re\left( d_0 \right) \in \, \mathbb{R}^{+}$, such that
\begin{equation}\label{condition-on-k}
			1 \, - \, k^2 \, \eta_1 \, a^2 \, \lambda_{n_{0}}^{(1)}(B_{1}) \, = \, \pm \; c_0\; a^h \quad \text{and} \quad 1 \, + \, \eta_2 \, \lambda_{n_*}^{(3)}(B_{2}) \, = \; \pm \; d_0 \; a^h,\; ~~ a \ll 1,
		\end{equation}
		where $\lambda_{n_0}^{(1)}(B_{1})$ is the eigenvalue corresponding to $e_{n_0}^{(1)}$ in $B_{1}$, related to the Newtonian operator $N_{B_{1}}(\cdot)$, and $\lambda_{n_*}^{(3)}(B_{2})$ is the eigenvalue to $e_{n_*}^{(3)}$ in $B_{2}$, related to the Magnetization operator $\nabla M_{B_{2}}(\cdot)$.
  \item[] 
\end{enumerate}

\noindent The conditions $(\ref{def-eta})$ and $(\ref{condition-on-k})$ can be derived from the Lorentz model by choosing appropriate incident frequency $k$. Indeed, recall the Lorentz model for the relative permittivity that
\begin{equation}\label{Lorentz model}
\epsilon_r=1+\dfrac{k_\mathrm{p}^2}{k_0^2-k^2-ik\xi},
\end{equation}
where $k_\mathrm{p}$ is the plasmonic \mk{frequency}, $k_0$ is the undamped \mk{frequency} resonance of the background and $\xi$ is the damping frequency with $\xi \, \ll \, 1$. The details are given in Section \ref{Appendix}.
\end{assumption}

\noindent Based on the above conditions, we are now in a position to state our main result.

\begin{theorem}\label{main-1}
Let \textbf{Assumption \ref{Assumption-I-II-III-IV}}, on the problem \eqref{U}-\eqref{EincHinc}-\eqref{Radiation-conditions}, which is generated  by the dimer $D$, be satisfied. Let $x$ be away from $D$, then for $t, h \in (0, 1)$ such that  
\begin{equation*}\label{conditions-t-h}
 4 \, - \, h \, - \, 4t \, > \, 0, 
\end{equation*} 
the scattered wave admits the following expression
\begin{equation}\label{general-scattered-field}
		E^{s}(x) \, =   \, k^2  \, \sum_{m=1}^{2}  \, \left[  \Upsilon_{k}(x,z_{m}) \cdot  \tilde{R}_{m} \, - \,  \underset{y}{\nabla} \Phi_{k}(x,z_{m}) \times \tilde{Q}_{m}  \right] \, + \, \mathcal{O}\left( a^{\min(3; 7-2h-3t; 10-2h-7t)} \right),
\end{equation}
and its far-field admits the following expansion 
\begin{equation}\label{approximation-E}
E^\infty(\hat{x}) \, = \, \frac{k^{2}}{4 \, \pi} \, \sum_{m=1}^{2} e^{-i k \hat{x}\cdot z_m} \, \left[ \left( I \, - \, \hat{x} \otimes \hat{x} \right) \cdot   \tilde{R}_{m}  \, + \,  i \, k \, \hat{x} \times  \, \tilde{Q}_{m}  \right] \, + \, \mathcal{O}\left( a^{\min(3; 7-2h-3t; 10-2h-7t)} \right).
\end{equation}
 Here, 
		$\left( \tilde{Q}_{1}, \tilde{R}_{1}, \tilde{Q}_{2}, \tilde{R}_{2} \right)$ is the vector solution to the following algebraic system
        \begin{equation}\label{eq-al-D1} 
    \left[ \begin{pmatrix}
    I_{3} & 0 & 0 & 0 \\
    0 & I_{3} & 0  & 0 \\
    0 & 0 & I_{3} & 0 \\  
    0 & 0 & 0 & I_{3}
    \end{pmatrix} \, - \, 
    \begin{pmatrix}
    0 & 0 & \mathcal{B}_{13} & \mathcal{B}_{14} \\
    0 & 0 & \mathcal{B}_{23}  & \mathcal{B}_{24} \\
    \mathcal{B}_{31} & \mathcal{B}_{32} & 0 & 0 \\  
    \mathcal{B}_{41} & \mathcal{B}_{42} & 0 & 0
    \end{pmatrix} \right]
    \cdot 
    \begin{pmatrix}
     \tilde{Q}_{1} \\
     \tilde{R}_{1} \\
     \tilde{Q}_{2} \\
     \tilde{R}_{2}
    \end{pmatrix}
    =
    \begin{pmatrix}
     \frac{i \, k \, \eta_{0}}{\pm \, c_{0}} \, a^{3-h} \, {\bf P}_{0, 1}^{(1)} \cdot H^{Inc}(z_{1}) \\
     a^{3} \, {\bf P}_{0, 1}^{(2)} \cdot E^{Inc}(z_{1}) \\
     i \, k \, \eta_{2} \, a^{5} \, {\bf P}_{0, 2}^{(1)} \cdot H^{Inc}(z_{2}) \\
     \frac{\eta_{2}}{\pm \, d_{0}} \, a^{3-h} \, {\bf P}_{0, 2}^{(2)} \cdot E^{Inc}(z_{2})
    \end{pmatrix},
\end{equation}
with 
\begin{minipage}[t]{0.45\textwidth}
\begin{eqnarray*}
    \mathcal{B}_{13} \, &:=&  \, \frac{k^{4} \, \eta_{0}}{\pm \, c_{0}} \, a^{3-h} \, {\bf P}_{0, 1}^{(1)} \cdot  \Upsilon_{k}(z_1, z_2) \\
     \mathcal{B}_{23} \, &:=&  \, k^{2} \, a^{3} \, {\bf P}_{0, 1}^{(2)} \cdot \nabla \Phi_{k}(z_1, z_2) \times \\
    \mathcal{B}_{31} \, &:=&  \, k^{4} \, \eta_{2} \, a^{5} \, {\bf P}_{0, 2}^{(1)} \cdot  \Upsilon_{k}(z_2, z_1)  \\
    \mathcal{B}_{41} \, &:=& \,  \, \frac{k^{2} \, \eta_{2}}{\pm \, d_{0}} \, a^{3-h} \, {\bf P}_{0, 2}^{(2)} \cdot \nabla \Phi_{k}(z_2, z_1) \times \\
\end{eqnarray*}
\end{minipage}
\begin{minipage}[t]{0.45\textwidth}
\begin{eqnarray*}
    \mathcal{B}_{14} \, &:=&  \, \frac{k^{2} \, \eta_{0}}{\pm \, c_{0}} \, a^{3-h} \, {\bf P}_{0, 1}^{(1)} \cdot \nabla \Phi_{k}(z_1, z_2) \times \\
    \mathcal{B}_{24} \, &:=&  \, k^{2} \, a^{3} \, {\bf P}_{0, 1}^{(2)} \cdot \Upsilon_{k}(z_1, z_2)  \\
    \mathcal{B}_{32} \, &:=& \,  \, k^{2} \, \eta_{2} \, a^{5} \, {\bf P}_{0, 2}^{(1)} \cdot  \nabla \Phi_{k}(z_2, z_1)\times \\
    \mathcal{B}_{42} \, &:=&  \, \frac{k^{2} \, \eta_{2}}{\pm \, d_{0}} \, a^{3-h} \, {\bf P}_{0, 2}^{(2)} \cdot \Upsilon_{k}(z_2, z_1) 
\end{eqnarray*}
\end{minipage}
\newline
%
where $\Upsilon_k(\cdot, \cdot)$ is the Dyadic Green's kernel given by
  \begin{equation}\label{dyadicG}
  \Upsilon_{k}(\cdot, \cdot) \, := \, \frac{1}{k^{2}} \, \nabla \nabla \Phi_{k}(\cdot, \cdot) \, + \, \Phi_{k}(\cdot, \cdot) \, I_{3}, 
  \end{equation}
  with $\Phi_{k}(\cdot, \cdot)$ being the fundamental solution of the Helmholtz equation given by 
  \begin{equation}\label{Helmholtz-kernel}
      \Phi_{k}(x, y) \, := \, \frac{e^{i \, k \, \left\vert x \, - \, y \right\vert}}{4 \, \pi \, \left\vert x \, - \, y \right\vert}, \quad x \neq y,
  \end{equation}
  $I_3$ is the identity matrix, and ${\bf{P}}_{0, i}^{(j)}$, for $i,j = 1,2$, are the polarization tensors defined by
		\begin{eqnarray*}\label{def-all-tensor}
		{\bf P}_{0, 1}^{(1)} \, &=& \, \int_{B_1}\phi_{n_{0},B_{1}}(x)\, dx\otimes  \int_{B_1}\phi_{n_{0},B_{1}}(x)\, dx,\quad 
		{\bf P}_{0, 1}^{(2)} \, = \, \sum_n \frac{1}{ \lambda_{n}^{(3)}(B_{1})} \, \int_{B_1} e_{n,B_{1}}^{(3)}(x)\,dx\otimes  \int_{B_1} e_{n,B_{1}}^{(3)}(x)\,dx,\notag\\
		{\bf P}_{0, 2}^{(1)} \, &=& \, \sum_n \int_{B_2}\phi_{n,B_{2}}(x)\,dx\otimes\int_{B_2} \phi_{n,B_{2}}(x)\, dx,\, \quad 
		{\bf P}_{0, 2}^{(2)} \, = \,  \int_{B_2} e_{n_{*},B_{2}}^{(3)}(x)\,dx\otimes \int_{B_2} e_{n_{*},B_{2}}^{(3)}(x)\,dx,
		\end{eqnarray*}
with $\phi_{n}(\cdot)$ satisfying $(\ref{pre-cond})$ and $e_n^{(3)}$ fulfilling $\nabla M_{B_{2}}\left(e^{(3)}_n\right) \, = \,  \lambda_{n}^{(3)}(B_{2}) \, e_n^{(3)}$.
\end{theorem}
\medskip
\noindent Moreover, if we further extract the very  dominant term $\left(Q_1,  R_2 \right)$, in $D$, from the solution to the algebraic system $(\ref{eq-al-D1})$, it leads to the following corollary. 
\begin{corollary}\label{corollary-la}
	Let $x$ be away from $D$, then for $t,h \in (0,1)$ fulfilling the condition 
	\begin{equation*}
		 4 \, - \, h \, - \, 4t \, > \, 0, 
	\end{equation*} 
    there holds the following expansion for the scattered field 
\begin{eqnarray}\label{CoroSF}
\nonumber
		E^{s}(x) \, &=& \,  \pm  k^2  a^{3-h}  \left[ \frac{\eta_{2}}{d_{0}}
        \Upsilon_{k}(x,z_{0}) \cdot {\bf P}_{0, 2}^{(2)} \cdot E^{Inc}(z_{0}) \, - \, \frac{i \, k}{c_{0}} \, \eta_{0}
        \underset{y}{\nabla} \Phi_{k}(x,z_{0}) \times {\bf P}_{0, 1}^{(1)} \cdot H^{Inc}(z_{0}) \right] \\ &+&  \mathcal{O}\left( a^{\min(3-h+t;\; 3;\; 10-2h-7t;\; 9-3h-5t)} \right),
\end{eqnarray} 
and the following expansion for the far-field,
\begin{eqnarray}\label{CoroFF}
\nonumber
E^\infty(\hat{x}) \, &=& \, \frac{k^{2}}{\pm 4  \pi}  e^{-i k \hat{x}\cdot  z_0} \, a^{3-h} \left[  \frac{\eta_2}{d_0} \left( I  -  \hat{x}\otimes\hat{x} \right) \cdot  {\bf P}_{0, 2}^{(2)} \cdot E^{Inc}(z_0)  -  \frac{\eta_0 k^{2}}{c_0}  \hat{x}\times {\bf P}_{0, 1}^{(1)} \cdot H^{Inc}(z_0) \right] \\ &+&  \mathcal{O}\left( a^{\min(3-h+t;\; 3;\; 10-2h-7t;\; 9-3h-5t)} \right), 
\end{eqnarray}
where $z_0$ denotes the intermediate point between $z_1$ and $z_2$, and ${\bf P}_{0, 1}^{(1)}$ and ${\bf P}_{0, 2}^{(2)}$ are given in Theorem \ref{main-1}. 
\end{corollary}
\medskip
\noindent We conclude this subsection by mentioning the following compact formulas for rewriting the tensors ${\bf P}_{0, i}^{(j)}$, with $1 \leq i , j \leq 2$, and their values for the particular case where the nano-particles are balls. 

\begin{enumerate}
    \item[]
    \item The tensor 
    \begin{equation*}
    {\bf P}_{0, 1}^{(1)} \, := \, \int_{B_{1}} \phi_{n_{0},B_{1}}(x) \, dx \, \otimes \, \int_{B_{1}} \phi_{n_{0},B_{1}}(x) \, dx.
    \end{equation*}
    reduces, under the particular case of $B_{1}$ being the unit ball, i.e., $B_{1} \equiv B(0,1)$, and $n_{0} \, = \, 1$, to
    \begin{equation}\label{SL1}
    {\bf P}_{0, 1}^{(1)} \, = \, \frac{12}{\pi^{3}} \, I_{3}. 
    \end{equation}
    \item[]
    \item We have the compact form of the tensor
     \begin{equation*}\label{SL2}
    {\bf P}_{0, 2}^{(1)} \, := \, \sum_{n} \int_{B_{2}} \phi_{n,B_{2}}(x) \, dx \, \otimes \, \int_{B_{2}} \phi_{n,B_{2}}(x) \, dx \, =  \,  \int_{B_{2}} Q(x) \cdot \overset{1}{\mathbb{P}}\left(Q\right)(x) \, dx,
    \end{equation*}
    where $Q$ is the matrix given by $(\ref{DefQ(x)})$.
    \item[]
    \item We rewrite the tensor ${\bf P}_{0, 1}^{(2)}$ as
    \begin{equation*}
    {\bf P}_{0, 1}^{(2)} \, := \, \sum_{n} \frac{1}{\lambda_{n}^{(3)}(B_{1})} \, \int_{B_{1}} e_{n,B_{1}}^{(3)}(x) \, dx \, \otimes \, \int_{B_{1}} e_{n,B_{1}}^{(3)}(x) \, dx \, =  \,  \int_{B_1} \, \nabla M_{B_{1}}^{-1}\left( I_{3} \, \chi_{B_{1}} \right)(x)\,dx.
    \end{equation*}
    In addition, under the particular case of $B_{1}$ being the unit ball, i.e., $B_{1} \equiv B(0,1)$, we obtain  
    \begin{equation*}\label{SL3}
    {\bf P}_{0, 1}^{(2)} \, = \, 4 \, \pi \, I_{3}.
    \end{equation*}
    \item[]
    \item Finally the tensor
    \begin{equation*}
 {\bf P}_{0, 2}^{(2)} \, := \, \int_{B_{2}} e^{(3)}_{n_{\star},B_{2}}(x) \, dx \, \otimes \, \int_{B_{2}} e^{(3)}_{n_{\star},B_{2}}(x) \, dx.
\end{equation*} 
reduces, under the particular case of $B_{2}$ being the unit ball, i.e., $B_{2} \equiv B(0,1)$, and $n_{\star} \, = \, 1$, to
\begin{equation}\label{SL4}
 {\bf P}_{0, 2}^{(2)} \, = \, \frac{4 \, \pi}{27} \, I_{3}.
\end{equation} 
\end{enumerate}
The computation details can be found in Subsection \ref{AppRemark}. We observe that for the case of balls, the main tensors ${\bf P}_{0, 1}^{(1)}$ and ${\bf P}_{0, 2}^{(2)}$ are not vanishing and are proportional to the identity matrix.

\subsection{Discussion about the results}~
\bigskip

\noindent The estimation of the electromagnetic fields generated by a single type of nano-particles is already known in the literature, see \cite{Ammari-Li-Zou-2, CGS} for dielectric nano-particles and \cite{GS, CGS-Optic} for plasmonic nano-particles. The related results correspond to those derived here by keeping only the block matrix given by $(\mathcal{B}_{13}, \mathcal{B}_{23}, \mathcal{B}_{31}, \mathcal{B}_{41})$  or $(\mathcal{B}_{14}, \mathcal{B}_{24}, \mathcal{B}_{32}, \mathcal{B}_{42})$, respectively in \ref{eq-al-D1}. The originality here is to have derived the fields generated by such hybrid dimers having not only different shape but also different contrasting materials. We call such hybrid dimer heterogeneous dimers while those dimers with similar, or same, scales are called homogeneous dimers.
\bigskip

\noindent The approximations of the electric fields in \ref{general-scattered-field} and \ref{approximation-E} are modeled by the vectors $Q_1, Q_2, R_1$ and $R_2$ which are solutions of the algebraic system \ref{eq-al-D1}. Precisely, $Q_j$ and $R_j$ model the magnetic and electric poles of the nano-particles $D_j$, $j=1, 2$, respectively. Inverting this algebraic system, using the Born series expansions, provides with a cascade of fields approximations where the most dominant field is described in Corollary \ref{CoroFF}. In this corollary, we see that the generated electromagnetic field by the dimer is a combination of the electric pole generated by the plasmonic nano-particle and magnetic pole generated by the dielectric nano-particle. This shows how the dimer plays a role of dipole to generate the electromagnetic field.  The higher order terms in the Born series describe the two types of contributions that worth mentioning.
\begin{enumerate}
\item[] 
\item The first class of contributions consist in the higher order terms modeling multipoles for each nano-particle, taken in isolation. Such higher order terms are also seen when deriving the expansion for single nano-particle.  
\item[]
\item The second class of contributions consist in the mutual interaction between the two nano-particles. These terms model the multiple scattering between the two nano-particles. 
\item[]
\end{enumerate}

\noindent Based on this classification of the contributions, the mutual interaction between the two nano-particles, forming the heterogeneous dimer, is richer, as compared to single nano-particles or homogeneous dimers, as these contributions enter into game as combinations of higher order modes, for each nano-particle, with mutual interactions between the two nano-particles.
\bigskip

\noindent Such a mutual interaction between the two nano-particles is possible only because they are tuned to resonate at a common incident frequencies. Otherwise, we can also excite the dimer with frequencies away from the common resonances, but eventually near to resonances of one of the nano-particles, then the dimer will predominately behave as a single nano-particle. 
\bigskip

\noindent As a plasmonic-dielectric dimer has the potential of generating both electric and magnetic polarizations, then we expect to use a cluster of such dimers to generate both effective electric permittivity and magnetic permeability. As we excite such systems with nearly resonating incident frequencies, we expect to be able to generate both single negative (permittivity or permeability) or eventually double negative (permittivity and permeability). Such investigation will be reported in a forthcoming work. Let us mention that the use of cluster of nano-particles (single nano-particle) to generate single negative permeability of single negative permittivity is already confirmed in \cite{CGS-Negative-permeability} and \cite{CMS} respectively.
\bigskip

\noindent The rest of the paper is organized as follows. In Section \ref{sec-pre}, we introduce some preliminaries including the $\mathbb{L}^2$-Helmholtz decomposition  and the Lippmann Schwinger system of equations for the solution to \eqref{U}. Based on the Lippmann Schwinger system, we present the a-priori estimates first for single nano-particle and then for the dimer of nano-particles. The estimations for the related scattering coefficients, i.e. corresponding to the induced polarization tensors, are analyzed as well. In Section \ref{sec-la-system}, the precise form of the linear algebraic system is investigated. Section \ref{sec-proof-main} is devoted to prove Theorem \ref{main-1}  and Corollary \ref{corollary-la}, on the basis of the outcomes in Section \ref{sec-pre} and Section \ref{sec-la-system}. In Section \ref{sec-proof-prior}, we provide the proof of all the a-priori estimates presented in Section \ref{sec-pre}. In Section \ref{sec-proof-la}, detailed analysis are given  for the construction of the linear algebraic system presented in Section \ref{sec-la-system}. Section \ref{Appendix}, which will be given as an appendix, will be devoted to the justification of some mentioned results in Section \ref{SecI}.

\section{Some preliminaries and a-priori estimates}\label{sec-pre}

In this section, we present some necessary preliminaries and significant a-prior estimates. For the preliminaries, we cite some key points here for the completeness of the paper, please refer to \cite{CGS} for more details.

\subsection{$\mathbb{L}^2(B)$-Helmholtz decomposition.}
The following direct sum provides a useful decomposition of $\mathbb{L}^{2}(B)$-space, see \cite[Chapter IX, Table I, Page 314]{Dautry-Lions}, 
\begin{equation}\label{L2-decomposition}
\mathbb{L}^{2}(B) \, = \, \mathbb{H}_{0}\left(\div=0 \right)(B) \overset{\perp}{\oplus} \mathbb{H}_{0}\left(Curl=0 \right)(B) \overset{\perp}{\oplus} \nabla \mathcal{H}armonic(B),
\end{equation}   
where 
\begin{eqnarray*}
	\mathbb{H}_{0}\left(\div=0 \right)(B) &:=& \left\lbrace E \in \mathbb{L}^{2}(B), \, \div E = 0, \; \text{in} \, B, \;\; \nu \cdot E = 0 \, \; \text{on} \;\, \partial B \right\rbrace ,\\
	\mathbb{H}_{0}\left(Curl =0 \right)(B) &:=& \left\lbrace E \in \mathbb{L}^{2}(B), \, Curl \, E = 0, \; \text{in} \, B, \, \nu \times E = 0 \, \; \text{on} \;\, \partial B \right\rbrace,
\end{eqnarray*}
and 
\begin{equation*}
\nabla \mathcal{H}armonic(B) \, := \, \left\lbrace E: \; E = \nabla \psi, \, \psi \in \mathbb{H}^{1}(B), \, \Delta\psi \, = \, 0, \; \text{in} \; B \right\rbrace.
\end{equation*}
From the decomposition \eqref{L2-decomposition}, 
we define $\overset{1}{\mathbb{P}}, \overset{2}{\mathbb{P}}$ and $\overset{3}{\mathbb{P}}$ to be the natural projectors as follows
\begin{equation*}
\resizebox{.975\hsize}{!}{$
	\overset{1}{\mathbb{P}} := \mathbb{L}^{2}(B) \longrightarrow  \mathbb{H}_{0}\left(\div=0 \right)(B), \qquad
	\overset{2}{\mathbb{P}} := \mathbb{L}^{2}(B) \longrightarrow  \mathbb{H}_{0}\left(Curl = 0 \right)(B) \quad \text{and} \quad 	\overset{3}{\mathbb{P}} := \mathbb{L}^{2}(B) \longrightarrow  \nabla \mathcal{H}armonic(B).$}
\end{equation*} 

\subsection{Lippmann-Schwinger integral formulation of the solution.}

For any vector field $F$, we define the Newtonian operator $N^{k}_{D}(\cdot)$ and the Magnetization operator $\nabla M^{k}_{D}(\cdot)$ as follows
\begin{equation}\label{N-M opera}
N^{k}_{D}(F)(x):=\int_{D} \Phi_{k}(x,y)F(y)\,dy \quad \text{and} \quad \nabla M^{k}_{D}(F)(x):=\underset{x}{\nabla}\int_{D}\underset{y}{\nabla}\Phi_{k}(x,y)\cdot F(y)\,dy.
\end{equation}
where $\Phi_{k}(\cdot,\cdot)$ is the fundamental solution of the Helmholtz equation given by $(\ref{Helmholtz-kernel})$. 
The solution to \eqref{U} of the integro-differential form can be formulated as the following proposition.

\begin{proposition}
	The solution to the problem $(\ref{U})$  satisfies
	\begin{equation}\label{LS eq2}
		E(x) \, + \, \nabla M^{k}_{D}(\eta \, E)(x) - k^2  \, N^{k}_{D}(\eta \, E)(x) \, = \, E^{Inc}(x),\quad x \in \mathbb{R}^{3},
	\end{equation}
	where $\eta(\cdot)$ is defined by \eqref{def-eta}.
\end{proposition}
\begin{proof}
The proposition can be proved by utilizing the Stratton-Chu formula directly, see \cite[Theorem 6.1]{colton2019inverse} for more detailed discussions.     
\end{proof}
Motivated by the study of the L.S.E given by $(\ref{LS eq2})$, on the sub-spaces involved in the $\mathbb{L}^{2}$-space decomposition, see for instance $(\ref{L2-decomposition})$, using spectral theory techniques, it is crucial for us to discuss the Magnetization operator with vanishing frequency, i.e. $\nabla M(\cdot) \, := \, \nabla M^{0}(\cdot)$, and the Newtonian operator with vanishing frequency, i.e. $N(\cdot) \, := \, N^{0}(\cdot)$, on the L.H.S of $(\ref{LS eq2})$. This idea is clarified by the following remark.

\begin{remark}
The case for a domain $D$ being small leads to the expansions of the Magnetization operator $\nabla M^{k}_{D}(\cdot)$ and the Newtonian operator $N^{k}_{D}(\cdot)$, defined in $(\ref{N-M opera})$ as
	\begin{eqnarray}\label{expansion-gradMk}
	\nonumber
	\nabla M^{k}_{D}(F)(x)  &=& \nabla M_{D}(F)(x) +  \frac{k^{2}}{2} \, N_{D}(F)(x) + \frac{i k^{3}}{12 \pi} \int_{D} F(y) dy - \frac{k^{2}}{2} \int_{D} \Phi_{0}(x,y) \frac{A(x,y)\cdot F(y)}{\vert x -y \vert^{2}}  dy \\ 
	&-& \frac{1}{4 \pi} \; \sum_{n \geq 3} \frac{(ik)^{n+1}}{(n+1)!} \; \int_{D} \;   \underset{y}{\nabla}\underset{y}{\nabla}\left\vert x - y \right\vert^{n}  \cdot F(y) \, dy, \quad x \in D, 
	\end{eqnarray}
	and 
	\begin{equation}\label{expansion-Nk}
	N^{k}_{D}(F)(x) =  N_{D}(F)(x) + \frac{ik}{4\pi} \int_{D} F(y) dy + \frac{1}{4 \pi} \; \sum_{n \geq 1} \frac{(ik)^{n+1}}{(n+1)!} \; \int_{D} \; \left\vert x - y \right\vert^{n}  F(y) \, dy, \quad x \in D,
	\end{equation}
 where $A(\cdot, \cdot)$ is the matrix given by $A(x,y) := \left( x - y \right) \otimes \left( x - y \right)$. For more details on the derivation of $(\ref{expansion-gradMk})$ and $(\ref{expansion-Nk})$, we refer the readers to \cite[Section 2.2]{CGS}.
\end{remark}
 In addition to the above remark, the following behaviors of the Magnetization operator and the Newtonian operator on the sub-spaces involved in the $\mathbb{L}^{2}$-space decomposition, given by $(\ref{L2-decomposition})$, hold. 
 \begin{remark}\label{Remark23}
Two points are in order. 
\begin{enumerate}
    \item[]
    \item The Newtonian operator $N_{B}(\cdot)$ projected onto  the subspace $\mathbb{H}_{0}(\div = 0)(B)$ \big(respectively, $\mathbb{H}_{0}(Curl = 0)(B)$\big) admits an eigensystem that we denote by $\left( \lambda^{(1)}_{n}(B);e^{(1)}_{n, B} \right)$ \Big( respectively, $\left( \lambda^{(2)}_{n}(B);e^{(2)}_{n, B} \right)$ \Big). Besides, we have 
    \begin{equation}\label{RefNeeded1}
     \int_{B} e^{(j)}_{n,B}(y) \, dy = 0, \; \forall \; n \, \in 
      \mathbb{N} \; \text{and} \; j = 1, 2.
   \end{equation}
    \item[]
    \item For the Magnetization operator, the following relations hold. 
    \begin{enumerate}
        \item[$\ast$] $\nabla M^{k}_{B}(\cdot)$ projected onto  the subspace $\mathbb{H}_{0}(\div = 0)(B)$ is a vanishing operator, i.e.
        \begin{equation}\label{grad-M-1st}
 	\forall \; E \in \mathbb{H}_{0}\left(\div=0 \right)(B) \; \text{we have} \; \nabla M_{B}^{k}(E) = 0.
 	\end{equation}
        \item[$\ast$] $\nabla M^{k}_{B}(\cdot)$ projected onto  the subspace $\mathbb{H}_{0}(Curl = 0)(B)$ satisfy 
        \begin{equation}\label{grad-M-2nd}
 	\forall \; E \in \mathbb{H}_{0}\left(Curl = 0 \right)(B) \; \text{we have} \; \nabla M_{B}^{k}(E) \, = \, k^{2} \, N^{k}_{B}(E) \, + \, E \, \chi_{B},
      \end{equation}
      where $\chi_{B}(\cdot)$ is the characteristic function set. 
       \item[$\ast$]  $\nabla M_{B}(\cdot)$ projected onto the subspace $\nabla \mathcal{H}armonic$ admits an eigensystem that we denote by $\left( \lambda^{(3)}_{n}(B);e^{(3)}_{n,B} \right)$. 
    \end{enumerate}
\end{enumerate}
For the existence and the construction of $\left( \lambda^{(j)}_{n}(B);e^{(j)}_{n,B} \right)_{n \in \mathbb{N}},  j=1,2,3$, we refer to \cite[Section 5]{GS}. More properties for the Magnetization operator, such as the self-adjointness, positivity, spectrum, boundedness, etc., can be found in \cite{AhnDyaRae99, Dyakin-Rayevskii, friedman1980mathematical, friedman1981mathematical, 10.2307/2008286} and \cite{Raevskii1994}. Besides, $(\ref{RefNeeded1})$ can be proved by using $(\ref{L2-decomposition})$ and knowing that $I_{3} \in \nabla \mathcal{H}armonic$.
 \end{remark}

\subsection{A-priori Estimates.} Based on the decomposition \eqref{L2-decomposition}, we present here some necessary a-prior estimates related to the electric total field $E$, derived from the Lippmann Schwinger equation \eqref{LS eq2}, and some scattering coefficients, which play an important role in the proof of our main results. In order to achieve this, we will require an intermediate result that will clarify the total field estimates that can be derived from a single nano-particle, whether it's dielectric or plasmonic. This is the subject of the following lemma.  
\begin{lemma}[Estimate for just one nano-particle]\label{es-oneP}
	Under \textbf{Assumption \ref{Assumption-I-II-III-IV}}, we consider the problem \eqref{U} with only one distributed nano-particle. Let $k$ fulfill
	\begin{equation*}\label{choice-k-1st-regime}
		k^2 \, := \, \frac{1 \mp c_{0} \, a^h}{\eta_{1} \, a^2 \, \lambda_{n_{0}}^{(1)}(B_{1})} \, \sim 1.
	\end{equation*}
        Then, for $h<2$, there holds
        \begin{enumerate}
            \item[]
            \item the electric field generated by a single dielectric nano-particle
            \begin{equation*}
		\lVert \tilde{E} \rVert_{\mathbb{L}^2(B_1)} \, = \, \mathcal{O}\left(a^{1-h}\right),
	\end{equation*}
            \item[]
            \item the electric field generated by a single Plasmonic nano-particle
            \begin{equation*}
		\lVert\tilde{E} \rVert_{\mathbb{L}^2(B_2)} \, = \, \mathcal{O}\left(a^{-h}\right),
	\end{equation*}
        \end{enumerate}
	 and, regardless on the used nano-particle, we have
	\begin{equation}\label{*add1}
	 \overset{2}{\mathbb{P}}\left( \tilde{E} \right) \, = \, 0.
	\end{equation}
\end{lemma}
\begin{proof}
    See Subsection \ref{subsec-51}. 
\end{proof}
    \medskip
Furthermore, in the sequel, we will frequently use certain notations that require clarification.
\begin{notation} In the presence of a Dimer, we denote by $E_{m}(\cdot)$ the restriction of $E(\cdot)$ onto the nano-particle $D_{m}$, for $m=1,2$, i.e. $E_{m}(\cdot) := E|_{{D_{m}}}(\cdot)$. Besides, we denote by $\tilde{F}(\cdot)$ the vector field that we obtain by scaling $F(\cdot)$, from $D$ to $B$, i.e.,
\begin{equation*}
    \tilde{F}\left( \eta \right) \, = \, F(z \, + \, a \, \eta), \qquad \eta \in B. 
\end{equation*}
\end{notation}
\medskip
Based on the estimates of the total electric field generated by the presence of a single nano-particle, whether it's dielectric or plasmonic, see Lemma \ref{es-oneP}, we develop the following proposition to estimate the total electric field generated by the presence of a dimer.   
\begin{proposition}[Estimation for the dimer]\label{lem-es-multi}
	Under \textbf{Assumption \ref{Assumption-I-II-III-IV}}, we consider the problem \eqref{U} for the dimer $D$. Let $k$ fulfill
    \begin{equation}\label{choice-k-1st-regime-Prop26}
		k^2 :=  \frac{1 \mp c_{0} \, a^h}{\eta_{1} \, a^2 \, \lambda_{n_{0}}^{(1)}(B_{1})} \, \sim 1.
	\end{equation}
 Then for $t, h\in(0,1)$, there holds the estimation 
\begin{equation}\label{max-P1P3}
\begin{cases}
			\left\lVert \overset{1}{\mathbb{P}}\left(\tilde{E}_1 \right)\right\rVert_{\mathbb{L}^2(B_1)} \, = \, \mathcal{O}\left(a^{1-h}\right) \qquad & \text{,} \qquad \left\lVert \overset{1}{\mathbb{P}}\left(\tilde{E}_2 \right)\right\rVert_{\mathbb{L}^2(B_2)} \, = \, \mathcal{O}\left(a^{\min(1; 4-h-3t)} \right)\\
               \\
             \left\lVert\overset{3}{\mathbb{P}}\left(\tilde{E}_1 \right)\right\rVert_{\mathbb{L}^2(B_1)}\, =  \, \mathcal{O}\left( a^2\right)  \qquad  & \text{,} \qquad \left\lVert\overset{3}{\mathbb{P}}\left(\tilde{E}_2 \right)\right\rVert_{\mathbb{L}^2(B_2)} \,  =  \, \mathcal{O}\left( a^{-h} \right)
		 \end{cases}
\end{equation}
In addition, we have 
\begin{equation}\label{Equa1108}
    \overset{2}{\mathbb{P}}\left(\tilde{E}_1 \right) \; = \; \overset{2}{\mathbb{P}}\left(\tilde{E}_2 \right) \; = \; 0.
\end{equation}
   \end{proposition}
\begin{proof}
    See Subsection \ref{subsec-52}. 
\end{proof}
\medskip
In addition to estimating the electric total field generated by a dimer, which can be found in Proposition \ref{lem-es-multi}, it's necessary to estimate the scattering coefficients related to the problem  $(\ref{U})$. The purpose of the following definition is to define the scattering coefficients related to the problem $(\ref{U})$, associated with the used dimer, which will be utilized to justify our derived results.   \begin{definition}\label{def-scattering coeff}
    	We define $W_{1}, W_{2}, V_{1}$ and $V_{2}$ to be solutions of,  
    	\begin{eqnarray}\label{AI1}
    		\left(I + \eta_1 \, \nabla M^{-k}_{D_{1}} \, - \, k^{2} \, \eta_1 \, N^{-k}_{D_{1}} \right)\left( W_1 \right)(x) &=& \mathcal{P}(x,z_1), \quad \; \, \; \quad x \in D_1, \label{DefW1}  \\
    		\left(I + \eta_2 \, \nabla M^{-k}_{D_{2}} \, - \, k^{2} \, \eta_2 \, N^{-k}_{D_{2}} \right)\left( W_2 \right)(x) &=& \overset{1}{\PP}\left(\mathcal{P}(x,z_2)\right), \quad x \in D_2,   \label{DefW2} \\
            \left(I + \eta_m \, \nabla M^{-k}_{D_{m}} \, - \, k^{2} \, \eta_m \, N^{-k}_{D_{m}} \right)\left( V_m \right)(x) &=& I_{3}, \quad \;\;  \qquad \qquad x \in D_m, \quad m = 1, 2, \label{def-Vm}
    	\end{eqnarray}
    	where, for $m=1,2$, the operators $\nabla M^{-k}_{D_{m}}(\cdot)$ and $N^{-k}_{D_{m}}(\cdot)$ are the adjoint operators to $\nabla M^k_{D_{m}}(\cdot)$ and $N^k_{D_{m}}(\cdot)$, introduced in \eqref{N-M opera}, and $\mathcal{P}(x, z)$ is the matrix expressed by
    	\begin{equation*}
    	\mathcal{P}(x, z) \, := \, \left(\begin{array}{c}
    	(x-z)_1 \, I_{3}\\ 
    	(x-z)_2 \, I_{3}\\ 
    	(x-z)_3 \, I_{3}
    	\end{array} \right).
    	\end{equation*} 
    \end{definition}
    \medskip
    Based on the estimates given in Proposition $\ref{lem-es-multi}$, and using the notations introduced by Definition $\ref{def-scattering coeff}$, we provide in the following proposition the estimates related to the problem $(\ref{U})$'s scattering coefficients. 
\begin{proposition}[Estimation of the scattering coefficients]\label{prop-scoeff}
        For $h \in (0,2)$, under assumption \eqref{condition-on-k}, the following estimations hold:
        \begin{enumerate}
            \item Regarding the scattering parameter $W_{1}$, defined by \eqref{AI1}, we have  \begin{equation}\label{*add0}
		\left\lVert \overset{1}{\mathbb{P}}\left(\tilde{W}_1\right)\right\rVert_{\mathbb{L}^2(B_1)}=\mathcal{O}(a^{1-h}) \quad \text{and} \quad
		\left\lVert \overset{j}{\mathbb{P}}\left(\tilde{W}_1\right)\right\rVert_{\mathbb{L}^2(B_1)}=\mathcal{O}(a^{3}), \quad \text{for} \quad j = 2, 3.	
		\end{equation}
            \item[]
            \item Regarding the scattering parameter $W_{2}$, defined by $(\ref{DefW2})$, we have  
        \begin{equation}\label{prop-es-W}
			\left\lVert \overset{1}{\mathbb{P}}\left(\tilde{W}_2\right)\right\rVert_{\mathbb{L}^2(B_2)}=\mathcal{O}(a),\quad 
			 \overset{2}{\mathbb{P}}\left(\tilde{W}_2\right) \, = \, 0 \quad \text{and} \quad
			\left\lVert\overset{3}{\mathbb{P}}\left(\tilde{W}_2\right)\right\rVert_{\mathbb{L}^2(B_2)}=\mathcal{O}(a^{5-h}).
		\end{equation}
        \item[] 
        \item Regarding the scattering parameter $V_{m}$, defined by \eqref{def-Vm}, for $m = 1, 2$, we have
        \begin{equation}\label{es-V12}
			\left\lVert \tilde{V}_1\right\rVert_{\mathbb{L}^2(B_1)}=\mathcal{O}\left(a^2\right)\quad\mbox{and}\quad\left\lVert \tilde{V}_2\right\rVert_{\mathbb{L}^2(B_2)}=\mathcal{O}\left(a^{-h}\right).
		\end{equation}
        \item[] 
        \end{enumerate}
	\end{proposition}
 \begin{proof}
     See Subsection \ref{subsec:scattering parameter}.
 \end{proof}
\section{Linear algebraic system of the dimer}\label{sec-la-system}
In this section, we shall present the linear algebraic system  derived from the L.S.E $(\ref{LS eq2})$ by projecting the solution $E$ onto the two sub-spaces $\mathbb{H}_0(\div=0)$ and $\nabla \mathcal{H}armonic$. We start by using $(\ref{Eq0404})$ to derive, for $m=1,2$, the following expression
\begin{equation}\label{def-F_j}
\overset{1}{\mathbb{P}}(E_m) \, = \, Curl \left( F_m \right), \; \text{in} \; D_{m}, \quad \text{with} \quad \nu \times F_m \,= \, 0, \; \text{on} \; \partial D_{m}, \quad \text{and} \quad \div(F_m) \, = \, 0, \; \text{in} \; D_{m}.
\end{equation}
Set
\begin{equation}\label{Qj1}
Q_m:=\eta_m\int_{D_{m}}F_m(y)\,dy\quad\mbox{and}\quad R_m:=\eta_m\int_{D_{m}} \overset{3}{\mathbb{P}}\left(E_m\right)(y)\,dy,
\end{equation} 
where $\eta_m$ is defined in \eqref{def-eta}. 

\begin{proposition}\label{prop-la-1}
	Under \textbf{Assumption \ref{Assumption-I-II-III-IV}}, for $t, h \in (0,1)$, such that 
    \begin{equation*}
        4 \, - \, h \, - \, 4 \, t \, > \, 0,
    \end{equation*}
    there holds the coming linear algebraic systems related to the problem $(\ref{U})$,  
        \begin{equation}\label{al-D1}
    \begin{pmatrix}
    I_{3} & 0 & - \, \mathcal{B}_{13} & - \, \mathcal{B}_{14} \\
    0 & I_{3} & - \, \mathcal{B}_{23}  & - \, \mathcal{B}_{24} \\
    - \, \mathcal{B}_{31} & - \, \mathcal{B}_{32} & I_{3} & 0 \\  
    - \, \mathcal{B}_{41} & - \, \mathcal{B}_{42} & 0 & I_{3}
    \end{pmatrix} 
    \cdot 
    \begin{pmatrix}
     Q_{1} \\
     R_{1} \\
     Q_{2} \\
     R_{2}
    \end{pmatrix}
    =
    \begin{pmatrix}
     \frac{i \, k \, \eta_{0}}{\pm \, c_{0}} \, a^{3-h} \, {\bf P}_{0, 1}^{(1)} \cdot H^{Inc}(z_{1}) \\
     a^{3} \, {\bf P}_{0, 1}^{(2)} \cdot E^{Inc}(z_{1}) \\
     i \, k \, \eta_{2} \, a^{5} \, {\bf P}_{0, 2}^{(1)} \cdot H^{Inc}(z_{2}) \\
     \frac{\eta_{2}}{\pm \, d_{0}} \, a^{3-h} \, {\bf P}_{0, 2}^{(2)} \cdot E^{Inc}(z_{2})
    \end{pmatrix} 
    + \begin{pmatrix}
       Error_1^{(1)} \\ 
       Error_1^{(2)} \\
       Error_2^{(1)} \\
       Error_2^{(2)}
    \end{pmatrix}
\end{equation}
with 
\begin{minipage}[t]{0.45\textwidth}
\begin{eqnarray*}
    \mathcal{B}_{13} \, &:=&  \, \frac{k^{4} \, \eta_{0}}{\pm \, c_{0}} \, a^{3-h} \, {\bf P}_{0, 1}^{(1)} \cdot  \Upsilon_{k}(z_1, z_2) \\
     \mathcal{B}_{23} \, &:=&  \, k^{2} \, a^{3} \, {\bf P}_{0, 1}^{(2)} \cdot \nabla \Phi_{k}(z_1, z_2) \times \\
    \mathcal{B}_{31} \, &:=&  \, k^{4} \, \eta_{2} \, a^{5} \, {\bf P}_{0, 2}^{(1)} \cdot  \Upsilon_{k}(z_2, z_1)  \\
    \mathcal{B}_{41} \, &:=& \, \frac{k^{2} \, \eta_{2}}{\pm \, d_{0}} \, a^{3-h} \, {\bf P}_{0, 2}^{(2)} \cdot \nabla \Phi_{k}(z_2, z_1) \times \\
\end{eqnarray*}
\end{minipage}
\begin{minipage}[t]{0.45\textwidth}
\begin{eqnarray*}
    \mathcal{B}_{14} \, &:=&  \, \frac{k^{2} \, \eta_{0}}{\pm \, c_{0}} \, a^{3-h} \, {\bf P}_{0, 1}^{(1)} \cdot \nabla \Phi_{k}(z_1, z_2) \times \\
    \mathcal{B}_{24} \, &:=&  \, k^{2} \, a^{3} \, {\bf P}_{0, 1}^{(2)} \cdot \Upsilon_{k}(z_1, z_2)  \\
    \mathcal{B}_{32} \, &:=&  \, k^{2} \, \eta_{2} \, a^{5} \, {\bf P}_{0, 2}^{(1)} \cdot  \nabla \Phi_{k}(z_2, z_1)\times \\
    \mathcal{B}_{42} \, &:=&  \, \frac{k^{2} \, \eta_{2}}{\pm \, d_{0}} \, a^{3-h} \, {\bf P}_{0, 2}^{(2)} \cdot \Upsilon_{k}(z_2, z_1) 
\end{eqnarray*}
\end{minipage}

	where
        \begin{eqnarray*}
           {\bf P}_{0, 1}^{(1)} \, &:=& \, \int_{B_1}\phi_{n_{0}, B_{1}}(x)\, dx\otimes  \int_{B_1}\phi_{n_{0}, B_{1}}(x)\, dx \\
		{\bf P}_{0, 1}^{(2)} \, &:=& \, \sum_n \frac{1}{\lambda_{n}^{(3)}(B_{1})}\int_{B_1} e_{n, B_{1}}^{(3)}(x)\,dx\otimes  \int_{B_1} e_{n, B_{1}}^{(3)}(x)\,dx \\
           {\bf P}_{0, 2}^{(1)} \, &=& \, \sum_n \int_{B_2}\phi_{n, B_{2}}(x)\,dx\otimes\int_{B_2} \phi_{n, B_{2}}(x)\, dx \\ 
           {\bf P}_{0, 2}^{(2)} \, &=& \, \int_{B_2} e_{n_{*}, B_{2}}^{(3)}(x)\,dx\otimes \int_{B_2} e_{n_{*}, B_{2}}^{(3)}(x)\,dx,
        \end{eqnarray*}
	with $\phi_{n_{0}, B_{1}}(\cdot)$ satisfying $(\ref{def-phi-n0})$ and $\phi_{n, B_{2}}(\cdot)$ is given by $(\ref{pre-cond})$. Besides, 
        \begin{eqnarray*}\label{las-error}
        &Error_1^{(1)}  = \, \mathcal{O}\left( a^{\min(3; 7-2h-3t)} \right),
        &Error_1^{(2)}  = \, \mathcal{O}\left( a^{\min(4;7-h-4t)} \right), \notag\\
        &Error_2^{(1)} = \mathcal{O}(a^{\min(6; 9-h-4t)}),  
	&Error_2^{(2)} = \mathcal{O}\left( a^{\min(4-h; 7-h-4t; 6-3t; 7-2h-3t)} \right). 
        \end{eqnarray*}
\end{proposition}
\begin{proof}
See Section \ref{sec-proof-la}.
\end{proof}
\section{Proof of the main result}\label{sec-proof-main}
In this section, the proof of Theorem \ref{main-1} is presented as the following four steps with all the necessary propositions given in the previous sections. 
\medskip
\newline 
\noindent (I). Derivation of the scattered wave $E^s(x)$.

\bigskip

Thanks to the L.S.E given by $(\ref{LS eq2})$ and the fact that $E \, = \, E^{Inc} \, + \, E^{s}$, we deduce that
	\begin{equation*}
		E^{s}(x) \, = \, - \, \nabla M^{k}_{D}(\eta \, E)(x) \, + \, k^2  \, N^{k}_{D}(\eta \, E)(x),
	\end{equation*}
    which, by letting $x$ outside $D$ and using $(\ref{dyadicG})$, can be rewritten as
        \begin{equation*}
		E^{s}(x) \, = \, k^2  \, \int_{D} \Upsilon_{k}(x,y) \cdot \eta(y) \, E(y) \, dy \, = \, k^2  \, \sum_{m=1}^{2} \eta_{m} \, \int_{D_{m}} \Upsilon_{k}(x,y) \cdot  E_{m}(y) \, dy.
	\end{equation*}
Besides, by splitting $E_{m}$ as $E_{m} = \overset{1}{\mathbb{P}}\left(E_{m}\right) \, + \, \overset{3}{\mathbb{P}}\left(E_{m}\right)$, see $(\ref{Equa1108})$, and using the fact that $\nabla M^{k}(\cdot)|_{\mathbb{H}_{0}\left( \div = 0 \right)}$ is a vanishing operator we get 
\begin{eqnarray*}
		E^{s}(x) \, &=&  \, k^2  \, \sum_{m=1}^{2} \eta_{m} \, \int_{D_{m}} \Phi_{k}(x,y) \,  \overset{1}{\mathbb{P}}\left(E_{m}\right)(y) \, dy + k^2  \, \sum_{m=1}^{2} \eta_{m} \, \int_{D_{m}} \Upsilon_{k}(x,y) \cdot   \overset{3}{\mathbb{P}}\left(E_{m}\right)(y) \, dy \\
        &\overset{(\ref{def-F_j})}{=}& \, - \, k^2  \, \sum_{m=1}^{2} \eta_{m} \, \int_{D_{m}} \underset{y}{\nabla} \Phi_{k}(x,y) \times F_{m}(y) \, dy + k^2  \, \sum_{m=1}^{2} \eta_{m} \, \int_{D_{m}} \Upsilon_{k}(x,y) \cdot   \overset{3}{\mathbb{P}}\left(E_{m}\right)(y) \, dy.
	\end{eqnarray*}
On the R.H.S of the above expression, by expanding $\nabla \Phi_{k}(x,\cdot)$ and $\Upsilon_{k}(x,\cdot)$, near the center $z_{m}$, with $m=1,2$, and using $(\ref{Qj1})$ we obtain  
\begin{equation}\label{SMEq0758}
		E^{s}(x) \, =   \, - \, k^2  \, \sum_{m=1}^{2}  \, \left[  \underset{y}{\nabla} \Phi_{k}(x,z_{m}) \times Q_{m} \, -  \,  \Upsilon_{k}(x,z_{m}) \cdot  R_{m} \right] \, + \, Remainder,
\end{equation}
where $Remainder$ is given by 
\begin{eqnarray*}
    Remainder \, &:=& \, \, - \, k^2  \, \sum_{m=1}^{2} \eta_{m} \, \int_{D_{m}} \int_{0}^{1} \underset{y}{\nabla} \underset{y}{\nabla} \Phi_{k}(x,z_{m}+t(y-z_{m})) \cdot (y - z_{m}) \, dt \times F_{m}(y) \, dy \\ &+& k^2  \, \sum_{m=1}^{2} \eta_{m} \, \int_{D_{m}} \int_{0}^{1} \underset{y}{\nabla} \Upsilon_{k}(x,z_{m}+t(y-z_{m})) \cdot \mathcal{P}\left(y, z_{m} \right) dt \cdot   \overset{3}{\mathbb{P}}\left(E_{m}\right)(y) \, dy. 
\end{eqnarray*}
Next, we estimate the term $Remainder$.
\begin{eqnarray*}
    \left\vert Remainder \right\vert \, & \lesssim & \, \, \sum_{m=1}^{2} \left\vert \eta_{m} \right\vert \, \left\Vert \int_{0}^{1} \underset{y}{\nabla} \underset{y}{\nabla} \Phi_{k}(x,z_{m}+t(\cdot-z_{m})) \cdot (\cdot - z_{m}) \, dt \right\Vert_{\mathbb{L}^{2}(D_{m})} \; \left\Vert F_{m} \right\Vert_{\mathbb{L}^{2}(D_{m})} \\ &+&  \, \sum_{m=1}^{2} \left\vert \eta_{m} \right\vert \, \left\Vert \int_{0}^{1} \underset{y}{\nabla} \Upsilon_{k}(x,z_{m}+t(\cdot-z_{m})) \cdot \mathcal{P}\left(\cdot, z_{m} \right) dt \right\Vert_{\mathbb{L}^{2}(D_{m})} \, \left\Vert   \overset{3}{\mathbb{P}}\left(E_{m}\right) \right\Vert_{\mathbb{L}^{2}(D_{m})} \\
    & \lesssim & \, a^{4} \, \sum_{m=1}^{2} \left\vert \eta_{m} \right\vert \, \left[   \left\Vert \tilde{F}_{m} \right\Vert_{\mathbb{L}^{2}(B_{m})} \, + \, \left\Vert   \overset{3}{\mathbb{P}}\left(\tilde{E}_{m}\right) \right\Vert_{\mathbb{L}^{2}(B_{m})} \right]. 
\end{eqnarray*}
Thanks to $(\ref{max-P1P3}), (\ref{def-eta})$, and the formula $(\ref{F2aP1E2})$, we deduce that 
\begin{equation*}
    Remainder \, = \, \mathcal{O}\left( a^{4-h} \right).
\end{equation*}
Hence, $(\ref{SMEq0758})$ becomes, 
\begin{equation}\label{Eq1132}
		E^{s}(x) \, =   \, - \, k^2  \, \sum_{m=1}^{2}  \, \left[  \underset{y}{\nabla} \Phi_{k}(x,z_{m}) \times Q_{m} \, -  \,  \Upsilon_{k}(x,z_{m}) \cdot  R_{m} \right] \, + \, \mathcal{O}\left( a^{4-h} \right),
\end{equation}
where $\left(Q_{1}, R_{1}, Q_{2}, R_{2} \right)$ is solution of $(\ref{al-D1})$.   
\bigskip
\newline 
\noindent (II). Derivation of the far-field $E^\infty(\hat{x})$.

\bigskip

To estimate the far-field, we use the fact that $E \, = \, E^{Inc} \, + \, E^{s}$, the formula $(\ref{N-M opera})$
and formula $(\ref{def-far})$, to obtain
	\begin{eqnarray*}
		E^{\infty}(\hat{x}) \, &=& \, \frac{k^{2}}{4 \, \pi} \, \left( I - \hat{x} \otimes \hat{x} \right) \cdot \int_{D} e^{- \, i \, k \, \hat{x} \cdot y} \; \eta(y) \; E(y) \; dy, \quad D =  D_{1}\cup D_2. \notag\\
		& \overset{(\ref{def-eta})}{=} & \, \frac{k^{2}}{4 \, \pi} \, \eta_1 \, \left( I - \hat{x} \otimes \hat{x} \right) \cdot \int_{D_{1}} e^{- \, i \, k \, \hat{x} \cdot y} \; E_{1}(y) \; dy \; + \; \, \frac{k^{2}}{4 \, \pi} \, \eta_2 \, \left( I - \hat{x} \otimes \hat{x} \right) \cdot \int_{D_{2}} e^{-i \, k \, \hat{x} \cdot y} \; E_{2}(y) \; dy.
	\end{eqnarray*}
By expanding the function $y \rightarrow e^{- \, i \, k \, \hat{x} \cdot y}$ at $y=z_m$, with $m = 1, 2$, we obtain 
	\begin{equation}\label{eq-far2}
		E^{\infty}(\hat{x}) \, = \, \frac{k^{2}}{4 \, \pi} \, \sum_{m=1}^{2} \, \eta_{m} \,  \left( I - \hat{x} \otimes \hat{x} \right) \cdot    \int_{D_{m}} \left[  e^{-i  k  \hat{x} \cdot z_{m}} - i  k  e^{- i  k  \hat{x} \cdot z_{m}}  (y-z_{m}) \cdot \hat{x}  \right]  E_{m}(y) dy \, + \, \sum_{m=1}^2 R^{(1)}_{m},  
	\end{equation}
where
\begin{equation*}
	R^{(1)}_{m} \, := \, \frac{k^{4}}{8 \, \pi} \, \eta_{m} \, (I-\hat{x}\otimes\hat{x}) \cdot \int_{D_{m}} \left((y-z_{m})\cdot\hat{x}\right)^2\int_{0}^1 (1-t)e^{-i k \hat{x}\cdot (z_1+t(y-z_{m}))}\,dt E_{m}(y)\,dy, \quad m = 1, 2.
\end{equation*}
It is direct to get from $(\ref{def-eta})$ that, there holds
\begin{equation}\label{es-R11}
 R_1^{(1)}  \, = \, \mathcal{O}\left( a^{\frac{3}{2}}\left\lVert E_1\right\rVert_{\mathbb{L}^2(D_1)}\right)
\quad 
\text{and} \quad 
 R_2^{(1)}  \, = \, \mathcal{O}\left( a^{\frac{7}{2}} \left\lVert E_2\right\rVert_{\mathbb{L}^2(D_2)} \right).
\end{equation}
Based on Proposition \ref{es-oneP}, formula $(\ref{*add1})$, by splitting $E_{m}$ as $E_{m} = \overset{1}{\mathbb{P}}\left(E_{m}\right) \, + \, \overset{3}{\mathbb{P}}\left(E_{m}\right)$, we use 
$(\ref{RefNeeded1})$, and the estimations derived in $(\ref{es-R11})$, to rewrite $(\ref{eq-far2})$ as,  
	\begin{eqnarray}\label{Before-estimating-the number of particles-N}
 \nonumber
		E^{\infty}(\hat{x}) & = & \frac{k^{2}}{4 \, \pi} \, \sum_{m=1}^{2} \, \eta_{m} \, \left( I - \hat{x} \otimes \hat{x} \right) \cdot e^{- i  k  \hat{x} \cdot z_{m}}  \, \int_{D_{m}} \overset{3}{\mathbb{P}}\left(E_{m}\right)(y)  dy \\ \nonumber &-& i  \, \frac{k^{3}}{4 \, \pi} \, \sum_{m=1}^{2} \eta_{m}  \left( I - \hat{x} \otimes \hat{x} \right) \cdot e^{- i  k  \hat{x} \cdot z_{m}} \, \int_{D_{m}} \hat{x} \cdot (y-z_{m}) \; \overset{1}{\mathbb{P}}\left(E_{m}\right)(y)  dy   \\
		& + & R_{1}^{(2)} \, + \, R_{2}^{(2)} \, + \,   \mathcal{O}\left( a^{\frac{3}{2}} \left\lVert E_1 \right\rVert_{\mathbb{L}^2(D_1)}\right) + \O\left( a^{\frac{7}{2}} \left\lVert E_2 \right\rVert_{\mathbb{L}^2(D_2)}\right),
	\end{eqnarray}
	where 
	\begin{equation*}
	R_{m}^{(2)} \; := \; - i \, \frac{k^3}{4\pi} \, \eta_{m} \left( I \, - \, \hat{x} \otimes \hat{x} \right) e^{- i k \hat{x}\cdot z_{m}} \cdot \int_{D_{m}} \hat{x}\cdot (y-z_m)\overset{3}{\PP}(E_m)(y)\, dy, \quad m = 1, 2.
	\end{equation*}
 Then, using $(\ref{def-eta})$, the following estimation holds, 
 \begin{equation*}
      R_1^{(2)} \, = \, \mathcal{O}\left(  a^{\frac{1}{2}} \, \left\lVert \overset{3}{\PP}(E_1)\right\rVert_{\mathbb{L}^2(D_1)} \right) \quad \text{and} \quad    R_2^{(2)}  \, = \, \mathcal{O}\left(  a^{\frac{5}{2}} \, \left\lVert \overset{3}{\PP}(E_2)\right\rVert_{\mathbb{L}^2(D_2)} \right). 
 \end{equation*}
Thus, we can deduce from \eqref{Before-estimating-the number of particles-N} that,
      	\begin{eqnarray*}\label{...particles-N}
 \nonumber
		E^{\infty}(\hat{x}) & = & \frac{k^{2}}{4 \, \pi} \, \sum_{m=1}^{2} \, \eta_{m} \, \left( I - \hat{x} \otimes \hat{x} \right) \cdot e^{- i  k  \hat{x} \cdot z_{m}}  \, \int_{D_{m}} \overset{3}{\mathbb{P}}\left(E_{m}\right)(y)  dy \\ \nonumber &-& i \, \frac{k^{3}}{4 \, \pi} \,  \, \sum_{m=1}^{2} \eta_{m}  \left( I - \hat{x} \otimes \hat{x} \right) \cdot e^{- i  k  \hat{x} \cdot z_{m}} \, \int_{D_{m}} \hat{x} \cdot (y-z_{m}) \; \overset{1}{\mathbb{P}}\left(E_{m}\right)(y)  dy   \\ \nonumber
		& + & \mathcal{O}\left( a^{\frac{1}{2}} \, \left\lVert \overset{3}{\PP}(E_1)\right\rVert_{\mathbb{L}^2(D_1)} \right) \, + \, \mathcal{O}\left(a^{\frac{5}{2}} \, \left\lVert \overset{3}{\PP}(E_2)\right\rVert_{\mathbb{L}^2(D_2)} \right) \\ &+&\,  \mathcal{O}\left( a^{\frac{3}{2}} \left\lVert E_1 \right\rVert_{\mathbb{L}^2(D_1)}\right) + \, \O\left(  a^{\frac{7}{2}} \left\lVert E_2 \right\rVert_{\mathbb{L}^2(D_2)}\right),
	\end{eqnarray*}
which, by using \eqref{def-F_j} and the notations given by \eqref{Qj1}, indicates 
	\begin{equation}\label{IA1}
		E^{\infty}(\hat{x})  \, = \, \frac{k^{2}}{4 \, \pi} \, \sum_{m=1}^{2}  \, e^{- \, i  k  \hat{x} \cdot z_{m}}  \, \left[ \left( I - \hat{x} \otimes \hat{x} \right) \cdot  R_{m} \, + \, i \, k \,  \,  \hat{x} \times Q_{m} \right] \, + \,  Error^{(1)},
	\end{equation}
	where the vectors $\left(Q_m, R_m \right)_{m=1, 2}$ are the solutions of the linear algebraic system given by \eqref{eq-al-D1}, and 
	\begin{equation*}
    \resizebox{.99\hsize}{!}{$
		Error^{(1)} \, := \,   \mathcal{O}\left( \, a^{\frac{1}{2}} \, \left\lVert \overset{3}{\PP}(E_1)\right\rVert_{\mathbb{L}^2(D_1)} \right) \, + \, \mathcal{O}\left( \, a^{\frac{5}{2}} \, \left\lVert \overset{3}{\PP}(E_2)\right\rVert_{\mathbb{L}^2(D_2)} \right)+\,   \mathcal{O}\left(a^{\frac{3}{2}} \left\lVert E_1 \right\rVert_{\mathbb{L}^2(D_1)}\right) + \, \O\left(  a^{\frac{7}{2}} \left\lVert E_2 \right\rVert_{\mathbb{L}^2(D_2)}\right)$},
	\end{equation*}
 which can be estimated, by using Proposition \ref{es-oneP} and Proposition \ref{lem-es-multi}, as $\mathcal{O}\left( a^{4-h} \right)$. Hence, $(\ref{IA1})$ becomes,  
 	\begin{equation}\label{Eq0254}
		E^{\infty}(\hat{x})  \, = \, \frac{k^{2}}{4 \, \pi} \, \sum_{m=1}^{2}  \, e^{- \, i  k  \hat{x} \cdot z_{m}}  \, \left[ \left( I - \hat{x} \otimes \hat{x} \right) \cdot  R_{m} \, + \, i \, k \,  \,  \hat{x} \times Q_{m} \right] \, + \, \mathcal{O}\left( a^{4-h} \right),
	\end{equation}
 where $\left(Q_{1}, R_{1}, Q_{2}, R_{2} \right)$ is solution of $(\ref{al-D1})$. 
 \bigskip
 \newline 
 We have seen from the estimation of the scattered field $E^{s}(\cdot)$, given by $(\ref{Eq1132})$, and the estimation of the far field $E^{\infty}(\cdot)$, given by $(\ref{Eq0254})$, that the determination of $\left(Q_{1}, R_{1}, Q_{2}, R_{2} \right)$ is of great importance to achieve the estimation of both $E^{s}(\cdot)$ and $E^{\infty}(\cdot)$.
The goal of the next step is to determine the dominant part related to $\left(Q_{1}, R_{1}, Q_{2}, R_{2} \right)$. 
 \bigskip
 \newline 
\noindent (III). Invertibility of the algebraic system \eqref{al-D1}.
\bigskip

We recall that $\left(Q_{1}, R_{1}, Q_{2}, R_{2} \right)$ is the solution to 
\begin{equation}\label{PerAS}
    \begin{pmatrix}
    I_{3} & 0 & - \, \mathcal{B}_{13} & - \, \mathcal{B}_{14} \\
    0 & I_{3} & - \, \mathcal{B}_{23}  & - \, \mathcal{B}_{24} \\
    - \, \mathcal{B}_{31} & - \, \mathcal{B}_{32} & I_{3} & 0 \\  
    - \, \mathcal{B}_{41} & - \, \mathcal{B}_{42} & 0 & I_{3}
    \end{pmatrix} 
    \cdot 
    \begin{pmatrix}
     Q_{1} \\
     R_{1} \\
     Q_{2} \\
     R_{2}
    \end{pmatrix}
    =
    \begin{pmatrix}
     \frac{i \, k \, \eta_{0}}{\pm \, c_{0}} \, a^{3-h} \, {\bf P}_{0, 1}^{(1)} \cdot H^{Inc}(z_{1}) \\
     a^{3} \, {\bf P}_{0, 1}^{(2)} \cdot E^{Inc}(z_{1}) \\
     i \, k \, \eta_{2} \, a^{5} \, {\bf P}_{0, 2}^{(1)} \cdot H^{Inc}(z_{2}) \\
     \frac{\eta_{2}}{\pm \, d_{0}} \, a^{3-h} \, {\bf P}_{0, 2}^{(2)} \cdot E^{Inc}(z_{2})
    \end{pmatrix} \, + \, \begin{pmatrix}
    Error_1^{(1)} \\
     Error_1^{(2)} \\
     Error_2^{(1)} \\
     Error_2^{(2)}
    \end{pmatrix}.
\end{equation}
Notice that the R.H.S of $(\ref{PerAS})$ is perturbed by the presence of an additive vector error term. 
Next, we investigate the associated unperturbed algebraic system given by 
\begin{equation}\label{unPerAS}
    \begin{pmatrix}
    I_{3} & 0 & - \, \mathcal{B}_{13} & - \, \mathcal{B}_{14} \\
    0 & I_{3} & - \, \mathcal{B}_{23}  & - \, \mathcal{B}_{24} \\
    - \, \mathcal{B}_{31} & - \, \mathcal{B}_{32} & I_{3} & 0 \\  
    - \, \mathcal{B}_{41} & - \, \mathcal{B}_{42} & 0 & I_{3}
    \end{pmatrix} 
    \cdot 
    \begin{pmatrix}
     \tilde{Q}_{1} \\
     \tilde{R}_{1} \\
     \tilde{Q}_{2} \\
     \tilde{R}_{2}
    \end{pmatrix}
    =
    \begin{pmatrix}
     \frac{i \, k \, \eta_{0}}{\pm \, c_{0}} \, a^{3-h} \, {\bf P}_{0, 1}^{(1)} \cdot H^{Inc}(z_{1}) \\
     a^{3} \, {\bf P}_{0, 1}^{(2)} \cdot E^{Inc}(z_{1}) \\
     i \, k \, \eta_{2} \, a^{5} \, {\bf P}_{0, 2}^{(1)} \cdot H^{Inc}(z_{2}) \\
     \frac{\eta_{2}}{\pm \, d_{0}} \, a^{3-h} \, {\bf P}_{0, 2}^{(2)} \cdot E^{Inc}(z_{2})
    \end{pmatrix}.
\end{equation}
It is clear that the difference between the solution to $(\ref{unPerAS})$ and the solution to $(\ref{PerAS})$ satisfies the following algebraic system,
\begin{equation*}\label{PerAS-unPerAS}
    \begin{pmatrix}
    I_{3} & 0 & - \, \mathcal{B}_{13} & - \, \mathcal{B}_{14} \\
    0 & I_{3} & - \, \mathcal{B}_{23}  & - \, \mathcal{B}_{24} \\
    - \, \mathcal{B}_{31} & - \, \mathcal{B}_{32} & I_{3} & 0 \\  
    - \, \mathcal{B}_{41} & - \, \mathcal{B}_{42} & 0 & I_{3}
    \end{pmatrix} 
    \cdot 
    \begin{pmatrix}
     Q_{1} \, - \, \tilde{Q}_{1} \\
     R_{1} \, - \, \tilde{R}_{1}\\
     Q_{2} \, - \, \tilde{Q}_{2} \\
     R_{2} \, - \, \tilde{R}_{2}
    \end{pmatrix}
    \, = \, \begin{pmatrix}
    Error_1^{(1)} \\
     Error_1^{(2)} \\
     Error_2^{(1)} \\
     Error_2^{(2)}
    \end{pmatrix},
\end{equation*}
which is invertible under the condition $4 \, - \, h \, - \, 4 \, t \, > \, 0$. By inverting the above system using Born series, we obtain 
\begin{equation}\label{QRKErr}
     \begin{pmatrix}
     Q_{1} \, - \, \tilde{Q}_{1} \\
     R_{1} \, - \, \tilde{R}_{1}\\
     Q_{2} \, - \, \tilde{Q}_{2} \\
     R_{2} \, - \, \tilde{R}_{2}
    \end{pmatrix} \, = \, \sum_{n \geq 0} \mathbb{K}_{n} \cdot \begin{pmatrix}
    Error_1^{(1)} \\
     Error_1^{(2)} \\
     Error_2^{(1)} \\
     Error_2^{(2)}
    \end{pmatrix},
\end{equation}
where the matrix $\mathbb{K}_{n}$ is given by  
\begin{equation}\label{DefKn}
    \mathbb{K}_{n} \, := \, \begin{pmatrix}
    0 & 0 &  \mathcal{B}_{13} &  \mathcal{B}_{14} \\
    0 & 0 &  \mathcal{B}_{23}  &  \mathcal{B}_{24} \\
     \mathcal{B}_{31} &  \mathcal{B}_{32} & 0 & 0 \\  
     \mathcal{B}_{41} &  \mathcal{B}_{42} & 0 & 0
    \end{pmatrix}^{n}. 
\end{equation}
In order to evaluate the L.H.S of $(\ref{QRKErr})$, we keep only the zeroth order term and the first order term on the R.H.S of the Born series to derive the following estimations
\begin{eqnarray}\label{es-error-correction}
    |Q_1-\tilde{Q}_1| \, & \lesssim & \,  Error_1^{(1)}  \, + \, a^{3-h} \, d^{-3} \, Error_2^{(1)} \, + \, a^{3-h} \, d^{-2} \,  Error_2^{(2)} \; = \; \mathcal{O}\left(
    a^{\min(3; 7-2h-3t)} \right), \notag \\
    |R_1-\tilde{R}_1| & \lesssim & \, Error_1^{(2)} \, + \, \left( a^3 \, d^{-2} \, + \, a^{9-2h} \, d^{-8} \right) \,  Error_2^{(1)} \, + \, a^3 \, d^{-3} Error_2^{(2)} \; = \; \mathcal{O}\left( a^{\min(4; 7-h-4t)}\right), \notag\\ \nonumber
    |Q_2- \tilde{Q}_2|&\lesssim& Error_2^{(1)} \, + \, a^5 \, d^{-3} \, Error_1^{(1)} \, 
    + \, \left(a^5 d^{-2} + a^{11-2h} d^{-8}\right) \, Error_1^{(2)} \\ \nonumber &=& \; \mathcal{O}\left( a^{\min(6; 9-h-4t; 8-3t)}\right), \notag\\ \nonumber
    |R_2-\tilde{R}_2| \, &\lesssim& \, Error_2^{(2)} \, + \, a^{3-h} \, d^{-2} \,  Error_1^{(1)} \, + \, a^{3-h} \, d^{-3} \,  Error_1^{(2)} \\ &=& \; \mathcal{O}\left( a^{\min(4-h; 7-h-4t; 6-3t; 7-2h-3t; 10-2h-7t)}\right).
\end{eqnarray}

\medskip
 
\noindent (IV).  The corresponding revised Foldy-Lax approximation. 

\bigskip

We recall first the scattered field expansion given by $(\ref{Eq1132})$, 
\begin{equation*}
		E^{s}(x) \, =   \, - \, k^2  \, \sum_{m=1}^{2}  \, \left[  \underset{y}{\nabla} \Phi_{k}(x,z_{m}) \times Q_{m} \, -  \,  \Upsilon_{k}(x,z_{m}) \cdot  R_{m} \right] \, + \, \mathcal{O}\left( a^{4-h} \right),
\end{equation*}
where $\left(Q_{1}, R_{1}, Q_{2}, R_{2} \right)$ is solution to $(\ref{al-D1})$. Now, by using the estimates derived in $(\ref{es-error-correction})$, under the condition $4-h-4t>0$, and the fact that $x$ is away from $D$, we obtain 
\begin{equation*}
		E^{s}(x) \, =   \, - \, k^2  \, \sum_{m=1}^{2}  \, \left[  \underset{y}{\nabla} \Phi_{k}(x,z_{m}) \times \tilde{Q}_{m} \, -  \,  \Upsilon_{k}(x,z_{m}) \cdot  \tilde{R}_{m} \right] \, + \, \mathcal{O}\left( a^{\min(3; \, 7-2h-3t; \, 10-2h-7t)} \right),
\end{equation*}
where $\left(\tilde{Q}_{1}, \tilde{R}_{1}, \tilde{Q}_{2}, \tilde{R}_{2} \right)$ is solution to the unperturbed algebraic system $(\ref{unPerAS})$. 
This justify $(\ref{general-scattered-field})$. Similarly, by recalling the far-field expression given by $(\ref{Eq0254})$ and using the estimates derived in $(\ref{es-error-correction})$, with the condition $4-h-4t>0$, and the fact that $x$ is away from $D$, we can obtain 
\begin{equation*}
    E^\infty(\hat{x})=\frac{k^2}{4\pi}\sum_{m=1}^2 \, e^{- i k \hat{x}\cdot z_m}\left[ (I-\hat{x}\otimes\hat{x})\tilde{R}_m\, +\, i\, k\, \hat{x}\,\otimes\, \tilde{Q}_m\right] \, +\, \mathcal{O}(a^{\min(3; 7-2h-3t; 10-2h-7t)}),
\end{equation*}
    where $\left(\tilde{Q}_{1}, \tilde{R}_{1}, \tilde{Q}_{2}, \tilde{R}_{2} \right)$ is solution to $(\ref{unPerAS})$. This justify $(\ref{approximation-E})$, and ends the proof of Theorem \ref{main-1}. 
 \subsection{Proof of Corollary \ref{corollary-la}}
 
 We start by recalling, from $(\ref{unPerAS})$, that $\left(\tilde{Q}_{1}, \tilde{R}_{1}, \tilde{Q}_{2}, \tilde{R}_{2} \right)$ is a vector solution of the following algebraic system 
\begin{equation*}
    \begin{pmatrix}
    I_{3} & 0 & - \, \mathcal{B}_{13} & - \, \mathcal{B}_{14} \\
    0 & I_{3} & - \, \mathcal{B}_{23}  & - \, \mathcal{B}_{24} \\
    - \, \mathcal{B}_{31} & - \, \mathcal{B}_{32} & I_{3} & 0 \\  
    - \, \mathcal{B}_{41} & - \, \mathcal{B}_{42} & 0 & I_{3}
    \end{pmatrix} 
    \cdot 
    \begin{pmatrix}
     \tilde{Q}_{1} \\
     \tilde{R}_{1} \\
     \tilde{Q}_{2} \\
     \tilde{R}_{2}
    \end{pmatrix}
    =
    \begin{pmatrix}
     \frac{i \, k \, \eta_{0}}{\pm \, c_{0}} \, a^{3-h} \, {\bf P}_{0, 1}^{(1)} \cdot H^{Inc}(z_{1}) \\
     a^{3} \, {\bf P}_{0, 1}^{(2)} \cdot E^{Inc}(z_{1}) \\
     i \, k \, \eta_{2} \, a^{5} \, {\bf P}_{0, 2}^{(1)} \cdot H^{Inc}(z_{2}) \\
     \frac{\eta_{2}}{\pm \, d_{0}} \, a^{3-h} \, {\bf P}_{0, 2}^{(2)} \cdot E^{Inc}(z_{2})
    \end{pmatrix},
\end{equation*}
which is invertible under the condition $4 \, - \, h \, - \, 4 \, t \, > \, 0$. By inverting the above system using Born series, we obtain 
\begin{equation}\label{Sol=S+Err}
    \begin{pmatrix}
     \tilde{Q}_{1} \\
     \tilde{R}_{1} \\
     \tilde{Q}_{2} \\
     \tilde{R}_{2}
    \end{pmatrix}
    \, = \,
    \begin{pmatrix}
     \frac{i \, k \, \eta_{0}}{\pm \, c_{0}} \, a^{3-h} \, {\bf P}_{0, 1}^{(1)} \cdot H^{Inc}(z_{1}) \\
     a^{3} \, {\bf P}_{0, 1}^{(2)} \cdot E^{Inc}(z_{1}) \\
     i \, k \, \eta_{2} \, a^{5} \, {\bf P}_{0, 2}^{(1)} \cdot H^{Inc}(z_{2}) \\
     \frac{\eta_{2}}{\pm \, d_{0}} \, a^{3-h} \, {\bf P}_{0, 2}^{(2)} \cdot E^{Inc}(z_{2})
    \end{pmatrix} \, + \, \sum_{n \geq 1} \mathbb{K}_{n} \cdot \begin{pmatrix}
     \frac{i \, k \, \eta_{0}}{\pm \, c_{0}} \, a^{3-h} \, {\bf P}_{0, 1}^{(1)} \cdot H^{Inc}(z_{1}) \\
     a^{3} \, {\bf P}_{0, 1}^{(2)} \cdot E^{Inc}(z_{1}) \\
     i \, k \, \eta_{2} \, a^{5} \, {\bf P}_{0, 2}^{(1)} \cdot H^{Inc}(z_{2}) \\
     \frac{\eta_{2}}{\pm \, d_{0}} \, a^{3-h} \, {\bf P}_{0, 2}^{(2)} \cdot E^{Inc}(z_{2})
    \end{pmatrix},
\end{equation}
where the matrix $\mathbb{K}_{n}$ is given by $(\ref{DefKn})$.  
By keeping the dominant part of $\sum_{n \geq 1} \mathbb{K}_{n}$, which is $\mathbb{K}_{1}$, the second term on the R.H.S of the above equation will be reduced to   
\begin{equation*}
\mathbb{K}_{1} \cdot \begin{pmatrix}
     \frac{i \, k \, \eta_{0}}{\pm \, c_{0}} \, a^{3-h} \, {\bf P}_{0, 1}^{(1)} \cdot H^{Inc}(z_{1}) \\
     a^{3} \, {\bf P}_{0, 1}^{(2)} \cdot E^{Inc}(z_{1}) \\
     i \, k \, \eta_{2} \, a^{5} \, {\bf P}_{0, 2}^{(1)} \cdot H^{Inc}(z_{2}) \\
     \frac{\eta_{2}}{\pm \, d_{0}} \, a^{3-h} \, {\bf P}_{0, 2}^{(2)} \cdot E^{Inc}(z_{2})
    \end{pmatrix} \, \simeq
   \begin{pmatrix}
        i \, k \, \eta_{2} \, a^{5} \, \mathcal{B}_{13} \cdot {\bf P}_{0, 2}^{(1)} \cdot H^{Inc}(z_{2}) \, + \,
     \frac{\eta_{2}}{\pm \, d_{0}} \, a^{3-h} \, \mathcal{B}_{23} \cdot {\bf P}_{0, 2}^{(2)} \cdot E^{Inc}(z_{2}) \\ 
        i \, k \, \eta_{2} \, a^{5} \, \mathcal{B}_{23} \cdot {\bf P}_{0, 2}^{(1)} \cdot H^{Inc}(z_{2}) \, + \,
     \frac{\eta_{2}}{\pm \, d_{0}} \, a^{3-h} \, \mathcal{B}_{24} \cdot {\bf P}_{0, 2}^{(2)} \cdot E^{Inc}(z_{2}) \\
        \frac{i \, k \, \eta_{0}}{\pm \, c_{0}} \, a^{3-h} \, \mathcal{B}_{31} \cdot {\bf P}_{0, 1}^{(1)} \cdot H^{Inc}(z_{1}) \, + \, 
     a^{3} \, \mathcal{B}_{32} \cdot {\bf P}_{0, 1}^{(2)} \cdot E^{Inc}(z_{1}) \\
     \frac{i \, k \, \eta_{0}}{\pm \, c_{0}} \, a^{3-h} \, \mathcal{B}_{41} \cdot {\bf P}_{0, 1}^{(1)} \cdot H^{Inc}(z_{1}) \, + \, 
     a^{3} \, \mathcal{B}_{42} \cdot {\bf P}_{0, 1}^{(2)} \cdot E^{Inc}(z_{1})
    \end{pmatrix}. 
\end{equation*}
By estimating component by component the above vector, using the expression of the matrix  $\mathcal{B}_{ij}$, given in Proposition \ref{prop-la-1}, and plugging the obtained result into $(\ref{Sol=S+Err})$ we can obtain the following estimations 
\begin{eqnarray}\label{final-QR}
    \tilde{Q}_1 \, &=& \, \frac{i \, \eta_0 \, k}{\pm \, c_0} \, a^{3-h}\, {\bf P}_{0, 1}^{(1)} \cdot H_1^{Inc}(z_1) \, + \, \mathcal{O}\left( a^{9-3h-5t} \right), \notag\\
    \tilde{R}_1 \, &=& \, a^{3}\, {\bf P}_{0, 1}^{(2)} \cdot E_1^{Inc}(z_1) \, + \, \frac{k^2 \eta_0}{\pm d_0}\, a^{6-h}\, {\bf P}_{0, 1}^{(2)} \cdot \Upsilon_k(z_1, z_2)\cdot{\bf P}_{0, 2}^{(2)} \cdot E^{Inc}(z_2)\, + \, \mathcal{O}(a^{9-2h-5t}), \notag\\
    \tilde{Q}_2 \, &=& \, i \, \eta_2 \, k \,  a^{5}\, {\bf P}_{0, 2}^{(1)} \cdot H_2^{Inc}(z_2) \, + \, \frac{i k^5 \eta_0 \eta_2}{\pm c_0}\, a^{8-h}\, {\bf P}_{0, 2}^{(1)} \cdot \Upsilon_k(z_2, z_1) \cdot {\bf P}_{0, 1}^{(1)} \cdot H^{Inc}(z_1)\, +\, \mathcal{O}(a^{11-2h} d^{-5}), \notag\\
    \tilde{R}_2 \, &=&  \, \frac{\eta_2}{\pm \, d_{0}}  \,  a^{3-h}\, {\bf P}_{0, 2}^{(2)} \cdot E_2^{Inc}(z_2) \, + \, \mathcal{O}\left( a^{\min(11-2h-5t; 14-3h-8t)} \right). \notag\\
    &&
\end{eqnarray}
It is direct to observe from the above estimations that $\left( \tilde{Q}_{1}, \tilde{R}_{2}\right)$ dominates $\left( \tilde{R}_{1}, \tilde{Q}_{2}\right)$. Consequently, by returning back to \eqref{general-scattered-field} and using the above observation we can get that  
\begin{eqnarray*}
		E^{s}(x) \, &=&    \, k^2    \, \left[ \Upsilon_{k}(x,z_{2}) \cdot  \tilde{R}_{2} \, - \,  \underset{y}{\nabla} \Phi_{k}(x,z_{1}) \times \tilde{Q}_{1}    \right] \\
        &+&    \, k^2    \, \left[ \Upsilon_{k}(x,z_{1}) \cdot  \tilde{R}_{1} \, - \,  \underset{y}{\nabla} \Phi_{k}(x,z_{2}) \times \tilde{Q}_{2}    \right] + \, \mathcal{O}\left( a^{\min(3; 7-2h-3t; 10-2h-7t)} \right).
\end{eqnarray*}
Since $x$ is away from $D$ for the scattered wave, by utilizing the estimation of $\left(\tilde{R}_{1}, \tilde{Q}_{2} \right)$, we can know that the second term on the R.H.S is estimated of order  $\mathcal{O}\left(a^{3} \right)$. Hence, with \eqref{final-QR},
\begin{eqnarray*}
		E^{s}(x) \, &=&    \, k^2    \, \left[ \Upsilon_{k}(x,z_{2}) \cdot  \tilde{R}_{2} \, - \,  \underset{y}{\nabla} \Phi_{k}(x,z_{1}) \times \tilde{Q}_{1}    \right] \, + \, \mathcal{O}\left( a^{\min(3; 7-2h-3t; 10-2h-7t)} \right) \\
        &=& \, \pm   \, k^2  \, a^{3-h}  \, \left[ \frac{\eta_{2}}{d_{0}} \, 
        \Upsilon_{k}(x,z_{2}) \cdot  {\bf P}_{0, 2}^{(2)} \cdot E_2^{Inc}(z_2) \, - \, \frac{i \, k \, \eta_{0}}{c_{0}} \, \underset{y}{\nabla} \Phi_{k}(x,z_{1}) \times {\bf P}_{0, 1}^{(1)} \cdot H_1^{Inc}(z_1)    \right] \\ &+& \, \mathcal{O}\left( a^{\min(3; 7-2h-3t; 10-2h-7t; 9-3h-5t) } \right),
\end{eqnarray*}
which, by taking Taylor expansion of $H^{Inc}(\cdot)$ and $E^{Inc}(\cdot)$ at the intermediate point, between $z_{1}$ and $z_{2}$ with $\left\vert z_{1} \, - \, z_{2}\right\vert \, \lesssim \, a^{t}$, i.e.,
\begin{eqnarray*}
     	H_1^{Inc}(z_1)&=& H^{Inc}(z_0) + i k \int_0^1 e^{i k \theta\cdot(z_0+t(z_1-z_0))}(\theta^\perp\times\theta)\left\langle \theta, (z_1-z_0) \right\rangle\,dt \, = \, H^{Inc}(z_0) \, + \,  \mathcal{O}\left( a^{t} \right), \\
     	E_2^{Inc}(z_2)&=& E^{Inc}(z_0) + i k \int_0^1 e^{i k \theta\cdot(z_0+t(z_2-z_0))} (\theta^\perp\times\theta) \left\langle \theta, (z_2-z_0) \right\rangle\,dt\, = \, E^{Inc}(z_0) \, + \, \mathcal{O}\left( a^{t} \right), 
     \end{eqnarray*}
gives us under the condition $4-h-4t>0$ that,  
\begin{eqnarray}\label{Es}
        E^{s}(x) \, &=& \, \pm   \, k^2  \, a^{3-h}  \, \left[ \frac{\eta_{2}}{d_{0}} \, 
        \Upsilon_{k}(x,z_{2}) \cdot  {\bf P}_{0, 2}^{(2)} \cdot E^{Inc}(z_0) \, - \, \frac{i \, k \, \eta_{0}}{c_{0}} \, \underset{y}{\nabla} \Phi_{k}(x,z_{1}) \times {\bf P}_{0, 1}^{(1)} \cdot H^{Inc}(z_0)    \right] \notag\\ &+& \, \mathcal{O}\left( a^{\min(3-h+t;\; 3; \; 10-2h-7t;\; 9-3h-5t) } \right).
\end{eqnarray}
Based on \eqref{Es}, the last step consists in taking the Taylor expansion near the intermediate point $z_{0}$ for the functions $\Upsilon_{k}(x,\cdot)$ and $\nabla \Phi_{k}(x,\cdot)$, i.e.,   
\begin{equation*}
    \Upsilon_{k}(x,z_{2}) \, = \, \Upsilon_{k}(x,z_{0}) \, + \, \mathcal{O}\left( a^{t} \right) \quad \text{and} \quad  \nabla \Phi_{k}(x,z_{1}) \, = \, \nabla \Phi_{k}(x,z_{0}) \, + \, \mathcal{O}\left( a^{t} \right),
\end{equation*}
where we can use the fact that $x$ is away from $D$ to derive,
\begin{eqnarray*}
        E^{s}(x) \, &=& \, \pm   \, k^2  \, a^{3-h}  \, \left[ \frac{\eta_{2}}{d_{0}} \, 
        \Upsilon_{k}(x,z_{0}) \cdot  {\bf P}_{0, 2}^{(2)} \cdot E^{Inc}(z_0) \, - \, \frac{i \, k \, \eta_{0}}{c_{0}} \, \underset{y}{\nabla} \Phi_{k}(x,z_{0}) \times {\bf P}_{0, 1}^{(1)} \cdot H^{Inc}(z_0)    \right] \\ &+& \, \mathcal{O}\left( a^{\min(3-h+t;\; 3; \; 10-2h-7t;\; 9-3h-5t) } \right),
\end{eqnarray*}
which justifies $(\ref{CoroSF})$. Furthermore, by returning back to \eqref{approximation-E} we obtain  \begin{eqnarray*}
		E^{\infty}(\hat{x})  \, &=& \, \frac{k^{2}}{4 \, \pi} \left[   \, e^{- \, i  k  \hat{x} \cdot z_{2}}  \,  \left( I - \hat{x} \otimes \hat{x} \right) \cdot  \tilde{R}_{2} \, + \, i \, k \, e^{- \, i  k  \hat{x} \cdot z_{1}} \,  \hat{x} \times \tilde{Q}_{1} \right] \\
        &+& \, \frac{k^{2}}{4 \, \pi} \left[   \, e^{- \, i  k  \hat{x} \cdot z_{1}}  \,  \left( I - \hat{x} \otimes \hat{x} \right) \cdot  \tilde{R}_{1} \, + \, i \, k \,  \, e^{- \, i  k  \hat{x} \cdot z_{2}} \, \hat{x} \times \tilde{Q}_{2} \right] \, + \, \mathcal{O}\left( a^{\min(3; 7-2h-3t; 10-2h-7t)} \right),
\end{eqnarray*}
which, by using the expression of $\left(\tilde{R}_{1}, \tilde{Q}_{2} \right)$ in \eqref{final-QR}, and the fact that the second term on the R.H.S of the above equation is of order $\mathcal{O}\left(a^{3} \right)$, can be further simplified as 
\begin{equation*}
		E^{\infty}(\hat{x})  \, = \, \frac{k^{2}}{4 \, \pi} \left[   \, e^{- \, i  k  \hat{x} \cdot z_{2}}  \,  \left( I - \hat{x} \otimes \hat{x} \right) \cdot  \tilde{R}_{2} \, + \, i \, k \, e^{- \, i  k  \hat{x} \cdot z_{1}} \,  \hat{x} \times \tilde{Q}_{1} \right]  \, + \, \mathcal{O}\left( a^{\min(3; 7-2h-3t; 10-2h-7t)} \right).
\end{equation*}
Now, taking Taylor expansion of $z \rightarrow e^{- \, i \, k \hat{x} \cdot z}$ at the intermediate point $z_0$, between $z_{1}$ and $z_{2}$ with $\left\vert z_{1} - z_{2} \right\vert \, \lesssim \, a^{t}$, we obtain
\begin{eqnarray*}
		E^{\infty}(\hat{x})  & = & \frac{k^{2}}{4 \, \pi}\, e^{- \, i  k  \hat{x} \cdot z_{0}} \, \left[     \,  \left( I - \hat{x} \otimes \hat{x} \right) \cdot  \tilde{R}_{2} \, + \, i \, k  \,  \hat{x} \times \tilde{Q}_{1} \right] \notag\\
        &+ & \mathcal{O}\left(a^{t} \; \left(\left\vert \tilde{R}_{2} \right\vert + \left\vert \tilde{Q}_{1} \right\vert \right) \right)  \, + \, \mathcal{O}\left( a^{\min(3; 7-2h-3t; 10-2h-7t)} \right).
\end{eqnarray*}
Now, using the expression of $\left(\tilde{R}_{2}, \, \tilde{Q}_{1} \right)$ in \eqref{final-QR}, we can derive the following estimation 
\begin{equation*}
    \left\vert a^{t} \; \left(\left\vert R_{2} \right\vert + \left\vert Q_{1} \right\vert \right) \right\vert \; \lesssim \;  a^{3-h+t}, 
\end{equation*}
and the following expression for $E^{\infty}(\cdot)$ under the condition $4-h-4t>0$ that, 
\begin{eqnarray*}
		E^{\infty}(\hat{x})  \, &=& \, \pm \, \frac{k^{2}}{4 \, \pi}\, a^{3-h} \, e^{- \, i  k  \hat{x} \cdot z_{0}} \, \left[ \frac{\eta_{2}}{d_{0}} 
        \,  \left( I - \hat{x} \otimes \hat{x} \right) \cdot  {\bf P}_{0, 2}^{(2)} \cdot E_2^{Inc}(z_0) \, - \, k^{2} \frac{\eta_{0}}{c_{0}} 
    \,  \hat{x} \times {\bf P}_{0, 1}^{(1)} \cdot H_1^{Inc}(z_0) \right]  \\ &+& \,\mathcal{O}\left( a^{\min(3-h+t;\; 3; \; 10-2h-7t;\; 9-3h-5t) } \right),
\end{eqnarray*}
where the last step consists in taking Taylor expansion for the incident fields, similar to \eqref{Es}, to justify $(\ref{CoroFF})$.  The proof of Corollary \ref{corollary-la} is now complete. 
\section{Proof of the a-priori estimates}\label{sec-proof-prior}
There will be three subsections in this section, in which we respectively prove Lemma \ref{es-oneP}, Proposition \ref{lem-es-multi} and Proposition \ref{prop-scoeff}.

\subsection{Proof of Lemma \ref{es-oneP}}{(Estimate for just one nano-particle)}\label{subsec-51}
To justify the estimation related to the dielectric nano-particle, i.e. estimation of  $\lVert \tilde{E} \rVert_{\mathbb{L}^2(B_1)}$, we refer the readers to \cite[Subsection 5.1, Proof of Proposition 2.2]{CGS} for a detailed proof on the estimation of the $\mathbb{L}^{2}$-norm of the electric field generated by a single dielectric nano-particle embedded into a background made up of a vacuum.  In this subsection, by following a similar arguments used in \cite[Subsection 5.1]{CGS} with some necessary slight changes, we investigate only the estimation of $\lVert \tilde{E} \rVert_{\mathbb{L}^2(B_2)}$, i.e. we estimate the $\mathbb{L}^{2}$-norm of the electric field generated by a single plasmonic nano-particle embedded into a background of a vacuum. To do this, we start by recalling the L.S.E given by $(\ref{LS eq2})$,
	\begin{equation*}
	E \, + \, \eta_2 \, \nabla M^{k}_{D_{2}}(E) \, - \,  k^{2}  \, \eta_2 \, N^{k}_{D_{2}}(E) \, = \, E^{Inc}, \quad \text{in} \;\; D_{2},
	\end{equation*}
	which after scaling to $B_2$ yields
	\begin{equation}\label{Equa1104}
	\tilde{E} \, + \, \eta_2 \, \nabla M^{ka}_{B_{2}}(\tilde{E}) \, - \, k^{2}  \, \eta_2 \, a^{2} \, N^{ka}_{B_{2}}(\tilde{E}) \, = \, \tilde{E}^{Inc}, \quad \text{in} \;\; B_2.
	\end{equation}
 The next step consists in projecting $\tilde{E}$ onto the three sub-spaces introduced in \eqref{L2-decomposition}. We only note the main results. The reader should refer to \cite[Subsection 5.1]{CGS} for the calculations details.
 \begin{enumerate}
     \item[]
     \item Projecting onto $\mathbb{H}_0(Curl=0)(B_{2})$. By using the fact that 
     \begin{equation*}
     \tilde{E}^{Inc} \in \mathbb{H}\left( \div = 0 \right)(B_{2}) \, = \, \left( \mathbb{H}_{0}\left( \div = 0 \right) \oplus \nabla \mathcal{H}armonic \right)(B_{2}) \perp \mathbb{H}_{0}\left( Curl = 0 \right)(B_{2}),
     \end{equation*}
     and the relation $(\ref{grad-M-2nd})$, in $(\ref{Equa1104})$, we can derive that $\left\langle \tilde{E} , e_{n, B_{2}}^{(2)} \right\rangle_{\mathbb{L}^{2}(B_{2})} \, = \, 0$, hence 
     \begin{equation}\label{Exact2ndproj}
               \overset{2}{\mathbb{P}}\left( \tilde{E} \right)   \; = \; 0,
     \end{equation}
     		which justifies \eqref{*add1}.
     \item[]
     \item Projecting onto $\mathbb{H}_0(\div=0)(B_{2})$. In $(\ref{Equa1104})$, by taking the inner product with respect to $e^{(1)}_{n, B_{2}}(\cdot)$, using $(\ref{grad-M-1st})$ and $(\ref{Exact2ndproj})$, summing up with respect to the index $n$, we end up with the following estimation
  \begin{equation}\label{12120111}
     \left\Vert \overset{1}{\mathbb{P}}\left( \tilde{E} \right) \right\Vert^{2}_{\mathbb{L}^{2}(B_{2})}  \; \lesssim \; \left\lVert \overset{1}{\PP}\left(\tilde{E}^{inc}\right)\right\rVert_{\mathbb{L}^2(B_2)}^2    + a^{8}  \, \left\Vert \overset{3}{\mathbb{P}}\left( \tilde{E} \right) \right\Vert^{2}_{\mathbb{L}^{2}(B_2)}.
  \end{equation}
     \item[]
     \item Projecting onto $\nabla\mathcal{H}armonic(B_{2})$. In $(\ref{Equa1104})$, by taking the inner product with respect to $e^{(3)}_{n, B_{2}}(\cdot)$, using $(\ref{expansion-gradMk})$ and $(\ref{expansion-Nk})$, and keeping only the dominant terms we can end up with the following approximation
     \begin{equation*}
         \left\langle \tilde{E}, e^{(3)}_{n, B_{2}} \right\rangle_{\mathbb{L}^{2}(B_{2})}  \, \simeq \, \frac{\left\langle \tilde{E}^{Inc},e^{(3)}_{n, B_{2}} \right\rangle_{\mathbb{L}^{2}(B_{2})}}{\left( 1 + \eta_2  \, \lambda^{(3)}_{n}(B_{2}) \right)} + \frac{\eta_2}{\left( 1 + \eta_2  \, \lambda^{(3)}_{n}(B_{2}) \right)}  \; \frac{\left( ka \right)^{2}}{2} \, \left\langle N_{B_{2}}\left( \tilde{E}\right) ,  e^{(3)}_{n, B_{2}} \right\rangle_{\mathbb{L}^{2}(B_{2})}. 
     \end{equation*}
     Taking the square modulus on the both sides of the above equation, summing up with respect to the index $n$ and using $(\ref{def-eta})$ and \eqref{condition-on-k}, we deduce 
   \begin{equation}\label{Hsah}
	\left\Vert \overset{3}{\mathbb{P}}\left( \tilde{E} \right) \right\Vert^{2}_{\mathbb{L}^{2}(B_2)} \lesssim a^{-2h} \; \left| \left\langle \tilde{E}^{inc}, e_{n_{*}, B_{2}}^{(3)}\right\rangle_{\mathbb{L}^{2}(B_2)} \right|^2 +\sum_{n\neq n_{*}} \frac{\left| \left\langle \tilde{E}^{inc}, e_{n, B_{2}}^{(3)} \right\rangle\right|^2}{\left|1+\eta_2 \lambda_{n}^{(3)}(B_{2})\right|^2}+ a^{4-2h} \left\Vert  \tilde{E} \right\Vert^{2}_{\mathbb{L}^{2}(B_2)}.
\end{equation}
  \end{enumerate}
	Now, by adding $(\ref{12120111})$, $(\ref{Exact2ndproj})$ and  $(\ref{Hsah})$, under the condition $h \, < \, 2$, we deduce  
	\begin{equation}\notag
	\left\Vert \tilde{E} \right\Vert_{\mathbb{L}^{2}(B_2)} \; = \; \mathcal{O}\left( a^{-h} \right). 
	\end{equation}
This ends the proof of Lemma \ref{es-oneP}.
\subsection{Proof of Proposition \ref{lem-es-multi}}{(Estimation for the dimer).}\label{subsec-52}
For shortness notation, we introduce the following integro-differential operator  
\begin{equation}\label{def-Tkm}
\boldsymbol{T}_{k, m} \, := \, \left(I \, + \, \eta_m \, \nabla M^{k}_{D_{m}} \, - \, k^{2} \, \eta_m \, N^{k}_{D_{m}} \right), \quad\mbox{for}\quad m=1, \, 2.
\end{equation}
Since the operators $\nabla M^{-k}_{D_{m}}(\cdot)$ and $N^{-k}_{D_{m}}(\cdot)$ are the adjoint operators of $\nabla M^{k}_{D_{m}}(\cdot)$ and $N^{k}_{D_{m}}(\cdot)$, respectively, we deduce that $\boldsymbol{T}_{- k, m}(\cdot)$ is the adjoint operator of $\boldsymbol{T}_{k, m}(\cdot)$. To justify the estimations given by $(\ref{max-P1P3})$, we start by recalling the L.S.E, given by \eqref{LS eq2}, 
\begin{eqnarray}\label{LS-es}
\nonumber
	E(x) \, &+& \,  \left(\eta_1 \,  \underset{x}{\nabla}M^{k}_{D_{1}}(E_{1})(x) \, + \, \eta_2 \,  \underset{x}{\nabla}M^{k}_{D_{2}}(E_{2})(x)\right) \\ &-& k^{2}  \, \left(\eta_1 \, N^{k}_{D_{1}}(E_{1})(x) \, + \, \eta_2 \, N^k_{D_{2}}(E_2)(x)\right) \, = \,  E^{Inc}(x), \quad x \in D.
\end{eqnarray}
Now, we estimate the following $\mathbb{L}^{2}$-norm,
\begin{equation*}\label{MLXC}
   \left\Vert \overset{j}{\mathbb{P}}\left( \tilde{E}_{m} \right) \right\Vert_{\mathbb{L}^{2}(B_{m})}, \quad  m=1, 2 \quad \text{and} \quad j=1,3.
\end{equation*}
\begin{enumerate}
        \item[] 
	\item Estimation of $ \left\Vert  \overset{1}{\mathbb{P}}\left(\tilde{E}_{1} \right) \right\Vert^{2}_{\mathbb{L}^2(B_1)}$. \\
 Let $x \in D_{1}$ and rewrite $(\ref{LS-es})$ as\footnote{As $x$ and $y$ are in two disjoint domains, the integral appearing on the left hand side of $(\ref{*add3})$ is well defined.} 
	\begin{equation}\label{*add3}
	\TT_{k ,1}(E_1)(x) \, - \,  \eta_2 \, \left( -\nabla M^{k}_{D_{2}} \, + \, k^2 \, N^{k}_{D_{2}} \right) \, \left(E_2\right) (x)\, = \, E^{Inc}_1(x), \quad x \in D_{1}. 
	\end{equation} 
     Then, by using $(\ref{grad-M-1st})$, taking the inverse operator of $\boldsymbol{T}_{k, 1}$ on the both sides of $(\ref{*add3})$ and using the fact that, for $m = 1, 2$, $E_{m} \, = \, \overset{1}{\mathbb{P}}\left(E_{m}\right) \, + \, \overset{3}{\mathbb{P}}\left(E_{m}\right)$, we get 
	\begin{eqnarray}\label{es-ls1}
		\overset{1}{\mathbb{P}}\left(E_{1}\right)(x)  &-& \, k^2 \, \eta_2 \, \boldsymbol{T}_{k, 1}^{-1}  \int_{D_2} \Phi_k(x,y) \cdot \overset{1}{\mathbb{P}}\left(E_{2}\right)(y) \, dy  \notag \\ 
		&=&\boldsymbol{T}_{k, 1}^{-1}\left(E^{Inc}_{1}\right)(x)-\overset{3}{\mathbb{P}}\left(E_{1}\right)(x) + \, \eta_2 \, \boldsymbol{T}_{k, 1}^{-1}  \left( -\nabla M_{D_2}^k +k^2 N_{D_2}^k\right) \left(\overset{3}{\mathbb{P}}\left(E_{2}\right)\right)(x), \quad x \in D_{1}.  
	\end{eqnarray} 
 Besides, we know that
 \begin{eqnarray*}\label{LS-split1}
     \left(-\nabla M_{D_2}^k +k^2 N_{D_2}^k\right)\left(\overset{3}{\PP}(E_2)\right)(x) \, 
     & = & \, I_{2}^{0}(x) \, + \, \int_{D_2}\underset{y}{\nabla} \, \underset{y}{\nabla} \, \left( \Phi_k \, - \, \Phi_0 \right)(x, y) \cdot \overset{3}{\PP}(E_2)(y)\,dy \\
      &+&  k^2 \int_{D_2} \Phi_k (x, y)\overset{3}{\PP}(E_2)(y)\,dy,
 \end{eqnarray*}
 where
 \begin{equation}\label{def-I-0-2}
     I_{2}^0(x) \, := \, \int_{D_2}\underset{y}{\nabla} \, \underset{y}{\nabla} \, \Phi_0(x, y) \cdot \overset{3}{\PP}(E_2)(y)\,dy \, = \, - \, \nabla M_{D_{2}}\left( \overset{3}{\PP}(E_2) \right)(x), \qquad x \in D_{1}.
 \end{equation}
 Then we can reformulate $(\ref{es-ls1})$ as
 \begin{eqnarray}\label{LS-split2}
     \overset{1}{\PP}(E_1)(x) & - & \eta_2 k^2 \boldsymbol{T}_{k, 1}^{-1} \int_{D_2} \Phi_k(x, y)\overset{1}{\PP}(E_2)(y)\,dy\notag\\
     & = &\boldsymbol{T}_{k, 1}^{-1} (E_1^{Inc})(x)- \overset{3}{\PP}(E_1)(x)+\eta_2 \boldsymbol{T}_{k, 1}^{-1}  (I_2^0)\notag\\
     & + &\eta_2 \boldsymbol{T}_{k, 1}^{-1} \int_{D_2} \underset{y}{\nabla} \underset{y}{\nabla}\left(\Phi_k-\Phi_0\right)(x, y) \cdot \overset{3}{\PP}(E_2)(y)\,dy \, + \, \eta_2 k^2 \boldsymbol{T}_{k, 1}^{-1}\int_{D_2} \Phi_k (x, y)\overset{3}{\PP}(E_2)(y)\,dy.
 \end{eqnarray}
   In \eqref{LS-split2}, by taking the Taylor expansion of $\Phi_k(\cdot,\cdot)$ and $\underset{y}{\nabla}\underset{y}{\nabla}(\Phi_k-\Phi_0)(\cdot,\cdot)$ respectively, we can get
	\begin{eqnarray}\label{LS-P11-taylor}
    \nonumber
		&& \overset{1}{\mathbb{P}}\left(E_{1} \right)(x) \\ \nonumber
   &-& \, \eta_2 \, k^{2} \, \boldsymbol{T}_{k ,1}^{-1}\underset{y}{\nabla}\left({\Phi_k}I\right)(z_1, z_2) \cdot \int_{D_2}\mathcal{P}(y, z_2)\cdot \overset{1}{\PP}(E_2)(y)\,dy \\ \nonumber
		&-& \eta_2 \,k^{2} \, \boldsymbol{T}_{k ,1}^{-1}    \int_{D_2} \int_{0}^{1} (1-t)(y-z_{2})^{\perp} \cdot \underset{y}{\nabla} \underset{y}{\nabla}\left( \Phi_{k} \right)(z_{1},z_{2}+t(y-z_{2}))\cdot (y-z_{2}) \,dt  \cdot   \overset{1}{\mathbb{P}}\left(E_{2}\right)(y) \; dy  \\ \nonumber
		&+& \eta_2 \,k^{2} \, \boldsymbol{T}_{k ,1}^{-1}    \int_{D_2}  \int_0^1  \underset{y}{\nabla} \underset{y}{\nabla} \left( \Phi_{k} \right)(z_{1}+t(\cdot - z_1),z_{2})\cdot (\cdot -z_{1}) \,dt  \cdot  \mathcal{P}(y, z_2) \, \cdot \overset{1}{\mathbb{P}}\left(E_{2}\right)(y) \; dy  \\ \nonumber
  &-& \eta_2\, k^2 \, \boldsymbol{T}_{k, 1}^{-1}\int_{D_2}\int_0^1 \int_0^1 (1-t) (y- z_2)^\perp \cdot \underset{x}{\nabla} (\underset{y}{\nabla} \underset{y}{\nabla}\Phi_k)(z_1+s(\cdot-z_1), z_2+t(y-z_2)) \cdot (y-z_2)\,dt \cdot \\ && \qquad \qquad \qquad \qquad \qquad \qquad \cdot
  \mathcal{P}(\cdot, z_1)\, ds\cdot \overset{1}{\PP}(E_2)(y)\,dy \notag\\ 
		&=&  \boldsymbol{T}_{k, 1}^{-1}\left(E^{Inc}_{1}\right)(x)-\overset{3}{\mathbb{P}}\left(E_{1}\right)(x) +  \, \eta_2 \, \boldsymbol{T}_{k, 1}^{-1}(I_2^0)(x) \, +\,  \eta_2 \, \boldsymbol{T}_{k, 1}^{-1}    \underset{y}{\nabla} \underset{y}{\nabla}(\Phi_k- \Phi_0)(z_1, z_2) \cdot \int_{D_{2}}  \overset{3}{\mathbb{P}}\left(E_{2}\right)(y) \, dy \notag\\ \nonumber
  & + &  \eta_2 \, \boldsymbol{T}_{k, 1}^{-1} \int_{D_2}\int_0^1\underset{x}{\nabla} \left(\underset{y}{\nabla} \underset{y}{\nabla}(\Phi_k- \Phi_0)\right)(z_1+t(\cdot-z_1), z_2) \cdot \mathcal{P}(\cdot, z_1)\,dt\cdot\overset{3}{\PP}(E_2)(y)\,dy\\
  & + & \eta_2 \, \boldsymbol{T}_{k, 1}^{-1} \int_{D_2}\int_0^1\underset{y}{\nabla} \left(\underset{y}{\nabla}\underset{y}{\nabla}(\Phi_k- \Phi_0)\right)(z_1, z_2+t(y-z_2)) \cdot \mathcal{P}(y, z_2)\,dt\cdot\overset{3}{\PP}(E_2)(y)\,dy\notag\\ \nonumber
		&+& \eta_2 \, k^2  \, \boldsymbol{T}_{k, 1}^{-1} (\Phi_k I)(z_1, z_2) \cdot \int_{D_{2}} \overset{3}{\mathbb{P}}\left(E_{2}\right)(y) \, dy \\ &+& \, \eta_2 \, k^2\, \boldsymbol{T}_{k, 1}^{-1} \int_{D_2}\int_{0}^1 \underset{x}{\nabla}(\Phi_k I)(z_1+t(\cdot - z_1), z_2) \cdot \mathcal{P}(\cdot, z_1)\,dt\cdot\overset{3}{\PP}(E_2)(y)\,dy\notag\\
  && \qquad \qquad + \, \eta_2 \, k^2\, \boldsymbol{T}_{k, 1}^{-1} \int_{D_2}\int_{0}^1 \underset{y}{\nabla}(\Phi_k I)(z_1, z_2+t(y-z_2))\mathcal{P}(y, z_2)\,dt\cdot\overset{3}{\PP}(E_2)(y)\,dy.
	\end{eqnarray} 
	By scaling the above equation, from $D_{m}$ to $B_{m}$, with $m=1,2$, taking the $\mathbb{L}^{2}(B_{1})$-inner product with respect to $e_{n, B_{1}}^{(1)}(\cdot)$ and with the help of the adjoint operator of $\boldsymbol{T}^{-1}_{k  a, 1}$, that $\left( \boldsymbol{T}^{-1}_{k  a, 1} \right)^{\star} = \boldsymbol{T}^{-1}_{- \, k  a, 1}$, we further have
	\begin{eqnarray}\label{Sys-Equa-HA}
	\nonumber
&& \left\langle \tilde{E}_{1}, e_{n, B_{1}}^{(1)} \right\rangle_{\mathbb{L}^{2}(B_{1})} \\ \nonumber
 &-& \, \eta_2 \,k^{2} \,a^{4} \left\langle \underset{y}{\nabla}(\Phi_k I)(z_1, z_2) \cdot \int_{B_2}\mathcal{P}(y, 0)\cdot \overset{1}{\PP}(\tilde{E}_2)(y)\,dy,\, \boldsymbol{T}_{-ka, 1}^{-1}\left(e_{n, B_{1}}^{(1)}\right)\right\rangle_{\mathbb{L}^{2}(B_{1})} \\ \nonumber
	&-& \eta_2 \,k^{2} \,a^{5}  \, \left\langle   \int_{B_2} \int_{0}^{1} (1-t) y^{\perp} \cdot \underset{y}{\nabla}\underset{y}{\nabla}\left( \Phi_{k} \right)(z_{1},z_{2}+t a  y) \cdot y \,  dt \cdot   \overset{1}{\mathbb{P}}\left(\tilde{E}_{2}\right)(y) \; dy, \boldsymbol{T}_{- \, k  a, 1}^{-1}\left( e_{n, B_{1}}^{(1)} \right) \right\rangle_{\mathbb{L}^{2}(B_{1})} \\ \nonumber
	&+& \eta_2 \,k^{2} \,a^{5}  \, \left\langle   \int_{B_2} \int_{0}^{1}   \underset{y}{\nabla} \underset{y}{\nabla}\left( \Phi_{k} \right)(z_{1}+ ta\cdot, z_{2}) \cdot \cdot \,  dt \cdot \mathcal{P}(y, 0) \cdot  \overset{1}{\mathbb{P}}\left(\tilde{E}_{2}\right)(y) \; dy, \boldsymbol{T}_{- \, k  a, 1}^{-1}\left( e_{n, B_{1}}^{(1)} \right) \right\rangle_{\mathbb{L}^{2}(B_{1})} \\ \nonumber
 &-&   \eta_2 \, k^2 \, a^6 \, \Big\langle \int_{B_2}\int_0^1\int_0^1 (1-t) y^\perp \cdot \underset{x}{\nabla}(\underset{y}{\nabla} \underset{y}{\nabla}\Phi_k)(z_1+sa\cdot, z_2+tay) \cdot y\, dt\cdot\mathcal{P}(\cdot, 0)\,ds\cdot \overset{1}{\PP}(\tilde{E}_2)(y)\,dy, \\ && \nonumber \qquad \qquad \qquad \qquad \qquad \qquad \boldsymbol{T}_{- \, k  a, 1}^{-1}\left( e_{n, B_{1}}^{(1)} \right) \Big\rangle_{\mathbb{L}^{2}(B_{1})} \notag\\ \nonumber
	&=&\left\langle \tilde{E}^{Inc}_{1},\boldsymbol{T}_{- \, k  a, 1}^{-1}\left( e_{n}^{(1)} \right) \right\rangle_{\mathbb{L}^{2}(B_{1})}+\left\langle \eta_2 \left(\tilde{I_2^0}\right), \boldsymbol{T}_{-ka, 1}^{-1}(e_{n, B_{1}}^{(1)}) \right\rangle_{\mathbb{L}^{2}(B_{1})} \notag\\
 &+& \eta_2\,  a^3\,  \left\langle \underset{y}{\nabla} \underset{y}{\nabla}(\Phi_k-\Phi_0)(z_1, z_2) \cdot \int_{B_2}\overset{3}{\PP}\left(\tilde{E}_2\right)(y)\,dy, \, \boldsymbol{T}_{-ka, 1}^{-1} (e_{n, B_{1}}^{(1)})\right\rangle_{\mathbb{L}^{2}(B_{1})} \notag\\
 & + & \eta_2 \, a^4\, \left\langle \int_{B_2}\int_0^1 \underset{x}{\nabla}\left(\underset{y}{\nabla} \underset{y}{\nabla}(\Phi_k-\Phi_0)\right)(z_1+ta\cdot, z_2)\cdot\mathcal{P}(\cdot, 0)\,dt\cdot\overset{3}{\PP}\left(\tilde{E}_2\right)(y)\,dy,\, \boldsymbol{T}_{-ka, 1}^{-1}(e_{n, B_{1}}^{(1)}) \right\rangle_{\mathbb{L}^{2}(B_{1})} \notag\\
  & + & \eta_2 \, a^4\, \left\langle \int_{B_2}\int_0^1 \underset{y}{\nabla}\left(\underset{y}{\nabla}\underset{y}{\nabla}(\Phi_k-\Phi_0)\right)(z_1, z_2+tay)\cdot\mathcal{P}(y, 0)\,dt\cdot\overset{3}{\PP}\left(\tilde{E}_2\right)(y)\,dy,\, \boldsymbol{T}_{-ka, 1}^{-1}(e_{n, B_{1}}^{(1)}) \right\rangle_{\mathbb{L}^{2}(B_{1})} \notag\notag\\
  & + & \eta_2 \, k^2\, a^3\, \left\langle (\Phi_k I) (z_1, z_2) \cdot \int_{B_2}\overset{3}{\PP}\left(\tilde{E}_2\right)(y)\,dy, \, \boldsymbol{T}_{-ka, 1}^{-1}(e_{n, B_{1}}^{(1)}) \right\rangle_{\mathbb{L}^{2}(B_{1})} \notag\\
  &+&  \eta_2 \, k^2\, a^4\, \left\langle \int_{B_2}\int_0^1 \underset{x}{\nabla}(\Phi_k I)(z_1+ta\cdot, z_2)\cdot\mathcal{P}(\cdot, 0)\,dt\cdot\overset{3}{\PP}\left(\tilde{E}_2\right)(y)\,dy,\, \boldsymbol{T}_{-ka, 1}^{-1}(e_{n, B_{1}}^{(1)}) \right\rangle_{\mathbb{L}^{2}(B_{1})} \notag\\
  & & \qquad + \, \eta_2 \, k^2\, a^4\, \left\langle \int_{B_2}\int_0^1 \underset{y}{\nabla}(\Phi_k I)(z_1, z_2+tay)\cdot\mathcal{P}(y, 0)\,dt\cdot\overset{3}{\PP}\left(\tilde{E}_2\right)(y)\,dy,\, \boldsymbol{T}_{-ka, 1}^{-1}(e_{n, B_{1}}^{(1)}) \right\rangle_{\mathbb{L}^{2}(B_{1})}.
	\end{eqnarray}  
 Thanks to \cite[Lemma 5.1]{CGS}, we know that 
		\begin{equation}\label{Tka1}
		\TT_{-ka, 1}^{-1}(e_{n, B_{1}}^{(1)}) \, = \, \frac{e_{n, B_{1}}^{(1)}}{1-k^2 a^2 \eta_1 \lambda_{n}^{(1)}(B_{1})} \, + \, R_{n, 1},
	\end{equation}
	where
	\begin{equation}\label{def-Rn1}
		R_{n, 1}:=\frac{1\mp c_0a^h}{4\pi \lambda_{n_{0}}^{(1)}(B_{1}) \left( 1 \, - \, k^2 \, a^2 \, \eta_1 \, \lambda_{n}^{(1)}(B_{1})\right)}\sum_{\ell \geq 1} \frac{(- i k a)^{\ell+1}}{(\ell+1)!}\TT_{-ka, 1}^{-1}\int_{B_1}| \cdot-y|^\ell e_{n, B_{1}}^{(1)}(y)\,dy.
	\end{equation}
Based on \eqref{Tka1} and \eqref{def-Rn1}, by taking the square modulus on the both sides of \eqref{Sys-Equa-HA}, summing up with respect to the index $n$, due to the fact that for any vector function $F$, there holds
\begin{equation}\label{sum-Rn1}
    \sum_n\left| \left\langle F, \, R_{n,1}\right\rangle_{\mathbb{L}^{2}(B_{1})} \right|^2=
    \begin{cases}
        0\quad &\mbox{if } F \mbox{ is a constant},\\
       \mathcal{O}\left( a^{4-4h} \, \left|\left\langle F, e_{n_{0}, B_{1}}^{(1)} \right\rangle_{\mathbb{L}^{2}(B_{1})} \right|^2 \right) \quad&\mbox{otherwise},
    \end{cases}
\end{equation}
seeing \cite[Section 5.2, Proof of Proposition 2.3]{CGS} for the detailed calculations, and knowing $(\ref{choice-k-1st-regime-Prop26})$, we can obtain that
\begin{eqnarray}\label{es-P11-1}
\nonumber
\left\lVert\overset{1}{\PP}\left(\tilde{E_1}\right)\right\lVert_{\mathbb{L}^2(B_1)}^2 \,
& \lesssim & 
\sum_n\frac{\left| \langle \tilde{E}_1^{Inc}, e_{n, B_{1}}^{(1)} \rangle_{\mathbb{L}^2(B_1)} \right|^2}{\left|1-k^2 a^2 \eta_1 \lambda_n^{(1)}(B_{1})\right|^2} + a^{4-4h}\left|\left\langle \tilde{E}_1^{Inc}, e_{n_{0}, B_{1}}^{(1)}\right\rangle_{\mathbb{L}^2(B_1)} \right|^2 \\ &+& \, \sum_n\frac{\left| \langle \eta_2 \left(\tilde{I_2^0}\right), e_{n, B_{1}}^{(1)} \rangle_{\mathbb{L}^2(B_1)} \right|^2}{\left|1-k^2 a^2 \eta_1 \lambda_n^{(1)}(B_{1})\right|^2} \, + \,  \sum_n\left| \left\langle \eta_2\left(\tilde{I_2^0}\right), R_{n,1}\right\rangle_{\mathbb{L}^2(B_1)} \right|^2 \, + \, \mathrm{T}_{4, 4}, 
\end{eqnarray}
where we have used the fact that $h<2$, and we have denoted by $\mathrm{T}_{4,4}$ the following term, 
\begin{eqnarray}\label{TermT4}
\nonumber
&& \mathrm{T}_{4, 4} \, := \, a^{8-2h} \, \Bigg[  \\ \nonumber && \,  a^{2} \left\lVert \int_{B_2}\int_0^1 \underset{y}{\nabla} \underset{y}{\nabla}\Phi_k(z_1+ta\cdot, z_2)\cdot \cdot \,dt\cdot \mathcal{P}(y, 0)\cdot\overset{1}{\PP}\left(\tilde{E}_2\right)(y)\,dy\right\rVert_{\mathbb{L}^2(B_1)}^2\notag\\
&+& a^{4} \left\lVert \int_{B_2}\int_0^1\int_0^1 (1-t) y^\perp \cdot \underset{x}{\nabla}\left(\underset{y}{\nabla}\underset{y}{\nabla}\Phi_k\right)(z_1+sa\cdot, z_2+tay)\cdot y\,dt\cdot \mathcal{P}(x, 0)\,ds\cdot\overset{1}{\PP}\left(\tilde{E}_2\right)(y)\,dy\right\rVert_{\mathbb{L}^2(B_1)}^2\notag\\
&+& \left\lVert \int_{B_2}\int_0^1 \underset{x}{\nabla} \left(\underset{y}{\nabla}\underset{y}{\nabla}\left(\Phi_k-\Phi_0\right)\right)(z_1+ta\cdot, z_2)\cdot\mathcal{P}(x, 0)\,dt\cdot\overset{3}{\PP}\left(\tilde{E}_2\right)(y)\,dy\right\rVert_{\mathbb{L}^2(B_1)}^2\notag\\
&& \qquad \qquad \qquad + \, \left\lVert \int_{B_2}\int_0^1 \underset{x}{\nabla}\left(\Phi_k I\right)(z_1+ta\cdot, z_2)\cdot\mathcal{P}(\cdot, 0)\,dt\cdot\overset{3}{\PP}\left(\tilde{E}_2\right)(y)\,dy\right\rVert_{\mathbb{L}^2(B_1)}^2\Bigg].
\end{eqnarray}
Then, we evaluate the R.H.S. of \eqref{es-P11-1} term by term as follows.
\begin{enumerate}
    \item[]
    \item Estimation of
    \begin{equation*}
    \mathrm{T}_{4, 1} \, := \, \sum_n  \frac{\left|\left\langle \tilde{E}_1^{Inc}, e_{n, B_{1}}^{(1)}\right\rangle_{\mathbb{L}^2(B_1)} \right|^2}{\left|1-k^2 a^2 \eta_1 \lambda_n^{(1)}(B_{1})\right|^2} \, + \, a^{4-4h} \, \left| \left\langle \tilde{E}_1^{Inc}, e_{n_{0}, B_{1}}^{(1)} \right\rangle_{\mathbb{L}^2(B_1)} \right|^2.
    \end{equation*}
    By keeping the dominant term of $\mathrm{T}_{4, 1}$ and using the fact that $h \, < \, 2$, we get the following approximation 
    \begin{equation*}
    \mathrm{T}_{4, 1} \, \simeq \, \frac{\left|\left\langle \tilde{E}_1^{Inc}, e_{n_{0}, B_{1}}^{(1)}\right\rangle_{\mathbb{L}^2(B_1)} \right|^2}{\left|1 \, - \, k^2 \, a^2 \, \eta_1 \,  \lambda_{n_{0}}^{(1)}(B_{1})\right|^2} \, \overset{(\ref{choice-k-1st-regime-Prop26})}{\simeq} \, a^{-2h} \, \left|\left\langle \tilde{E}_1^{Inc}, e_{n_{0}, B_{1}}^{(1)}\right\rangle_{\mathbb{L}^2(B_1)} \right|^2.
    \end{equation*}
    Besides, we know that 
\begin{equation*}
\resizebox{.99\hsize}{!}{$
    \left\langle \tilde{E}_1^{Inc}, e_{n, B_{1}}^{(1)} \right\rangle_{\mathbb{L}^2(B_1)} \overset{\eqref{pre-cond}}{=}\left\langle \tilde{E}_1^{Inc}, Curl\left( \phi_{n, B_{1}}\right) \right\rangle_{\mathbb{L}^2(B_1)} \, = \, \left\langle Curl\left(\tilde{E}_1^{Inc}\right), \phi_{n, B_{1}}\right\rangle_{\mathbb{L}^2(B_1)} \, = \, i \, k \, a \, \left\langle \tilde{H}_1^{Inc}, \phi_{n, B_{1}} \right\rangle_{\mathbb{L}^2(B_1)}$},
\end{equation*}
Hence, 
\begin{equation}\label{P11-es1}
    \mathrm{T}_{4, 1} \, \simeq  \, a^{-2h} \, \left\vert  i \, k \, a \, \left\langle \tilde{H}_1^{Inc}, \phi_{n, B_{1}} \right\rangle_{\mathbb{L}^2(B_1)} \right\vert^{2} \, = \, \mathcal{O}\left( a^{2-2h} \right).
\end{equation}
\item[]
\item Estimation of 
\begin{equation*}
    \mathrm{T}_{4, 2} \, := \, \sum_n \frac{\left| \left\langle \eta_2 \left( \tilde{I_2^0}\right), e_n^{(1)}\right\rangle_{\mathbb{L}^2(B_1)} \right|^{2}}{\left|1-k^2 a^2 \eta_1 \lambda_n^{(1)}(B_{1})\right|^2}.
\end{equation*}
By scaling the term $I_{2}^{0}(\cdot)$, see $(\ref{def-I-0-2})$ for its definition, using the scale of $\eta_{2}$, see $(\ref{def-eta})$, and the relation given by $(\ref{choice-k-1st-regime-Prop26})$, we derive the following estimation 
\begin{eqnarray*}
    \mathrm{T}_{4, 2} \, &\lesssim& \, a^{6-2h} \,  \left\lVert \int_{B_2} \underset{y}{\nabla} \underset{y}{\nabla}\Phi_0(\cdot, y) \cdot \overset{3}{\PP}\left(\tilde{E}_2\right)(y)\,dy\right\rVert_{\mathbb{L}^2(B_1)}^2\notag\\
    &\leq & \, a^{6-2h} \, \left\lVert \int_{\mathbb{R}^3} \underset{y}{\nabla} \underset{y}{\nabla} \Phi_{0}(\cdot, y) \cdot \overset{3}{\PP}\left(\tilde{E}_2\right)(y) \, \chi_{B_2 \cup B_1}(y)\, dy \right\rVert_{\mathbb{L}^2(\mathbb{R}^3)}^2, 
\end{eqnarray*}
where $\chi_{B_2 \cup B_1}(\cdot)$ is the characteristic function related to the domain $B_2 \cup B_1$. Now, by using the Calder\'{o}n-Zygmund inequality, we deduce   
\begin{equation}\label{P11-es2}
    \mathrm{T}_{4, 2} \, \lesssim  \, a^{6-2h} \,  \left\lVert\overset{3}{\PP}\left(\tilde{E}_2\right) \, \chi_{B_2\cup B_1}\right\rVert_{\mathbb{L}^2(\mathbb{R}^3)}^2 \, = \, a^{6-2h} \, \left\lVert \overset{3}{\PP}\left(\tilde{E}_2\right)\right\rVert_{\mathbb{L}^2(B_2)}^2.
\end{equation}
\item[]
\item Estimation of 
\begin{equation*}
    \mathrm{T}_{4, 3} \, := \, \sum_n\left|\left\langle \eta_2 \left(\tilde{I_2^0}\right), R_{n,1}\right\rangle_{\mathbb{L}^2(B_1)} \right|^2 \, \underset{(\ref{def-eta})}{\overset{(\ref{sum-Rn1})}{\lesssim}} \, a^{4-4h} \, \left|\left\langle \tilde{I_2^0}, e_{n_{0}, B_{1}}^{(1)}\right\rangle_{\mathbb{L}^2(B_1)} \right|^2. 
\end{equation*}
By scaling the term $I_{2}^{0}(\cdot)$, see $(\ref{def-I-0-2})$, we derive the following estimation
\begin{equation}\label{P11-es3}
    \mathrm{T}_{4, 3} \,  \lesssim  \, a^{10-4h} \left\lVert \int_{B_2} \underset{y}{\nabla} \underset{y}{\nabla} \Phi_0(\cdot,y) \cdot \overset{3}{\PP}\left(\tilde{E}_2\right)(y)\,dy\right\rVert_{\mathbb{L}^2(B_1)}^2\lesssim a^{10-4h} \left\lVert \overset{3}{\PP}\left(\tilde{E}_2\right)\right\rVert_{\mathbb{L}^2{(B_2)}}^2,
\end{equation}
following the similar argument to \eqref{P11-es2}.
\item[]
\item Estimation of $\mathrm{T}_{4,4}$. \\
A straightforward computation, using $(\ref{TermT4})$, and the fact that 
\begin{equation*}
    \nabla \nabla (\Phi_k \, - \, \Phi_0)(\cdot, \cdot) \; \sim \; \Phi_0(\cdot, \cdot),
 \end{equation*}
gives us 
\begin{equation}\label{P11-es-4}
    \left| \mathrm{T}_{4,4} \right| \, \lesssim \, a^{10-2h} \, d^{-6} \, \left\lVert \overset{1}{\PP}\left(\tilde{E}_2\right)\right\rVert_{\mathbb{L}^2(B_2)}^2+a^{12-2h} d^{-8} \left\lVert \overset{1}{\PP}\left(\tilde{E}_2\right)\right\rVert_{\mathbb{L}^2(B_2)}^2 + a^{8-2h} d^{-4}\left\lVert \overset{3}{\PP}\left(\tilde{E}_2\right)\right\rVert_{\mathbb{L}^2(B_2)}^2.
\end{equation}
\end{enumerate}
By gathering \eqref{P11-es1},  \eqref{P11-es2}, \eqref{P11-es3} and \eqref{P11-es-4}, we can obtain that
\begin{equation}\label{es-P11}
\left\lVert\overset{1}{\PP}\left(\tilde{E}_1\right)\right\rVert_{\mathbb{L}^2(B_1)}^2 \lesssim  a^{2-2h} + a^{10-2h} d^{-6} \left\lVert\overset{1}{\PP}\left(\tilde{E}_2\right)\right\rVert_{\mathbb{L}^2(B_2)}^2+\left(a^{6-2h}+ a^{8-2h} d^{-4}\right)\left\lVert\overset{3}{\PP}\left(\tilde{E}_2\right)\right\rVert_{\mathbb{L}^2(B_2)}^2.
\end{equation}   
	\item[] 
	\item Estimation of $ \left\Vert  \overset{1}{\mathbb{P}}\left(\tilde{E}_{2} \right) \right\Vert^{2}_{\mathbb{L}^2(B_1)}$. \\
        By repeating the same calculations steps as $(\ref{LS-es})-(\ref{es-ls1})$, we can derive the following equation,
	\begin{eqnarray*}\label{Eq0451}
        \nonumber
		\overset{1}{\PP}\left(E_2\right)(x) \,
  &-& \, k^2 \, \eta_{1} \, \TT_{k, 2}^{-1}\int_{D_{1}} \Phi_k(x, y) \, \overset{1}{\PP}\left(E_1\right)(y)\,dy \\ &=& \, \TT_{k, 2}^{-1}E_2^{Inc}(x) \, - \, \overset{3}{\PP}\left(E_2\right)(x) \, + \, \eta_1 \, \TT_{k, 2}^{-1} \left(-\nabla M^k +k^2 N^k\right) \left(\overset{3}{\PP}\left(E_1\right)\right)(x), \quad x \in D_{2}.
	\end{eqnarray*}
        Hence, by similar expansions,
	\begin{eqnarray}\label{LS-P12-mid1}
		\overset{1}{\PP}\left( E_2 \right)(x) \,
  &-& \, \eta_1 k^2 \TT_{k, 2}^{-1} \underset{y}{\nabla}(\Phi_k I)(z_2, z_1) \cdot \int_{D_1}\mathcal{P}(y, z_1)\cdot\overset{1}{\PP}\left(E_1\right)(y)\,dy \notag\\
		&-& \eta_1 k^2 \TT_{k, 2}^{-1}\int_{D_{1}} \int_0^1 (1-t)(y-z_1)^\perp \cdot \underset{y}{\nabla} \underset{y}{\nabla} \Phi_{k}(z_2, z_1+(y-z_1))\cdot(y-z_1)\,dt\cdot \overset{1}{\PP}\left(E_1 \right)(y)\,dy\notag\\
		&+&\eta_1 k^2 \TT_{k, 2}^{-1}\int_{D_{1}} \int_0^1 \underset{y}{\nabla} \underset{y}{\nabla} \Phi_{k}(z_2+t(\cdot - z_2), z_1)\cdot (\cdot - z_2)\,dt\cdot \mathcal{P}(y, z_1)\cdot \overset{1}{\PP}\left(E_1 \right)(y)\,dy\notag\\
  &-& \eta_1 \, k^2 \, \TT_{k, 2}^{-1}\int_{D_1}\int_0^1 \int_0^1 (1-t) (y-z_1)^\perp \cdot \underset{x}{\nabla}(\underset{y}{\nabla} \underset{y}{\nabla} \Phi_k)(z_2+s(\cdot-z_2), z_1+t(y-z_1)) \notag\\
&&\qquad\qquad\qquad\qquad\,\cdot (y-z_1)\,dt\cdot\mathcal{P}(\cdot, z_2)\,ds\cdot\overset{1}{\PP}(E_1)(y)\,dy \notag\\ \nonumber
		&=& \TT_{k, 2}^{-1}E^{Inc}_2(x)-\overset{3}{\PP}\left(E_2 \right)(x) \, + \, \eta_1 \boldsymbol{T}_{k, 2}^{-1} \left(I_1^0\right) \\ \nonumber &+& \, \eta_1\boldsymbol{T}_{k, 2}^{-1}\underset{y}{\nabla} \underset{y}{\nabla} \left( \Phi_k-\Phi_0\right)(z_2, z_1) \cdot \int_{D_1}\overset{3}{\PP}\left(E_1\right)(y)\,dy\notag\\
  &+& \eta_1\boldsymbol{T}_{k, 2}^{-1} \int_{D_1}\int_0^1 \underset{x}{\nabla} \left( \underset{y}{\nabla} \underset{y}{\nabla} \left(\Phi_k-\Phi_0\right)\right)(z_2+t(\cdot -z_2), z_1)\cdot\mathcal{P}(\cdot, z_2)\,dt\cdot\overset{3}{\PP}\left(E_1\right)(y)\,dy  \notag\\
  &+&\eta_1\boldsymbol{T}_{k, 2}^{-1} \int_{D_1}\int_0^1 \underset{y}{\nabla} \left( \underset{y}{\nabla} \underset{y}{\nabla} \left(\Phi_k-\Phi_0\right)\right)(z_2, z_1+t(y-z_1))\cdot\mathcal{P}(y, z_1)\,dt\cdot\overset{3}{\PP}\left(E_1\right)(y)\,dy \notag\\
  &+& k^2 \eta_1 \boldsymbol{T}_{k, 2}^{-1} \Phi_k (z_2, z_1)\int_{D_1}\overset{3}{\PP}\left(E_1\right)(y)\,dy\notag\\
  &+ & k^2 \eta_1 \boldsymbol{T}_{k, 2}^{-1} \int_{D_1}\int_0^1 \underset{x}{\nabla}(\Phi_k I)(z_2+t(\cdot-z_2), z_1) \cdot \mathcal{P}(\cdot, z_2)\,dt\cdot \overset{3}{\PP}(E_1)(y)\,dy \notag\\
  &+ & k^2 \eta_1 \boldsymbol{T}_{k, 2}^{-1} \int_{D_1}\int_0^1 \underset{y}{\nabla}(\Phi_k I)(z_2, z_1+t(y-z_1)) \cdot \mathcal{P}(y, z_1)\,dt\cdot \overset{3}{\PP}(E_1)(y)\,dy, 
	\end{eqnarray}
where 
 \begin{equation}\label{def-I10}
     I_1^0(x) \, := \, \int_{D_1}\underset{y}{\nabla} \underset{y}{\nabla} \Phi_0(x, y) \cdot \overset{3}{\PP}(E_1)(y)\,dy \, = \, - \, \nabla M_{D_{1}}\left( \overset{3}{\PP}(E_1)\right)(x), \qquad x \in D_{2}.
 \end{equation} 
	Similar to \eqref{Tka1}, there holds
	\begin{equation}\label{Tka2}
		\TT_{-ka, 2}^{-1}(e_n^{(1)}) \, = \, \frac{e_n^{(1)}}{\left( 1 \, - \, k^2 \, a^2 \, \eta_2 \, \lambda_{n}^{(1)}(B_{2}) \right)} \, + \, \widetilde{R_{n, 1}}, 
	\end{equation}
        where
	\begin{equation}\label{Eq1140}
		\widetilde{R_{n, 1}}:=\frac{k^2 a^2 \eta_2}{4\pi\left(1- k^2 a^2 \eta_2\lambda_{n}^{(1)}(B_{2})\right)}\sum_{\ell \geq 1} \frac{(- i k a)^{\ell+1}}{(\ell+1)!}\TT_{-ka, 2}^{-1}\int_{B_2}\lvert \cdot-y \rvert^\ell e_{n, B_{2}}^{(1)}(y)\,dy.
	\end{equation}
	Then by scaling \eqref{LS-P12-mid1} from $D_{m}$ to $B_{m}$, with $m = 1, 2$, taking the $\mathbb{L}^2(B_2)$-inner product with $e_{n, B_{2}}^{(1)}(\cdot)$, together with \eqref{Tka2}, and summing up the modulus square of the obtained equation with respect to the index $n$, we can derive that 
	\begin{eqnarray}\label{es-P12-1}
    \nonumber
		\left\lVert\overset{1}{\PP}\left(\tilde{E}_2 \right)\right\rVert^2_{\mathbb{L}^2(B_2)} \, & \lesssim & \, \sum_n \frac{\left| \left\langle \tilde{E}_2^{Inc}, e_{n, B_{2}}^{(1)}\right\rangle_{\mathbb{L}^2(B_2)} \right|^2}{\left|1- k^2 a^2 \eta_2\lambda_{n}^{(1)}(B_{2})\right|^2} \, + \, \sum_n \left| \left\langle \tilde{E}^{Inc}_2, \widetilde{R_{n, 1}}\right\rangle_{\mathbb{L}^2(B_2)}\right|^2
              \\ \nonumber &+& |\eta_1|^2 \sum_n \left| \left\langle \tilde{I_1^0}, e_{n, B_{2}}^{(1)}\right\rangle_{\mathbb{L}^2(B_2)} \right|^2
  + |\eta_1|^2 \sum_n \left| \left\langle \tilde{I_1^0}, \widetilde{R_{n,1}}\right\rangle_{\mathbb{L}^2(B_2)}\right|^2 \\ &+& Err\left(\overset{1}{\PP}\left(\tilde{E}_1\right)\right)+Err\left(\overset{3}{\PP}\left(\tilde{E}_1\right)\right),
	\end{eqnarray} 
	where 
	\begin{eqnarray}\label{Err-P11-1}
    \nonumber
	&&	Err\left(\overset{1}{\PP}\left(\tilde{E}_1 \right)\right) \\
   &:=& \, |\eta_1|^2 a^8 \sum_n \left| \left\langle \underset{y}{\nabla}(\Phi_k I)(z_2, z_1) \cdot \int_{B_1}\mathcal{P}(y, 0)\cdot\overset{1}{\PP}\left(\tilde{E_1}\right)(y)\,dy, \widetilde{R_{n, 1}} \right\rangle_{\mathbb{L}^2(B_2)}\right|^2\notag\\
		&+& |\eta_1|^2 a^{10} \sum_n \left|\left\langle \int_{B_1}\int_0^1 (1-t) y^\perp \cdot \underset{y}{\nabla} \underset{y}{\nabla}\Phi_{k}(z_2, z_1+tay)\cdot y\,dt\cdot\overset{1}{\PP}\left(\tilde{E}_1 \right)(y)\,dy, \widetilde{R_{n, 1}} \right\rangle_{\mathbb{L}^2(B_2)} \right|^2\notag\\
		&+& \sum_{n}  \frac{|\eta_1|^2 a^{10}}{\left| 1- k^2 a^2 \eta_2 \lambda_{n}^{(1)}(B_{2}) \right|^2} \resizebox{.7\hsize}{!}{$ \left| \left\langle \int_{B_1}\int_0^1 \underset{y}{\nabla} \underset{y}{\nabla} \Phi_{k}(z_2+ta\cdot, z_1)\cdot \cdot \,dt\cdot \mathcal{P}(y, 0) \cdot\overset{1}{\PP}\left(\tilde{E}_1\right)(y)\,dy, e_{n, B_{2}}^{(1)} \right\rangle_{\mathbb{L}^2(B_2)}\right|^2$}\notag\\
		&+& |\eta_1|^2 a^{10}\sum_n \left| \left\langle \int_{B_1}\int_0^1\underset{y}{\nabla}\underset{y}{\nabla}\Phi_{k}(z_2+ta\cdot, z_1)\cdot \cdot \,dt\cdot \mathcal{P}(y, 0)\cdot\overset{1}{\PP}\left(\tilde{E}_1\right)(y)\,dy, \widetilde{R_{n, 1}} \right\rangle_{\mathbb{L}^2(B_2)}\right|^2\notag\\
  &+& \sum_n\frac{|\eta_1|^2 a^{12}}{\left| 1- k^2 a^2 \eta_2 \lambda_{n}^{(1)}(B_{2}) \right|^2}  \Bigg|\Big\langle 
\int_{B_1}\int_0^1 \int_0^1 (1-t)y^\perp \cdot \underset{x}{\nabla}(\underset{y}{\nabla}\underset{y}{\nabla} \Phi_k)(z_2+sa\cdot, z_1+tay)\notag\\
&&\qquad\qquad\qquad\qquad\qquad \cdot y\,dt \cdot  \mathcal{P}(\cdot, 0)\,ds\cdot\overset{1}{\PP}\left(\tilde{E}_1\right)(y)\,dy, e_{n, B_{2}}^{(1)} \Big\rangle_{\mathbb{L}^2(B_2)}\Bigg|^2\notag\\ \nonumber
&+& |\eta_1|^2 a^{12} \sum_n\Big|\Big\langle \int_{B_1}\int_0^1\int_0^1 (1-t) y^\perp \cdot \underset{x}{\nabla}(\underset{y}{\nabla} \underset{y}{\nabla}\Phi_k)(z_2+sa\cdot, z_1+tay)\cdot y\,dt \\ &&\qquad\qquad\qquad\qquad\qquad \cdot\mathcal{P}(\cdot, 0)\,ds\cdot\overset{1}{\PP}\left(\tilde{E}_1\right)(y)\,dy, \widetilde{R_{n, 1}} \Big\rangle_{\mathbb{L}^2(B_2)}\Big|^2
	\end{eqnarray}
	and  
	\begin{eqnarray}\label{Err-P31-1}
    \nonumber
		&& Err\left(\overset{3}{\PP}\left(\tilde{E}_1 \right)\right)
  \\ &:=& \, |\eta_1|^2 a^6 \sum_n \left|\left\langle 
\Phi_k(z_2, z_1)\int_{B_1}\overset{3}{\PP}\left( \tilde{E}_1\right)(y)\,dy, \widetilde{R_{n ,1}} \right\rangle_{\mathbb{L}^2(B_2)} \right|^2 \notag\\
  &+& |\eta_1|^2 a^6 \sum_n \left|\left\langle \underset{y}{\nabla} \underset{y}{\nabla}\left( \Phi_k-\Phi_0\right)(z_2, z_1) \cdot \int_{B_1} \overset{3}{\PP}\left(\tilde{E}_1\right)(y)\,dy, \widetilde{R_{n, 1}}\right\rangle_{\mathbb{L}^2(B_2)}\right|^2 \notag\\
  &+& \sum_n \frac{|\eta_1|^2 a^8}{\left|1-k^2 a^2 \eta_2 \lambda_n^{(1)}(B_2)\right|^2} \resizebox{.7\hsize}{!}{$  \left|\left\langle \int_{B_1}\int_0^1 \underset{x}{\nabla}\left(\underset{y}{\nabla}\underset{y}{\nabla}\left(\Phi_k-\Phi_0\right)\right)(z_2+ta\cdot, z_1)\cdot\mathcal{P}(\cdot, 0)\,dt\cdot\overset{3}{\PP}\left(\tilde{E}_1\right)(y)\,dy, e_{n, B_{2}}^{(1)} \right\rangle_{\mathbb{L}^2(B_2)}\right|^2$}\notag\\
  &+& |\eta_1|^2 a^8 \sum_n \left|\left\langle \int_{B_1}\int_0^1 \underset{x}{\nabla}\left(\underset{y}{\nabla} \underset{y}{\nabla} \left(\Phi_k-\Phi_0\right)\right)(z_2+ta\cdot, z_1)\cdot\mathcal{P}(\cdot, 0)\,dt\cdot\overset{3}{\PP}\left(\tilde{E}_1\right)(y)\,dy, \widetilde{R_{n, 1}} \right\rangle_{\mathbb{L}^2(B_2)}\right|^2\notag\\
  &+& |\eta_1|^2 a^8 \sum_n \left|\left\langle \int_{B_1}\int_0^1 \underset{y}{\nabla}\left(\underset{y}{\nabla}\underset{y}{\nabla}\left(\Phi_k-\Phi_0\right)\right)(z_2, z_1+tay)\cdot\mathcal{P}(y, 0)\,dt\cdot\overset{3}{\PP}\left(\tilde{E}_1\right)(y)\,dy, \widetilde{R_{n, 1}} \right\rangle_{\mathbb{L}^2(B_2)}\right|^2\notag\\
  &+& \sum_n \frac{|\eta_1|^2 a^8}{\left|1-k^2 a^2 \eta_2 \lambda_n^{(1)}(B_2)\right|^2}\left|\left\langle \int_{B_1}\int_0^1 \underset{x}{\nabla}(\Phi_k I)(z_2+ta\cdot, z_1)\cdot\mathcal{P}(\cdot, 0)\,dt\cdot\overset{3}{\PP}\left(\tilde{E}_1\right)(y)\,dy, e_{n, B_{2}}^{(1)}\right\rangle_{\mathbb{L}^2(B_2)}\right|^2\notag\\
    &+& |\eta_1|^2 a^8\sum_n \left|\left\langle \int_{B_1}\int_0^1 \underset{x}{\nabla}(\Phi_k I)(z_2+ta\cdot, z_1)\cdot\mathcal{P}(\cdot, 0)\,dt\cdot\overset{3}{\PP}\left(\tilde{E}_1\right)(y)\,dy, \widetilde{R_{n ,1}}\right\rangle_{\mathbb{L}^2(B_2)}\right|^2\notag\\
    && \qquad \qquad \qquad + \, |\eta_1|^2 a^8\sum_n \left|\left\langle \int_{B_1}\int_0^1 \underset{y}{\nabla}(\Phi_k I)(z_2, z_1+tay)\cdot\mathcal{P}(y, 0)\,dt\cdot\overset{3}{\PP}\left(\tilde{E}_1\right)(y)\,dy, \widetilde{R_{n ,1}}\right\rangle_{\mathbb{L}^2(B_2)}\right|^2.
  	\end{eqnarray} 
        To estimate the terms given by $\underset{n}{\sum} \left\vert \left\langle \cdots; \widetilde{R_{n, 1}}\right\rangle_{\mathbb{L}^2(B_2)} \right\vert$, by following similar arguments to that deriving \eqref{sum-Rn1}, and keeping the dominant part of $(\ref{Eq1140})$, we can derive that for any vector field $F(\cdot)$, there holds
	\begin{eqnarray*}
		\sum_n\left|\left\langle F, \widetilde{R_{n, 1}} \right\rangle_{\mathbb{L}^2(B_2)} \right|^2 \, & \lesssim & \, a^8 \, \left\lVert \int_{B_2} \left\vert \cdot - x \right\vert \, \TT_{ka, 2}^{-1}(F)(x)\,dx\right\rVert_{\mathbb{L}^2(B_2)}^2 \, \lesssim \, a^8 \left\lVert \TT_{ka, 2}^{-1}\left(\overset{1}{\PP}(F)\right)\right\rVert_{\mathbb{L}^2(B_2)}^2,
        \end{eqnarray*}
        which, by using $(\ref{Tka2})$, gives us 
        \begin{equation}\label{Eq1314}
		\sum_n\left| \left\langle F, \widetilde{R_{n, 1}} \right\rangle_{\mathbb{L}^2(B_2)} \right|^2 \, \lesssim \,  a^8 \, \sum_n \frac{\left\vert \left\langle F, e_{n, B_{2}}^{(1)} \right\rangle_{\mathbb{L}^2(B_2)}\right\vert^{2}}{\left\vert 1 \, - \, k^2 \, a^2 \, \eta_2 \, \lambda_{n}^{(1)}(B_{2}) \right\vert^2}  \, \overset{(\ref{def-eta})}{=} \, \mathcal{O}\left( a^8 \, \left\lVert F\right\rVert_{\mathbb{L}^2(B_2)}^2\right).
	\end{equation}
In particular, if $F(\cdot)$ is a constant vector, or its first projection, implies that 
\begin{equation*}
\underset{n}{\sum} \left| \left\langle F, \widetilde{R_{n, 1}}\right\rangle_{\mathbb{L}^2(B_2)}\right|^2 \, = \, 0.
\end{equation*}
We are now in a position to estimate the R.H.S. of \eqref{es-P12-1} as follows.
\begin{enumerate}
    \item[]
    \item Estimation of
    \begin{equation*} 
    \mathrm{S}_{1} \, := \, |\eta_1|^2 \sum_n \left|\left\langle \tilde{I_1^0}, e_{n, B_{2}}^{(1)}\right\rangle_{\mathbb{L}^2(B_2)}\right|^2 \; \; \overset{(\ref{def-eta})}{\lesssim} \;\; a^{-4} \, \sum_n \left|\left\langle \tilde{I_1^0}, e_{n, B_{2}}^{(1)}\right\rangle_{\mathbb{L}^2(B_2)}\right|^2.
    \end{equation*}
    By scaling the $I_1^0(\cdot)$'s expression, see $(\ref{def-I10})$, and using the Calder\'{o}n-Zygmund inequality, we obtain 
\begin{equation}\label{P12-es1}
    \mathrm{S}_{1} \,  \lesssim  \,  a^2 \, \left\lVert \int_{B_1}\underset{y}{\nabla} \underset{y}{\nabla}\Phi_0(\cdot, y) \cdot \overset{3}{\PP}\left(\tilde{E}_1\right)(y)\,dy\right\rVert_{\mathbb{L}^2(B_2)}^2\lesssim a^2 \left\lVert \overset{3}{\PP}\left(\tilde{E}_1\right)\right\rVert_{\mathbb{L}^2(B_1)}^2.
\end{equation}
    \item[]
    \item Estimation of 
     \begin{equation*}
     \mathrm{S}_{2} \; := \; |\eta_1|^2 \, \sum_n \left|\left\langle \tilde{I_1^0}, \widetilde{R_{n, 1}}\right\rangle_{\mathbb{L}^2(B_1)}\right|^2. 
     \end{equation*}
From \eqref{Eq1314} and by the use of the Calder\'{o}n-Zygmund inequality, we deduce
\begin{equation}\label{P12-es2}
    \mathrm{S}_{2} \; \lesssim \; a^{10} \; \left\lVert \int_{B_1}\underset{y}{\nabla} \underset{y}{\nabla}\Phi_0(\cdot, y) \cdot \overset{3}{\PP}\left(\tilde{E}_1\right)(y)\,dy\right\rVert_{\mathbb{L}^2(B_2)}^2\lesssim a^{10} \left\lVert\overset{3}{\PP}\left(\tilde{E}_1\right)\right\rVert_{\mathbb{L}^2(B_1)}^2
\end{equation} 
    \item[]
    \item Estimation of $Err\left(\overset{1}{\PP}\left(\tilde{E}_1\right)\right)$. \\  
By using $(\ref{Eq1314})$ and $(\ref{def-eta})$, we can further simplify \eqref{Err-P11-1} as,
	\begin{eqnarray}\label{Err-P11-2}
    \nonumber
		&& Err\left(\overset{1}{\PP}\left(\tilde{E}_1 \right)\right) \\ & \lesssim & a^6 \, \left\lVert \int_{B_1} \int_0^1 \underset{y}{\nabla} \underset{y}{\nabla}\Phi_{k}(z_2+ta\cdot, z_1)\cdot \cdot \,dt\cdot\mathcal{P}(y, 0)\cdot\overset{1}{\PP}\left(\tilde{E}_1 \right)(y)\,dy\right\rVert_{\mathbb{L}^2(B_2)}^2\notag\\
		&+&  a^{14} \, \left\lVert \int_{B_1}\int_0^1 \underset{y}{\nabla} \underset{y}{\nabla}\Phi_{k}(z_2+ta\cdot, z_1)\cdot \cdot \,dt\cdot \mathcal{P}(y, 0)\cdot\overset{1}{\PP}\left(\tilde{E}_1 \right)(y)\,dy \right\rVert_{\mathbb{L}^2(B_2)}^2\notag\\
  &+& a^8 \, \left\lVert \int_{B_1}\int_0^1\int_0^1 (1-t) y^\perp \cdot  \underset{x}{\nabla}(\underset{y}{\nabla}\underset{y}{\nabla} \Phi_k(z_2+sa\cdot, z_1+tay)\cdot y\,dt\cdot\mathcal{P}(\cdot, 0)\,ds\cdot\overset{1}{\PP}\left(\tilde{E}_1\right)(y)\,dy\right\rVert^2_{\mathbb{L}^2(B_2)}\notag\\
  &+& a^{16} \left\lVert \int_{B_1}\int_0^1\int_0^1 (1-t) y^\perp \cdot \underset{x}{\nabla}(\underset{y}{\nabla} \underset{y}{\nabla}\Phi_k(z_2+sa\cdot, z_1+tay)\cdot y\,dt\cdot\mathcal{P}(\cdot, 0)\,ds\cdot\overset{1}{\PP}\left(\tilde{E}_1\right)(y)\,dy\right\rVert^2_{\mathbb{L}^2(B_2)}\notag\\
		& & \qquad = \, \mathcal{O}\left( a^6 \, d^{-6} \, \left\lVert\overset{1}{\PP}\left(\tilde{E}_1 \right)\right\rVert_{\mathbb{L}^2(B_1)}^2\right).
	\end{eqnarray}
    \item[]
    \item Estimate for $Err\left(\overset{3}{\PP}\left(\tilde{E}_1\right)\right)$. \\
    By using $(\ref{Eq1314})$ and $(\ref{def-eta})$, for $h<1$, we can further simplify \eqref{Err-P31-1} as,
	\begin{eqnarray}\label{Err-P31-2}
		Err\left(\overset{3}{\PP}\left(\tilde{E}_1\right)\right)&\lesssim& 
   a^4 \left\lVert \int_{B_1}\int_0^1 \underset{x}{\nabla}\left( \underset{y}{\nabla}\underset{y}{\nabla}\left(\Phi_k-\Phi_0\right)\right)(z_2+ta\cdot, z_1)\cdot\mathcal{P}(\cdot, 0)\,dt\cdot\overset{3}{\PP}\left(\tilde{E}_1\right)(y)\,dy\right\rVert_{\mathbb{L}^2(B_2)}^2\notag\\
  &+& a^{12} \left\lVert \int_{B_1}\int_0^1 \underset{x}{\nabla}\left( \underset{y}{\nabla}\underset{y}{\nabla}\left(\Phi_k-\Phi_0\right)\right)(z_2+ta\cdot, z_1)\cdot\mathcal{P}(\cdot, 0)\,dt\cdot\overset{3}{\PP}\left(\tilde{E}_1\right)(y)\,dy\right\rVert_{\mathbb{L}^2(B_2)}^2\notag\\
  &+& a^4 \left\lVert \int_{B_1}\int_0^1 \underset{x}{\nabla}(\Phi_k I)(z_2+ta\cdot, z_1)\cdot\mathcal{P}(\cdot, 0)\,dt\cdot\overset{3}{\PP}\left(\tilde{E}_1\right)(y)\,dy\right\rVert_{\mathbb{L}^2(B_2)}^2\notag\\
  &+& a^{12} \left\lVert \int_{B_1}\int_0^1 \underset{x}{\nabla}(\Phi_k I)(z_2+ta\cdot, z_1)\cdot\mathcal{P}(\cdot, 0)\,dt\cdot\overset{3}{\PP}\left(\tilde{E}_1\right)(y)\,dy\right\rVert_{\mathbb{L}^2(B_2)}^2\notag\\
  &\lesssim&  a^4 d^{-4} \left\lVert \overset{3}{\PP}\left(\tilde{E}_1\right)\right\rVert_{\mathbb{L}^2(B_1)}^2.
	\end{eqnarray}
\end{enumerate} 
 By returning to $(\ref{es-P12-1})$, using \eqref{P12-es1}, \eqref{P12-es2}, \eqref{Err-P11-2} and \eqref{Err-P31-2}, we deduce that
	\begin{equation}\label{es-P12}
		\left\lVert \overset{1}{\PP}\left(\tilde{E}_2 \right)\right\rVert_{\mathbb{L}^2(B_2)}^2 \; \lesssim \; a^2 \; + \; a^6 \; d^{-6} \; \left\lVert \overset{1}{\PP}\left(\tilde{E}_1 \right)\right\rVert_{\mathbb{L}^2(B_1)}^2 \; + \; (a^2 \;+\; a^4 \; d^{-4} ) \; \left\lVert\overset{3}{\PP}\left(\tilde{E}_1\right)\right\rVert_{\mathbb{L}^2(B_1)}^2.
	\end{equation}
	\item Estimation of $ \left\Vert  \overset{3}{\mathbb{P}}\left(\tilde{E}_{1} \right) \right\Vert^{2}_{\mathbb{L}^2(B_1)}$. \\
 By using $(\ref{grad-M-1st})$, we rewrite $(\ref{*add3})$ as 
	\begin{eqnarray}\label{LS-P31-1}
        \nonumber
		\left(I \, + \, \eta_1 \, \nabla M^{k}_{D_{1}} \right)\left( \overset{3}{\PP}\left(E_1 \right)\right)(x) \, &=& \, E_{1}^{Inc}(x) \, - \, \overset{1}{\PP}\left( E_1 \right)(x) \, + \, k^2 \, \eta_2 \, N^{k}_{D_{2}}\left(\overset{1}{\PP}\left(E_2\right)\right)(x) \\ \nonumber &+& \, \eta_2 \, \left( -\nabla M^k+ k^2 N^k\right)\left(\overset{3}{\PP}\left(E_2\right)\right)(x) \\ &+& k^2 \eta_1 N^{k}_{D_{1}}\left(\overset{1}{\PP}\left(E_1\right)\right)(x) \, + \, k^2 \, \eta_1 \, N^{k}_{D_{1}} \left(\overset{3}{\PP}\left(E_1 \right)\right)(x), \quad x \in D_{1}.
	\end{eqnarray}
	For the third term on the R.H.S of \eqref{LS-P31-1}, by using Taylor expansions and the relation $(\ref{RefNeeded1})$, we have 
 \begin{eqnarray}\label{expan-phi1}
     N^{k}_{D_{2}}\left(\overset{1}{\PP}\left(E_2\right)\right)(x) \,
     & = & \, \underset{x}{\nabla}(\Phi_k I)(z_1, z_2) \cdot \int_{D_2}\mathcal{P}(y, z_2)\cdot\overset{1}{\PP}(E_2)(y)\,dy\notag\\
     & + & \int_{D_2}\int_0^1 (1-t) (y-z_2)^\perp \cdot \underset{y}{\nabla}\underset{y}{\nabla}  \Phi_k(z_1, z_2+t(y-z_2))\cdot (y-z_2)\,dt\cdot\overset{1}{\PP}(E_2)(y)\,dy\notag\\
     & - & \int_{D_2}\int_0^1 \underset{y}{\nabla}\underset{y}{\nabla} \Phi_k (z_1+ t(x-z_1), z_2)\cdot (y-z_2)\cdot (x-z_1)\,dt\cdot\overset{1}{\PP}(E_2)(y)\,dy\notag\\ \nonumber
     & + & \int_{D_2}\int_0^1 \int_0^1 (1-t) (y-z_2)^\perp \cdot \underset{x}{\nabla}(\underset{y}{\nabla} \underset{y}{\nabla} \Phi_k)(z_1+s(\cdot-z_1), z_2+t(y-z_2))\cdot (y-z_2)\,dt \\ && \qquad \qquad \qquad \cdot (\cdot - z_1)\,ds\cdot \overset{1}{\PP}(E_2)(y)\,dy.
 \end{eqnarray} 
	Regarding the fourth term on the R.H.S of \eqref{LS-P31-1}, similar to \eqref{LS-P11-taylor}, by Taylor expansion, we have 
	\begin{eqnarray}\label{expan-upsilon1}
    \nonumber
    &&
 \eta_2 \, \left(-\nabla M^k + k^2 N^k\right)\left(\overset{3}{\PP}(E_2)\right)(x) \\ \nonumber
 &=& \eta_2 \, I_2^0(x) \, + \, \eta_2 \, \underset{y}{\nabla} \underset{y}{\nabla}\left(\Phi_k-\Phi_0\right)(z_1, z_2) \cdot \int_{D_2}\overset{3}{\PP}(E_2)(y)\,dy \\ &+& \eta_2 k^2 \Phi_k (z_1, z_2)\int_{D_2}\overset{3}{\PP}(E_2)(y)\,dy\notag\\
 &+& \eta_2 \int_{D_2}\int_0^1 \underset{x}{\nabla}\left(\underset{y}{\nabla} \underset{y}{\nabla}\left(\Phi_k-\Phi_0\right)\right)(z_1+t(\cdot-z_1), z_2)\cdot\mathcal{P}(\cdot, z_1)\,dt\cdot \overset{3}{\PP}(E_2)(y)\,dy\notag\\
 &+&\eta_2 \int_{D_2}\int_0^1 \underset{y}{\nabla}\left(\underset{y}{\nabla}\underset{y}{\nabla}\left(\Phi_k-\Phi_0\right)\right)(z_1, z_2+t(y-z_2))\cdot\mathcal{P}(y, z_2)\,dt\cdot \overset{3}{\PP}(E_2)(y)\,dy\notag\\
 &+& \eta_2 k^2 \int_{D_2}\int_0^1 \underset{x}{\nabla}(\Phi_k I)(z_1+t(\cdot-z_1), z_2)\cdot\mathcal{P}(\cdot, z_1)\,dt\cdot\overset{3}{\PP}(E_2)(y)\,dy\notag\\
 &+& \eta_2 k^2 \int_{D_2}\int_0^1 \underset{y}{\nabla}(\Phi_k I)(z_1, z_2+t(y-z_2))\cdot\mathcal{P}(y, z_2)\,dt\cdot\overset{3}{\PP}(E_2)(y)\,dy
	\end{eqnarray} 
 where $I_2^0(\cdot)$ is defined by \eqref{def-I-0-2}. By plugging \eqref{expan-phi1} and \eqref{expan-upsilon1} into \eqref{LS-P31-1}, scaling from $D_{m}$  to $B_{m}$, with $m = 1, 2$, and taking the $\mathbb{L}^2(B_1)$-inner product with respect to $e_{n, B_{1}}^{(3)}(\cdot)$, we obtain
	\begin{equation}\label{eq-LS-P31}
		\left\langle \overset{3}{\PP}\left(\tilde{E}_1\right),\, e_{n, B_{1}}^{(3)} \right\rangle_{\mathbb{L}^2(B_1)}\, = \, \frac{ \left\langle \tilde{E}^{Inc}_1, e_{n, B_{1}}^{(3)}\right\rangle_{\mathbb{L}^2(B_1)} \, + \,\eta_2 \left\langle 
\tilde{I_2^0}, e_{n, B_{1}}^{(3)} \right\rangle_{\mathbb{L}^2(B_1)} \, + \, Error_n^1 }{\left(1 \, + \, \eta_1 \, \lambda_{n}^{(3)}(B_{1}) \right)},
	\end{equation} 
	where, by using $(\ref{expansion-gradMk})$ and $(\ref{condition-on-k})$, the expression of $Error_n^1$ is given by 
 \begin{eqnarray}\label{Error-n1}
 \nonumber
	&&	Error_n^1 
  \\ \nonumber &:=& - \, \frac{1 \, \mp \, c_0 \, a^h}{2 \, \lambda_{n_{0}}^{(1)}(B_{1})}\left\langle N\left(\overset{3}{\PP}\left(\tilde{E}_1 \right)\right), e_{n, B_{1}}^{(3)} \right\rangle_{\mathbb{L}^2(B_1)} \, - \, \eta_1\frac{i k^3 a^3}{12\pi}\left\langle \int_{B_1}\overset{3}{\PP}\left(\tilde{E}_1 \right)(y)\,dy, e_{n, B_{1}}^{(3)} \right\rangle_{\mathbb{L}^2(B_1)} \\ \nonumber &+& k^2 \,  \eta_1 \, a^2 \, \left\langle N^{ka}\left(\overset{3}{\PP}\left( \tilde{E}_1 \right)\right), e_{n, B_{1}}^{(3)} \right\rangle_{\mathbb{L}^2(B_1)} \\ \nonumber &+&  \frac{1\mp c_0 a^h}{2\lambda_{n_{0}}^{(1)}(B_{1})}\left\langle \int_{B_1} \Phi_{0}(\cdot, y)\frac{A(\cdot, y) \cdot \overset{3}{\PP}\left(\tilde{E}_1\right)(y)}{\lvert \cdot-y\rvert^2}\,dy, e_{n, B_{1}}^{(3)} \right\rangle_{\mathbb{L}^2(B_1)} \\ &+& \frac{\eta_1}{4\pi}\sum_{\ell\geq 3} ( i k a)^{\ell+1}\left\langle \int_{B_1}\frac{\nabla \nabla \left(\lvert \cdot - y \rvert^\ell\right)}{(\ell+1)!} \cdot \overset{3}{\PP}\left(\tilde{E}_1\right)(y)\,dy, e_{n, B_{1}}^{(3)} \right\rangle_{\mathbb{L}^2(B_1)} \notag\\
		&+& \frac{1\mp c_0 a^h}{4\pi \lambda_{n_{0}}^{(1)}(B_{1})} \sum_{\ell \geq 1} (i k a)^{\ell+1}\left\langle \int_{B_1}\frac{\lvert \cdot - y \rvert^\ell}{(\ell+1)!} \overset{1}{\PP}\left(\tilde{E}_1 \right)(y)\,dy, e_{n, B_{1}}^{(3)} \right\rangle_{\mathbb{L}^2(B_1)} \notag \\
        &+& k^2 \eta_2 a^4 \left\langle 
\underset{y}{\nabla}(\Phi_k I)(z_1, z_2) \cdot \int_{B_2} \mathcal{P}(y, 0)\cdot\overset{1}{\PP}\left( \tilde{E}_2 \right)(y)\,dy, e_{n, B_{1}}^{(3)} \right\rangle_{\mathbb{L}^2(B_1)} \notag\\
&+& k^2 \, \eta_2 \, a^5 \, \left\langle \int_{B_2}\int_0^1 (1-t) y^\perp \cdot \underset{y}{\nabla} \underset{y}{\nabla} \Phi_k (z_1, z_2+tay)\cdot y\,dt\cdot\overset{1}{\PP}\left(\tilde{E}_2\right)(y)\,dy, e_{n, B_{1}}^{(3)} \right\rangle_{\mathbb{L}^2(B_1)}\notag\\
&-& k^2 \, \eta_2 \, a^5 \, \left\langle \int_{B_2}\int_0^1 \underset{y}{\nabla} \underset{y}{\nabla} \Phi_{k}(z_1+ta\cdot, z_2)\cdot y \cdot \cdot \,dt\cdot\overset{1}{\PP}\left(\tilde{E}_2\right)(y)\,dy, e_{n, B_{1}}^{(3)}\right\rangle_{\mathbb{L}^2(B_1)} \notag\\
  &+& k^2 \, \eta_2 \, a^6 \, \left\langle 
\int_{B_2}\int_0^1 \int_0^1 (1-t)y^\perp \cdot \underset{x}{\nabla}(\underset{y}{\nabla} \underset{y}{\nabla} \Phi_k(z_1+sa\cdot, z_2+tay)\cdot y\,dt \cdot \cdot \,ds\cdot\overset{1}{\PP}\left(\tilde{E}_2\right)(y)\,dy, e_{n, B_{1}}^{(3)} \right\rangle_{\mathbb{L}^2(B_1)} \notag\\ \nonumber
&+& \eta_2 a^3 \left\langle \underset{y}{\nabla} \underset{y}{\nabla}\left(\Phi_k-\Phi_0\right)(z_1, z_2) \cdot \int_{B_2}\overset{3}{\PP}\left(\tilde{E}_2\right)(y)\,dy, e_{n, B_{1}}^{(3)}\right\rangle_{\mathbb{L}^2(B_1)} \\ \nonumber &+& \, \eta_2 k^2 a^3 \left\langle \Phi_k(z_1, z_2)\int_{B_2} \overset{3}{\PP}\left(\tilde{E}_2\right)(y)\,dy, e_{n, B_{1}}^{(3)}\right\rangle_{\mathbb{L}^2(B_1)}\notag\\
&+& \eta_2 \, a^4 \, \left\langle \int_{B_2} \int_0^1 \underset{x}{\nabla}\left(\underset{y}{\nabla} \underset{y}{\nabla}\left(\Phi_k-\Phi_0\right)\right)(z_1+ta\cdot, z_2)\cdot\mathcal{P}(\cdot, 0)\,dt \cdot\overset{3}{\PP}\left(\tilde{E}_2\right)(y)\,dy, e_{n, B_{1}}^{(3)}\right\rangle_{\mathbb{L}^2(B_1)} \notag\\
&+& \eta_2 a^4 \left\langle \int_{B_2} \int_0^1 \underset{y}{\nabla}\left(\underset{y}{\nabla}\underset{y}{\nabla}\left(\Phi_k-\Phi_0\right)\right)(z_1, z_2+tay)\cdot\mathcal{P}(y, 0)\,dt \cdot\overset{3}{\PP}\left(\tilde{E}_2\right)(y)\,dy, e_{n, B_{1}}^{(3)}\right\rangle_{\mathbb{L}^2(B_1)} \notag\\
&+&  \eta_2 k^2 a^4 \left\langle \int_{B_2}\int_0^1 \underset{x}{\nabla}(\Phi_k I)(z_1+ta\cdot, z_2)\cdot\mathcal{P}(\cdot, 0)\,dt\cdot\overset{3}{\PP}\left(\tilde{E}_2\right)(y)\,dy, e_{n, B_{1}}^{(3)}\right\rangle_{\mathbb{L}^2(B_1)} \notag\\
&& \qquad \qquad \, + \, \eta_2 k^2 a^4 \left\langle \int_{B_2}\int_0^1 \underset{y}{\nabla}(\Phi_k I)(z_1, z_2+tay)\cdot\mathcal{P}(y, 0)\,dt\cdot\overset{3}{\PP}\left(\tilde{E}_2\right)(y)\,dy, e_{n, B_{1}}^{(3)}\right\rangle_{\mathbb{L}^2(B_1)}.
	\end{eqnarray}
	Now, by taking the square modulus on the both sides of \eqref{eq-LS-P31} and summing up with respect to the index $n$, we can obtain that 
	\begin{equation}\label{es-P31-1}
	\left\lVert \overset{3}{\PP}\left(\tilde{E}_1\right)\right\rVert_{\mathbb{L}^2(B_1)}^2 \lesssim a^{4} \, \left\lVert \tilde{E}_1^{Inc}\right\rVert_{\mathbb{L}^2(B_1)}^2 \, + \, a^4 \, \sum_n \left|\left\langle \tilde{I_2^0}, e_{n, B_{1}}^{(3)}\right\rangle_{\mathbb{L}^2(B_1)}\right|^2 \, + \, a^{4} \, \sum_n\left| Error_n^1\right|^2.
	\end{equation}
It is direct to derive from \eqref{def-I-0-2} that
\begin{equation}\label{P31-es1}
    \sum_n \left|\left\langle \tilde{I_2^0}, e_{n, B_{1}}^{(3)}\right\rangle\right|^2 \, 
     \lesssim \, a^6 \left\lVert \int_{B_2} \underset{y}{\nabla} \underset{y}{\nabla}\Phi_0(\cdot, y) \cdot \overset{3}{\PP}\left(\tilde{E}_2\right)(y)\,dy\right\rVert_{\mathbb{L}^2(B_1)}\lesssim a^6 \left\lVert \overset{3}{\PP}\left(\tilde{E}_2\right)\right\rVert_{\mathbb{L}^2(B_2)}
\end{equation}
In order to estimate $\underset{n}{\sum} \left\vert Error_n^1 \right\vert^{2}$, we return to $(\ref{Error-n1})$. With the help of the continuity of the Newtonian operator, by keeping the dominant terms, we can obtain that
        \begin{equation}\label{es-Error-n1}
            \sum_n\left| Error_n^{1}\right|^2 \, \lesssim \, \left\lVert\overset{3}{\PP}\left(\tilde{E}_1 \right)\right\rVert_{\mathbb{L}^2(B_1)}^2+a^4 \left\lVert\overset{1}{\PP}\left(\tilde{E}_1 \right)\right\rVert_{\mathbb{L}^2(B_1)}^2 \, + \, a^{8} \, d^{-4} \left\lVert\overset{1}{\PP}\left(\tilde{E}_2 \right)\right\rVert_{\mathbb{L}^2(B_2)}^2 \, + \, a^6 \, d^{-2} \left\lVert\overset{3}{\PP}\left(\tilde{E}_2 \right)\right\rVert_{\mathbb{L}^2(B_2)}^2. 
        \end{equation}
        Hence, by plugging \eqref{P31-es1} and \eqref{es-Error-n1} into $(\ref{es-P31-1})$, and using the fact that $d \, \sim \, a^{t}$, with $t \, < \, 1$, we deduce
	\begin{equation}\label{es-P31}
		\left\lVert\overset{3}{\PP}\left(\tilde{E}_1 \right)\right\rVert_{\mathbb{L}^2(B_1)}^2 \, \lesssim \, a^4 \, + \, a^{10} \, d^{-2} \, \left\lVert \overset{3}{\PP}\left(\tilde{E}_2\right)\right\rVert_{\mathbb{L}^2(B_2)}^2 \, + \, a^8 \, \left\lVert\overset{1}{\PP}\left(\tilde{E}_1\right)\right\rVert_{\mathbb{L}^2(B_1)}^2 \, + \, a^{12} \, d^{-4} \, \left\lVert \overset{1}{\PP}\left(\tilde{E}_2\right)\right\rVert_{\mathbb{L}^2(B_2)}^2.
	\end{equation}
	\item Estimation of $ \left\Vert  \overset{3}{\mathbb{P}}\left(\tilde{E}_{2} \right) \right\Vert^{2}_{\mathbb{L}^2(B_2)}$. \\
 For $x\in D_2$, by recalling the L.S.E \eqref{LS-es} and using $(\ref{grad-M-1st})$, we get
	\begin{eqnarray*}
		\left(I \, + \, \eta_2 \, \nabla M^{k}_{D_{2}}\right)\left(\overset{3}{\PP}\left(E_2\right)\right)(x) \, &=& \, E_2^{Inc}(x) \, - \, \overset{1}{\PP}\left(E_2\right)(x) \, + \, \eta_1 \left(-\nabla M^k + k^2 N^k\right)\left(\overset{3}{\PP}(E_1)\right)(x) \\ &+& k^2 \, \eta_1 \, \int_{D_{1}}\Phi_{k}(x, y) \, \overset{1}{\PP}\left(E_1\right)(y)\,dy\notag\\
		&+& k^2 \eta_2 N^{k}_{D_{2}}\left(\overset{1}{\PP}\left(E_2\right)\right)(x) \, + \, k^2 \, \eta_2 \,  N^{k}_{D_{2}}\left(\overset{3}{\PP}\left(E_2\right)\right)(x), \quad x \in D_{2}.
	\end{eqnarray*}
	Then, following the similar argument to \eqref{eq-LS-P31}, scaling from $D_{m}$ to $B_{m}$, with $m=1,2$, and taking the $\mathbb{L}^2(B_2)$-inner product with respect to $e_{n, B_{2}}^{(3)}(\cdot)$, we deduce
	\begin{equation}\label{P32-product}
		\left\langle \overset{3}{\PP}\left(\tilde{E}_2\right) , e_{n, B_{2}}^{(3)} \right\rangle_{\mathbb{L}^2(B_2)} \; = \; \frac{\left\langle \tilde{E}_2^{Inc}, e_{n, B_{2}}^{(3)} \right\rangle_{\mathbb{L}^2(B_2)} \, + \, \eta_1 \, \left\langle \tilde{I_1^0}, e_{n, B_{2}}^{(3)}\right\rangle_{\mathbb{L}^2(B_2)}
  \, + \, Error_n^2 }{\left( 1 \, + \, \eta_2 \, \lambda_{n}^{(3)}(B_{2})\right)},
	\end{equation}
	where 
	\begin{eqnarray*}
    \nonumber
    && 
	Error_n^2 \\ \nonumber &:=& \, - \, \eta_2 \, \left\langle \left(\nabla M^{ka}_{B_{2}} \, - \, \nabla M_{B_{2}} \right)\left(\overset{3}{\PP}\left(\tilde{E}_2\right)\right), e_{n, B_{2}}^{(3)} \right\rangle_{\mathbb{L}^2(B_2)} \, + \, k^2 \, a^2 \,  \eta_2\left\langle N^{ka}_{B_{2}}\left(\overset{1}{\PP}\left(\tilde{E}_2 \right)\right), e_{n, B_{2}}^{(3)} \right\rangle_{\mathbb{L}^2(B_2)} \\ &+& \,  k^2 a^2 \eta_2 \, \left\langle N^{ka}_{B_{2}}\left(\overset{3}{\PP}\left( \tilde{E}_2 \right)\right), e_{n, B_{2}}^{(3)} \right\rangle_{\mathbb{L}^2(B_2)} \\
 &+& k^2 \, \eta_1\, a^4\, \left\langle 
 \underset{y}{\nabla}(\Phi_k I)(z_2, z_1) \cdot \int_{B_1} \mathcal{P}(y, 0)\cdot \overset{1}{\PP}\left( \tilde{E}_1 \right)(y)\,dy, e_{n, B_{2}}^{(3)}\right\rangle_{\mathbb{L}^2(B_2)} \notag\\
 &+& k^2\, \eta_1\, a^5\, \left\langle \int_{B_1}\int_0^1 (1-t) \, y^\perp \cdot \underset{y}{\nabla} \underset{y}{\nabla} \Phi_k (z_2, z_1+tay)\cdot y\,dt\cdot\overset{1}{\PP}\left( \tilde{E}_1 \right)(y)\,dy, e_{n, B_{2}}^{(3)} \right\rangle_{\mathbb{L}^2(B_2)} \notag\\
 &-& k^2\, \eta_1\, a^5\, \left\langle \int_{B_1}\int_0^1 \underset{y}{\nabla} \underset{y}{\nabla} \Phi_k (z_2+ta\cdot, z_1) \cdot y\cdot \cdot \,dt\cdot\overset{1}{\PP}\left( \tilde{E}_1\right)(y)\,dy, e_{n, B_{2}}^{(3)}\right\rangle_{\mathbb{L}^2(B_2)} \notag\\
 &+& k^2\, \eta_1\, a^6\, \left\langle \int_{B_1}\int_0^1\int_0^1 (1-t)y^\perp \cdot \underset{x}{\nabla}(\underset{y}{\nabla} \underset{y}{\nabla} \Phi_k)(z_2+sa\cdot, z_1+tay)\cdot y\,dt\cdot \cdot \,ds\cdot\overset{1}{\PP}\left(\tilde{E}_1\right)(y)\,dy, e_{n, B_{2}}^{(3)}\right\rangle_{\mathbb{L}^2(B_2)} \notag\\ \nonumber
 &+&  \eta_1 a^3 \left\langle \underset{y}{\nabla} \underset{y}{\nabla}(\Phi_k-\Phi_0)(z_2, z_1)\cdot\int_{B_1}\overset{3}{\PP}\left(\tilde{E}_1\right)(y)\,dy, e_{n, B_{2}}^{(3)}\right\rangle_{\mathbb{L}^2(B_2)} \\ &+& k^2 \eta_1 a^3 \left\langle\Phi_k(z_2, z_1)\int_{B_1} \overset{3}{\PP}\left(\tilde{E}_1\right)(y)\,dy, e_{n, B_{2}}^{(3)} \right\rangle_{\mathbb{L}^2(B_2)} \notag\\
 &+& \eta_1 a^4 \left\langle \int_{B_1}\int_0^1 \underset{x}{\nabla}\left(\underset{y}{\nabla}\underset{y}{\nabla}(\Phi_k-\Phi_0)\right)(z_2+ta\cdot, z_1)\cdot\mathcal{P}(\cdot, 0)\,dt\cdot\overset{3}{\PP}\left(\tilde{E}_1\right)(y)\,dy, e_{n, B_{2}}^{(3)}\right\rangle_{\mathbb{L}^2(B_2)} \notag\\
 &+&  \eta_1 a^4 \left\langle \int_{B_1}\int_0^1 \underset{y}{\nabla}\left(\underset{y}{\nabla}\underset{y}{\nabla}(\Phi_k-\Phi_0)\right)(z_2, z_1+tay)\cdot\mathcal{P}(y, 0)\,dt\cdot\overset{3}{\PP}\left(\tilde{E}_1\right)(y)\,dy, e_{n, B_{2}}^{(3)}\right\rangle_{\mathbb{L}^2(B_2)} \notag\\
 &+& k^2 \eta_1 a^4 \left\langle \int_{B_1}\int_0^1 \underset{x}{\nabla}(\Phi_k I)(z_2+ta\cdot, z_1)\cdot\mathcal{P}(x, 0)\,dt\cdot\overset{3}{\PP}\left(\tilde{E}_1\right)(y)\,dy, e_{n, B_{2}}^{(3)}\right\rangle_{\mathbb{L}^2(B_2)} \notag\\
 &+& k^2 \eta_1 a^4 \left\langle \int_{B_1}\int_0^1 \underset{y}{\nabla}(\Phi_k I)(z_2, z_1+tay)\cdot\mathcal{P}(y, 0)\,dt\cdot\overset{3}{\PP}\left(\tilde{E}_1\right)(y)\,dy, e_{n, B_{2}}^{(3)}\right\rangle_{\mathbb{L}^2(B_2)}.
	\end{eqnarray*}
	By taking the square modulus of \eqref{P32-product},  summing up with respect to the index $n$, and using $(\ref{def-eta})$ and $(\ref{condition-on-k})$, we have
	\begin{equation*}\label{es-P32-1}
		\left\lVert\overset{3}{\PP}\left(\tilde{E}_2 \right)\right\rVert_{\mathbb{L}^2(B_2)}^2 \, \lesssim \, a^{-2h} \; \left(  \left\Vert  \tilde{E}_2^{Inc}  \right\Vert^{2}_{\mathbb{L}^2(B_2)} \, + \, |\eta_1|^2\sum_n\left|\left\langle \tilde{I_1^0}, e_{n, B_{2}}^{(3)}\right\rangle_{\mathbb{L}^2(B_2)}\right|^2 \, + \, \sum_n\left|  Error_n^2  \right|^2 \right).
	\end{equation*}
Similar to \eqref{P31-es1}, we can know that
\begin{equation}\label{P32-es1}
    |\eta_1|^2 \sum_n\left|\left\langle \tilde{I_1^0}, e_{n, B_{2}}^{(3)}\right\rangle_{\mathbb{L}^2(B_2)} \right|^2 \lesssim a^2\left\lVert \overset{3}{\PP}\left(\tilde{E}_1\right)\right\rVert_{\mathbb{L}^2(B_1)}^2.
\end{equation} 
To evaluate $\underset{n}{\sum} \, \left| Error_n^2 \right|^2$, we use both the expansions given by $(\ref{expansion-gradMk})$ and $(\ref{expansion-Nk})$, and the continuity of the Newtonian operator, to derive 
    \begin{eqnarray}\label{P32-es2}
    \nonumber
        \sum_{n} \, \left\vert Error_n^2 \right\vert^{2} \, & \lesssim & \, a^4 \left\lVert \overset{3}{\PP}\left(\tilde{E}_2\right)\right\rVert_{\mathbb{L}^2(B_2)}^2+ a^8\left\lVert \overset{1}{\PP}\left(\tilde{E}_2\right)\right\rVert_{\mathbb{L}^2(B_2)}^2 \\ &+& \, a^2 \, d^{-2} \, \left\lVert \overset{3}{\PP}\left(\tilde{E}_1\right)\right\rVert_{\mathbb{L}^2(B_1)}^2 \, + \, a^4 \, d^{-4} \, \left\lVert \overset{1}{\PP}\left(\tilde{E}_1\right)\right\rVert_{\mathbb{L}^2(B_1)}^2.
    \end{eqnarray}
        Hence, combining with \eqref{P32-es1} and \eqref{P32-es2}, using the fact that $h \in (0,2)$ and $t \in (0,1)$, we obtain
		\begin{equation}\label{es-P32}
		\left\lVert \overset{3}{\PP}\left(\tilde{E}_2 \right)\right\rVert_{\mathbb{L}^2(B_2)}^2\lesssim a^{-2h} \, + \, a^{2-2h} \, d^{-2} \left\lVert\overset{3}{\PP}\left(\tilde{E}_1 \right)\right\rVert_{\mathbb{L}^2(B_1)}^2  \, + \, a^{4-2h} \, d^{-4} \, \left\lVert\overset{1}{\PP}\left(\tilde{E}_1 \right)\right\rVert_{\mathbb{L}^2(B_1)}^2 \, +\, a^{8-2h} \, \left\lVert \overset{1}{\PP}\left(\tilde{E}_2\right)\right\rVert_{\mathbb{L}^(B_2)}^2.
	\end{equation}
\end{enumerate}
Recall the estimates obtained in \eqref{es-P11}, \eqref{es-P12}, \eqref{es-P31} and \eqref{es-P32}. By injecting \eqref{es-P32} and \eqref{es-P12} into \eqref{es-P11}, and using the fact that $t < 1$ and $h < 1$, we obtain 
\begin{equation}\label{es-P11P31}
     \left\lVert \overset{1}{\PP}\left(\tilde{E}_1\right)\right\rVert_{\mathbb{L}^2(B_1)}^2 \lesssim a^{2-2h} + \left(a^{8-4h} d^{-2}+ a^{10-4h} d^{-6}\right)\left\lVert \overset{3}{\PP}\left(\tilde{E}_1\right)\right\rVert_{\mathbb{L}^2(B_1)}^2.
\end{equation}
 Similarly, by substituting \eqref{es-P12} and \eqref{es-P32} into \eqref{es-P31},  we also have 
	\begin{equation}\label{es-P31P11}
		\left\lVert\overset{3}{\PP}\left(\tilde{E}_1\right)\right\rVert_{\mathbb{L}^2(B_1)}^2 \, \lesssim \, a^4 \, + \,  \left(a^{14-2h} \, d^{-6} \, + \, a^8\right)\left\lVert\overset{1}{\PP}\left(\tilde{E}_1\right)\right\rVert_{\mathbb{L}^2(B_1)}^2.
	\end{equation} 
 Together with \eqref{es-P11P31} and \eqref{es-P31P11}, for $h<1$, it yields that
 \begin{equation}\label{es-P11P31-final}
     \left\lVert\overset{1}{\PP}\left(\tilde{E}_1 \right)\right\rVert_{\mathbb{L}^2(B_1)}^2\lesssim a^{2-2h}\quad \mbox{and}\quad\left\lVert\overset{3}{\PP}\left(\tilde{E}_1 \right)\right\rVert_{\mathbb{L}^2(B_1)}^2\lesssim a^4.
 \end{equation}
 Then, by taking the estimates \eqref{es-P11P31-final} into \eqref{es-P12} and \eqref{es-P32}, we can respectively derive that
 \begin{equation*}
     \left\lVert\overset{1}{\PP}\left(\tilde{E}_2 \right)\right\rVert_{\mathbb{L}^2(B_1)}^2\lesssim a^2 + a^{8-2h} d^{-6}\quad \mbox{and}\quad \left\lVert\overset{3}{\PP}\left(\tilde{E}_2 \right)\right\rVert_{\mathbb{L}^2(B_1)}^2\lesssim a^{-2h}.
 \end{equation*} 
   This concludes the proof of Proposition \ref{lem-es-multi}.
	 \subsection{Proof of Proposition \ref{prop-scoeff}.}{(Estimation of the scattering coefficients).}\label{subsec:scattering parameter} The proof of the estimates in $(\ref{*add0})$ regarding $W_{1}$ can be directly seen in \cite[Proposition 2.4]{CGS}. Our main goal is to derive estimates $(\ref{prop-es-W})$ and $(\ref{es-V12})$ for $W_2$ and $V_1, V_2$.
	 
	 \begin{enumerate}
	 	\item[] 
            \item Estimation of $\left\Vert  \overset{j}{\mathbb{P}}\left(\tilde{W}_{2} \right) \right\Vert^{2}_{\mathbb{L}^2(B_2)}$, for $j=1, 2, 3$. \\ 
	 	By recalling the definition of $W_2$, given by $(\ref{DefW2})$, after scaling from $D_{2}$ to $B_{2}$,  we obtain
        \begin{equation}\label{ZM-Eq-1}
            \left(I + \eta_2 \, \nabla M^{-k \, a}_{B_{2}} \, - \, k^{2} \, a^{2} \, \eta_2 \, N^{-k \, a}_{B_{2}} \right)\left( \tilde{W}_2 \right)(x) \, = \, a \,  \overset{1}{\PP}\left(\mathcal{P}(x,0)\right), \quad x \in B_{2}.
        \end{equation}
        Now, by taking the inner product with respect to $e_{n, B_{2}}^{(2)}(\cdot)$, we obtain 
        \begin{equation*}
            \left\langle  \tilde{W}_2, \left(I + \eta_2 \, \nabla M^{k \, a}_{B_{2}} \, - \, k^{2} \, a^{2} \, \eta_2 \, N^{k \, a}_{B_{2}} \right)\left(e_{n, B_{2}}^{(2)}\right) \right\rangle_{\mathbb{L}^{2}(B_{2})} \, = \, 0, 
        \end{equation*}
        which, by using $(\ref{grad-M-2nd})$, can be reduced to 
        \begin{equation*}
            \left(1 \, + \, \eta_{2} \right) \; \langle  \tilde{W}_2, e_{n, B_{2}}^{(2)} \rangle_{\mathbb{L}^{2}(B_{2})} \, = \, 0.  
        \end{equation*}
        Thus, 
        \begin{equation}\label{ZM-Eq-2}
            \overset{2}{\PP}\left(\tilde{W}_2\right) \; = \; 0. 
        \end{equation}
       This implies, by gathering $(\ref{ZM-Eq-1}), 
 (\ref{ZM-Eq-2})$, and using $(\ref{L2-decomposition})$,
	 	\begin{equation}\label{LS-W2-1}
                    \left(I + \eta_2 \, \nabla M^{-k \, a}_{B_{2}} \, - \, k^{2} \, a^{2} \, \eta_2 \, N^{-k \, a}_{B_{2}} \right)\left(\overset{1}{\PP}\left(\tilde{W}_2\right) \, + \, \overset{3}{\PP}\left(\tilde{W}_2\right)\right)(x) \, = \, a \, \overset{1}{\PP}\left(\mathcal{P}(x, 0)\right), \quad x\in B_2.
	 	\end{equation}
        Now, by taking the inner product with respect to $e_{n, B_{2}}^{(1)}(\cdot)$ and using $(\ref{grad-M-1st})$, we obtain from \eqref{LS-W2-1} that
        \begin{equation*}
                    \left\langle \left(I + \eta_2 \, \nabla M^{-k \, a}_{B_{2}} \, - \, k^{2} \, a^{2} \, \eta_2 \, N^{-k \, a}_{B_{2}} \right)\left( \overset{1}{\PP}\left(\tilde{W}_2\right) \, + \, \overset{3}{\PP}\left(\tilde{W}_2\right)\right), e_{n, B_{2}}^{(1)} \right\rangle_{\mathbb{L}^{2}(B_{2})} \, = \, a \, \langle \mathcal{P}(\cdot, 0), e_{n, B_{2}}^{(1)}\rangle_{\mathbb{L}^{2}(B_{2})}.
	 	\end{equation*}
        With the adjoint operator of $\left(I + \eta_2 \, \nabla M^{k \, a}_{B_{2}} \, - \, k^{2} \, a^{2} \, \eta_2 \, N^{k \, a}_{B_{2}} \right)$ and using $(\ref{grad-M-1st})$ and $(\ref{L2-decomposition})$, we derive from the above equation the following relation
        \begin{equation}\label{Equa1036}
            \left\langle  \overset{1}{\PP}\left(\tilde{W}_2\right), e_{n, B_{2}}^{(1)} \right\rangle_{\mathbb{L}^{2}(B_{2})} \,= \, a \, \langle \mathcal{P}(\cdot, 0), e_{n, B_{2}}^{(1)}\rangle_{\mathbb{L}^{2}(B_{2})} \, + \, k^{2} \, a^{2} \, \eta_2 \,            \left\langle N^{- \, k \, a}_{B_{2}}\left( \overset{1}{\PP}\left(\tilde{W}_2\right) \, + \, \overset{3}{\PP}\left(\tilde{W}_2\right)\right), e_{n, B_{2}}^{(1)} \right\rangle_{\mathbb{L}^{2}(B_{2})}.
	 	\end{equation}
        By taking the square modulus on the both sides of \eqref{Equa1036} and using $(\ref{def-eta})$, we get  \begin{eqnarray*}
        \nonumber
            \left\vert \left\langle  \overset{1}{\PP}\left(\tilde{W}_2\right), e_{n, B_{2}}^{(1)} \right\rangle_{\mathbb{L}^{2}(B_{2})} \right\vert^{2} \, & \lesssim &  \, a^{2} \, \left\vert \langle \mathcal{P}(\cdot, 0), e_{n, B_{2}}^{(1)}\rangle_{\mathbb{L}^{2}(B_{2})} \right\vert^{2} \\ &+& \, a^{4} \, \left\vert           \left\langle N^{- \, k \, a}_{B_{2}}\left( \overset{1}{\PP}\left(\tilde{W}_2\right) \, + \, \overset{3}{\PP}\left(\tilde{W}_2\right)\right), e_{n, B_{2}}^{(1)} \right\rangle_{\mathbb{L}^{2}(B_{2})} \right\vert^{2}.
	 	\end{eqnarray*}
        Now, by summing up with respect to the index $n$ and using the fact that \\ $\left\Vert N^{- \, k \, a}_{B_{2}} \right\Vert_{\mathcal{L}\left(\mathbb{L}^{2}(B_{2});\mathbb{L}^{2}(B_{2})\right)} \, = \, \mathcal{O}\left( 1 \right)$, we obtain 
	
	 	\begin{equation}\label{es-P1W2-P3}
	 		\left\lVert\overset{1}{\PP}\left(\tilde{W}_2\right)\right\rVert_{\mathbb{L}^2(B_2)}^2 \, \lesssim \, a^2 \, + \, a^4 \,  \left\lVert\overset{3}{\PP}\left(\tilde{W}_2\right)\right\rVert_{\mathbb{L}^2(B_2)}^2.
	 	\end{equation}
	 \medskip	
       \newline

	 	Finally, by taking the inner product with respect to $e_{n, B_{2}}^{(3)}(\cdot)$, in $(\ref{LS-W2-1})$, and using $(\ref{L2-decomposition})$ and $(\ref{grad-M-1st})$, we obtain 
        \begin{eqnarray*}
            \left\langle \overset{3}{\PP}\left(\tilde{W}_2\right), e_{n, B_{2}}^{(3)} \right\rangle_{\mathbb{L}^{2}(B_{2})} \, &=& \, \eta_{2} \, \frac{\left\langle \left[ - \left( \nabla M^{ - \, k \, a}_{B_{2}} \, - \,  \nabla M_{B_{2}} \right) \, + \, k^{2} \, a^{2} \,  N^{- \, k \, a}_{B_{2}}\right]\left(  \overset{3}{\PP}\left(\tilde{W}_2\right)\right), e_{n, B_{2}}^{(3)} \right\rangle_{\mathbb{L}^{2}(B_{2})}}{\left( 1 \, + \, \eta_2 \, \lambda_{n}^{(3)}(B_{2}) \right)} \\
            &+& \eta_2 k^2 a^2 \frac{\left\langle N_{B_2}\left(\overset{1}{\PP}(\tilde{W}_2)\right), e_{n, B_{2}}^{(3)}\right\rangle_{\mathbb{L}^2(B_2)}}{\left(1+\eta_2 \lambda_n^{(3)}(B_2)\right)} \\ &+&
            \, \eta_{2} \, k^{2} \, a^{2} \, \frac{ \left\langle \left( N^{- \, k \, a}_{B_{2}} \, - \, N_{B_{2}} \right)\left( \overset{1}{\PP}\left(\tilde{W}_2\right) \right), e_{n, B_{2}}^{(3)} \right\rangle_{\mathbb{L}^{2}(B_{2})}}{\left( 1 \, + \, \eta_2 \, \lambda_{n}^{(3)}(B_{2}) \right)},
        \end{eqnarray*}
        where we know that
        \begin{eqnarray*}
            \left\langle N_{B_2}\left(\overset{1}{\PP}(\tilde{W}_2)\right), e_{n, B_{2}}^{(3)}\right\rangle_{\mathbb{L}^{2}(B_{2})} \, &=& \, \sum_n \left\langle \overset{1}{\PP}\left(\tilde{W}_2\right), e_{n, B_{2}}^{(1)}\right\rangle_{\mathbb{L}^{2}(B_{2})} \; \left\langle N_{B_2}(e_{n,B_{2}}^{(1)}), e_{n, B_{2}}^{(3)}\right\rangle_{\mathbb{L}^{2}(B_{2})} \\ &=& \sum \lambda_n^{(1)}(B_2) \left\langle \overset{1}{\PP}\left(\tilde{W}_2\right), e_{n,B_{2}}^{(1)}\right\rangle_{\mathbb{L}^{2}(B_{2})} \, \left\langle e_{n,B_{2}}^{(1)}, e_{n,B_{2}}^{(3)}\right\rangle_{\mathbb{L}^{2}(B_{2})} \, = \, 0.
        \end{eqnarray*}
        Thanks to $(\ref{expansion-gradMk}), (\ref{expansion-Nk})$ and $(\ref{RefNeeded1})$, we can approximate the R.H.S of the above equation by 
        \begin{eqnarray*}
            \left\langle \overset{3}{\PP}\left(\tilde{W}_2\right), e_{n,B_{2}}^{(3)} \right\rangle_{\mathbb{L}^{2}(B_{2})} \, &\simeq& \, \, \eta_{2} \, \frac{k^{2} \, a^{2}}{2} \, \frac{\left\langle N_{B_{2}}\left(  \overset{3}{\PP}\left(\tilde{W}_2\right)\right) , e_{n,B_{2}}^{(3)} \right\rangle_{\mathbb{L}^{2}(B_{2})}}{\left( 1 \, + \, \eta_2 \, \lambda_{n}^{(3)}(B_{2}) \right)} \\ &-& \, \eta_{2}  \, \frac{k^{4} \, a^4}{8 \, \pi} \, \frac{ \left\langle \int_{B_{2}}\left\vert \cdot - y \, \right\vert \, \overset{1}{\PP}\left(\tilde{W}_2\right)(y) \, dy, e_{n,B_{2}}^{(3)} \right\rangle_{\mathbb{L}^{2}(B_{2})}}{\left( 1 \, + \, \eta_2 \, \lambda_{n}^{(3)}(B_{2}) \right)}.
        \end{eqnarray*}
        Now, by taking the square modulus, summing up with respect to the index $n$, and using $(\ref{def-eta})$, we obtain 
        \begin{eqnarray*}
            \sum_{n} \left\vert \left\langle \overset{3}{\PP}\left(\tilde{W}_2\right), e_{n,B_{2}}^{(3)} \right\rangle_{\mathbb{L}^{2}(B_{2})} \right\vert^{2} \, & \lesssim & \,  a^{4}  \, \sum_{n} \frac{\left\vert \left\langle N_{B_{2}}\left(  \overset{3}{\PP}\left(\tilde{W}_2\right)\right), e_{n,B_{2}}^{(3)} \right\rangle_{\mathbb{L}^{2}(B_{2})} \right\vert^{2}}{\left\vert 1 \, + \, \eta_2 \, \lambda_{n}^{(3)}(B_{2}) \right\vert^{2}} \\ &+& \, a^{8} \, \sum_{n} \frac{\left\vert \left\langle \int_{B_{2}}\left\vert \cdot - y \, \right\vert \, \overset{1}{\PP}\left(\tilde{W}_2\right)(y) \, dy, e_{n,B_{2}}^{(3)} \right\rangle_{\mathbb{L}^{2}(B_{2})} \right\vert^{2}}{\left\vert 1 \, + \, \eta_2 \, \lambda_{n}^{(3)}(B_{2}) \right\vert^{2}}.
        \end{eqnarray*}
        Besides, knowing $(\ref{condition-on-k})$ and using the fact that $\left\Vert N \right\Vert_{\mathcal{L}\left(\mathbb{L}^{2}(B_{2}); \mathbb{L}^{2}(B_{2})\right)} \, = \, \mathcal{O}\left( 1 \right)$, we deduce  
	 	\begin{equation}\label{es-P3W2-P1}
	 		\left\lVert\overset{3}{\PP}\left(\tilde{W}_2\right)\right\rVert_{\mathbb{L}^2(B_2)}^2 \; \lesssim \; a^{8-2h} \, \left\lVert\overset{1}{\PP}\left(\tilde{W}_2\right)\right\rVert_{\mathbb{L}^2(B_2)}^2. 
	 	\end{equation}
	 	Finally, by gathering $(\ref{es-P1W2-P3})$ and $(\ref{es-P3W2-P1})$, we deduce
	 	\begin{equation}\notag
	 		\left\lVert\overset{1}{\PP}\left(\tilde{W}_2\right)\right\rVert_{\mathbb{L}^2(B_2)}=\O\left( a\right) \quad \text{and} \quad \left\lVert\overset{3}{\PP}\left(\tilde{W}_2\right)\right\rVert_{\mathbb{L}^2(B_2)}=\O(a^{5-h}).
	 	\end{equation}
            \item[] 
	 	\item Estimation of $\left\lVert \tilde{V}_{m} \right\rVert_{\mathbb{L}^2(B_{m})}$, for $m=1 , 2$. \\
	 	By recalling the definition of $V_m$, for $m=1, 2$,  
            see for instance $(\ref{def-Vm})$, using the operator notation given by $(\ref{def-Tkm})$, and scaling from $D_{m}$ to $B_{m}$, we obtain
	 		\begin{equation}\label{HA0754}
	 		\tilde{V}_{m}(x) \; = \; \TT_{-ka, m}^{-1}(I)(x), \quad x \in B_{m}, \;\; \text{for} \;\; m=1, 2.
	 	\end{equation}
   Next, we project $\tilde{V}_{1}$ onto the three sub-spaces appearing in $(\ref{L2-decomposition})$. 
	 	\begin{enumerate}
                    \item[]  
	 		\item We project $\tilde{V}_{1}$ onto 
                     $\mathbb{H}_{0}(\div=0)$. \\
	 		We have,
		 		\begin{equation}\notag
	 			0 \overset{(\ref{RefNeeded1})}{=} \left\langle I, e_{n,B_{1}}^{(1)} \right\rangle_{\mathbb{L}^{2}(B_{1})} \, \overset{(\ref{HA0754})}{=} \, \left\langle \TT_{-ka, 1}(\tilde{V}_1), e_{n,B_{1}}^{(1)} \right\rangle_{\mathbb{L}^{2}(B_{1})} \, = \, \left\langle \tilde{V}_{1}, \TT_{ka, 1}(e_{n,B_{1}}^{(1)}) \right\rangle_{\mathbb{L}^{2}(B_{1})},
	 		\end{equation}
	 		with
	 		\begin{equation}\notag
	 			\TT_{ka, 1}(e_{n,B_{1}}^{(1)}) \, = \, \left(I+\eta_1\nabla M^{ka}- k^2 a^2 \eta_1 N^{ka}\right)(e_{n,B_{1}}^{(1)}) \, \overset{(\ref{grad-M-1st})}{=} \, \left(I-k^2 a^2 \eta_1 N^{ka}\right)(e_{n,B_{1}}^{(1)}),
	 		\end{equation}
	 		then, by using $(\ref{expansion-Nk})$, we can get that
	 		\begin{equation}\label{V1-en1}
	 			\left\langle \tilde{V}_1, e_{n,B_{1}}^{(1)} \right\rangle_{\mathbb{L}^{2}(B_{1})} \; = \; \frac{k^2 a^2 \eta_1}{4\pi\left(1 \,- \, k^2 \, a^2 \, \eta_1 \, \lambda_{n}^{(1)}(B_{1})\right)}\sum_{\ell \geq 1}\frac{( i k a)^{\ell+1}}{(\ell+1)!} \,  \left\langle \tilde{V}_{1}, \int_{B_1}\lvert \cdot -y\rvert^\ell e_{n,B_{1}}^{(1)}(y)\,dy \right\rangle_{\mathbb{L}^{2}(B_{1})}.
	 		\end{equation}
	 		\item[]
	 		\item We project $\tilde{V}_1$ onto $\mathbb{H}_0(Curl=0)$.\\
	 		We have
                    \begin{eqnarray*}\notag
	 			0  \overset{(\ref{RefNeeded1})}{=}  \left\langle I, e_{n,B_{1}}^{(2)} \right\rangle_{\mathbb{L}^{2}(B_{1})} \, & \overset{(\ref{HA0754})}{=} & \, \left\langle \TT_{-ka, 1}(\tilde{V}_1), e_{n,B_{1}}^{(2)} \right\rangle_{\mathbb{L}^{2}(B_{1})} \\ &=& \, \left\langle \tilde{V}_{1}, \TT_{ka, 1}(e_{n,B_{1}}^{(2)}) \right\rangle_{\mathbb{L}^{2}(B_{1})} \, \overset{(\ref{grad-M-2nd})}{\underset{(\ref{def-Tkm})}{=}} \, \left(1 \, + \, \eta_1 \, \right) \, \left\langle \tilde{V}_1, e_{n,B_{1}}^{(2)} \right\rangle_{\mathbb{L}^{2}(B_{1})}, 
	 		\end{eqnarray*}
	 		which yields
	 		\begin{equation}\label{V1-en2}
	 			\left\langle \tilde{V}_1, e_{n,B_{1}}^{(2)} \right\rangle_{\mathbb{L}^{2}(B_{1})} \, = \, 0.
	 		\end{equation}
	 		\item[] 
	 		\item We project $\tilde{V}_1$ onto $\nabla \mathcal{H}armonic$. \\
                    As previously stated, we have
	 		\begin{eqnarray*}\notag
	 			\left\langle I, e_{n,B_{1}}^{(3)} \right\rangle_{\mathbb{L}^{2}(B_{1})} \, &=& \, \left\langle \tilde{V}_{1}, \TT_{ka, 1}\left(e_{n,B_{1}}^{(3)}\right)\right\rangle_{\mathbb{L}^{2}(B_{1})} \\ &\overset{(\ref{def-Tkm})}{=} & \, \left(1 \, + \, \eta_1\lambda_{n}^{(3)}(B_{1})\right) \, \left\langle \tilde{V}_{1}, e_{n,B_{1}}^{(3)} \right\rangle_{\mathbb{L}^{2}(B_{1})} \\
                   &+& \, \eta_1 \, \left\langle \tilde{V}_{1}, \left(\nabla M^{ka}-\nabla M\right)\left(e_{n,B_{1}}^{(3)}\right) \, - \, k^2 \, a^2 \,  N^{ka}\left(e_{n,B_{1}}^{(3)}\right)\right\rangle_{\mathbb{L}^{2}(B_{1})}.
	 		\end{eqnarray*}
	 		Then, 
        \begin{equation}\label{V1-en3}
        \left\langle \tilde{V}_1, e_{n,B_{1}}^{(3)} \right\rangle_{\mathbb{L}^{2}(B_{1})} \, = \,\frac{1}{\left(1 \, + \, \eta_1 \, \lambda_{n}^{(3)}(B_{1})\right)} \, \left\langle I, e_{n,B_{1}}^{(3)}\right\rangle_{\mathbb{L}^{2}(B_{1})} \, + \, \frac{\eta_1}{\left( 1 \, +\, \eta_1 \, \lambda_{n}^{(3)}(B_{1}) \right)}\left\langle \tilde{V}_1, \mathcal{S}^{ka}\left(e_{n,B_{1}}^{(3)}\right)\right\rangle_{\mathbb{L}^{2}(B_{1})}, 
         \end{equation}
	 		where 
	 		\begin{eqnarray}\label{def-Ska}
	 \mathcal{S}^{ka}\left(e_{n,B_{1}}^{(3)}\right)&:=&\frac{ k^2 a^2 }{2}N\left(e_{n,B_{1}}^{(3)}\right)+\frac{k^2 a^2}{2}\int_{B_1}\Phi_{0}(\cdot, y)\frac{A(\cdot, y)\cdot e_{n,B_{1}}^{(3)}(y)}{\lVert x-y\rVert^2}\,dy+\frac{i k^3 a^3}{6\pi}\int_{B_1} e_{n,B_{1}}^{(3)}(y)\,dy\notag\\
	 			&+& \frac{1}{4\pi}\sum_{\ell\geq 3} (i k a)^{\ell+1}\int_{B_1}\frac{\nabla \nabla (| \cdot -y|^\ell) \cdot e_{n,B_{1}}^{(3)}(y)}{(\ell+1)!}\,dy \\ \nonumber &+& \frac{k^2 a^2}{4\pi}\sum_{\ell \geq 1}(i k a)^{\ell+1}\int_{B_1}\frac{| \cdot - y |^\ell}{(\ell+1)!}e_{n,B_{1}}^{(3)}(y)\,dy.
	 		\end{eqnarray}
	 	\end{enumerate}	
	 	By combining with \eqref{V1-en1}, \eqref{V1-en2} and \eqref{V1-en3}, we can obtain that
	 	\begin{eqnarray*}
	 	\left\lVert \tilde{V}_1\right\rVert_{\mathbb{L}^2(B_1)}^2 \, &=& \, \sum_n \left|\left\langle \tilde{V}_1, e_{n,B_{1}}^{(1)} \right\rangle_{\mathbb{L}^2(B_1)} \right|^2 \, + \, \sum_n \left| \left\langle \tilde{V}_1, e_{n,B_{1}}^{(3)} \right\rangle_{\mathbb{L}^2(B_1)} \right|^2\notag\\
	 	&\lesssim& \, \sum_{n} \left| \frac{1}{\left( 1 \, -\, k^2 a^2 \eta_1\lambda_{n}^{(1)}(B_{1}) \right)}\sum_{\ell \geq 1} \frac{(i k a)^{\ell+1}}{(\ell+1)!} \, \left\langle \tilde{V}_1, \int_{B_1}| \cdot - y|^\ell e_{n,B_{1}}^{(1)}(y)\,dy \right\rangle_{\mathbb{L}^{2}(B_{1})} \right|^2\notag\\
	 	&+&\sum_{n} \frac{\left| \left\langle I, e_{n,B_{1}}^{(3)} \right\rangle_{\mathbb{L}^{2}(B_{1})} \right|^2}{\left|1 \, + \, \eta_{1} \, \lambda_{n}^{(3)}(B_{1}) \right|^2} \,  + \, |\eta_1|^2 \, \sum_n\frac{\left|\left\langle \tilde{V}_1, \mathcal{S}^{ka}\left(e_{n,B_{1}}^{(3)}\right) \right\rangle_{\mathbb{L}^{2}(B_{1})} \right|^2}{\left|1 \, + \, \eta_1 \, \lambda_{n}^{(3)}(B_{1})\right|^2}, 
        \end{eqnarray*}
        which, by using $(\ref{def-eta})$ and $(\ref{condition-on-k})$, can be reduced to 
        \begin{equation*}
            \left\lVert \tilde{V}_1\right\rVert_{\mathbb{L}^2(B_1)}^2 \,   \lesssim  \,  a^{4-2h} \, \left\Vert \tilde{V}_{1} \right\Vert^{2}_{\mathbb{L}^{2}(B_{1})} \, + \,  a^4 \,  \left\Vert I \right\Vert^{2}_{\mathbb{L}^{2}(B_{1})} \, + \, \left\Vert \mathcal{S}^{-ka}\left(\tilde{V}_1\right) \right\Vert_{\mathbb{L}^{2}(B_{1})}^{2}.
        \end{equation*}
        The estimation of the last term on the R.H.S, can be done by keeping the dominant term in $(\ref{def-Ska})$. More precisely, we have
        \begin{equation}\label{EstSka}
            \left\Vert \mathcal{S}^{-ka}\left(\tilde{V}_1\right) \right\Vert_{\mathbb{L}^{2}(B_{1})} \, \lesssim \, a^{2} \, \left\Vert N\left(\tilde{V}_1\right) \right\Vert_{\mathbb{L}^{2}(B_{1})} \, = \, \mathcal{O}\left( a^{2} \, \left\Vert \tilde{V}_{1} \right\Vert_{\mathbb{L}^{2}(B_{1})} \right),
        \end{equation}
        where the last estimation is a consequence of the continuity of the Newtonian operator. Hence, under the condition $h \, < \, 2$, we obtain 
            \begin{equation}\label{es-V1} 
            \left\lVert \tilde{V}_1 \right\rVert_{\mathbb{L}^2(B_1)} \; \lesssim \; a^2.
	 	\end{equation}
       \medskip
       \newline 
        To evaluate $\left\lVert \tilde{V}_2\right\rVert_{\mathbb{L}^2(B_2)}$, we recall $(\ref{HA0754})$ and we use similar arguments as those used for deriving $(\ref{es-V1})$, to obtain
	 	\begin{eqnarray}\label{V2-project}
	 		\left\langle \tilde{V}_2, e_{n,B_{2}}^{(1)} \right\rangle_{\mathbb{L}^2(B_2)} \, &=& \, \frac{ k^2 a^2 \eta_2}{4\pi \left(1- k^2 a^2 \eta_2\lambda_{n}^{(1)}(B_{2})\right)}\sum_{\ell \geq 1}\frac{(i k a)^{\ell+1}}{(\ell+1)!} \left\langle \tilde{V}_2, \int_{B_2} \lvert \cdot - y \rvert^\ell  e_{n,B_{2}}^{(1)}(y)\,dy \right\rangle_{\mathbb{L}^2(B_2)}, \notag\\
	 		\left\langle \tilde{V}_2, e_{n,B_{2}}^{(2)} \right\rangle_{\mathbb{L}^2(B_2)} \, &=& \, 0,\notag\\
	 	&&	\left\langle \tilde{V}_2, e_{n,B_{2}}^{(3)} \right\rangle_{\mathbb{L}^2(B_2)} \, = \,  \frac{\left\langle I, e_{n,B_{2}}^{(3)} \right\rangle_{\mathbb{L}^2(B_2)} \, + \, \eta_2 \, \left\langle \tilde{V}_2, \mathcal{S}^{ka}\left(e_{n,B_{2}}^{(3)}\right) \right\rangle_{\mathbb{L}^2(B_2)}}{\left( 1 \, + \, \eta_2 \, \lambda_{n}^{(3)}(B_{2})\right)},
	 	\end{eqnarray}
	 	with $\mathcal{S}^{ka}\left(e_{n,B_{2}}^{(3)}\right)$ being given by \eqref{def-Ska}. Thus, using $(\ref{def-eta}), (\ref{V2-project})$ and $(\ref{condition-on-k})$, we deduce 
            \begin{equation*}
                \left\lVert \tilde{V}_2\right\rVert_{\mathbb{L}^2(B_2)}^2 \, \lesssim \, a^{8} \, \left\lVert \tilde{V}_2\right\rVert_{\mathbb{L}^2(B_2)}^2 \, + \, a^{-2h} \, \left\lVert I \right\rVert_{\mathbb{L}^2(B_2)}^2 \, + \, a^{-2h} \, \lVert \mathcal{S}^{-ka}\left(\tilde{V}_2\right) \rVert_{\mathbb{L}^2(B_2)}^2, 
            \end{equation*}
            which, by using $(\ref{EstSka})$ and knowing that $h \, < \, 2$, can gives us 
            \begin{equation*}
                \left\lVert \tilde{V}_2\right\rVert_{\mathbb{L}^2(B_2)} \, = \, \mathcal{O}\left( a^{-h} \right).
            \end{equation*}
	 	\end{enumerate}
	 	            This ends the proof of Proposition \ref{prop-scoeff}.
	 	\section{Derivation of the linear algebraic system} \label{sec-proof-la}
	 	In this section, we shall rigorously present the derivation of the linear algebraic system associated with our main theorem by proving Proposition \ref{prop-la-1}, in $D_1$ and $D_2$, respectively, as follows. \\ \newline 
	 	 To investigate the linear algebraic system in $D_1$, we multiply each side of $(\ref{*add3})$ by $\eta_1$, and we take the inverse of the operator $\TT_{k ,1}$, to get
	 		\begin{equation}\label{LS-D1-2}
	 			\eta_1 \, E_1(x) \, -  \, k^2 \, \eta_1 \, \eta_2 \, \TT_{k ,1}^{-1}\int_{D_{2}}\Upsilon_k(x, y) \cdot E_2(y) \, dy \, = \, \eta_1 \, \TT_{k ,1}^{-1}E_1^{Inc}(x), \quad x \in D_{1},
	 		\end{equation}
            where $\Upsilon_k(x, y)$ is the Dyadic Green's kernel given by \eqref{dyadicG}.
	 		\begin{enumerate}
                    \item The algebraic system with respect to $\overset{1}{\PP}\left( E_1 \right)$. \\
	 		By multiplying by $\mathcal{P}(x, z_1)$ on the both sides of \eqref{LS-D1-2},  integrating over $D_1$,
            taking the adjoint operator of $\TT_{-k, 1}^{-1}$ and using $(\ref{DefW1})$, we obtain
	 		\begin{equation*}
	 			\eta_1 \, \int_{D_{1}}\mathcal{P}(x, z_1) \cdot E_1(x) \,dx \, - \, k^2 \, \eta_1 \, \eta_2 \, \int_{D_{1}}W_1(x) \cdot \int_{D_{2}}\Upsilon_k(x, y) \cdot E_2(y)\,dy\,dx \, = \, \eta_1 \, \int_{D_{1}}W_1(x) \cdot E_1^{Inc}(x)\,dx.
	 		\end{equation*}
            The above equation suggest the coming one,
	 		\begin{eqnarray}\label{LS-D1-3}
	 			\eta_1 \, \int_{D_{1}}\mathcal{P}(x, z_1) \cdot \overset{1}{\PP}\left( E_1 \right)(x)\,dx & \, - \, & \, k^2\, \eta_1 \, \eta_2  \,  \int_{D_{1}} W_1(x) \cdot \int_{D_{2}} \Phi_{k}(x, y) \, \overset{1}{\PP}\left(E_2\right)(y)\,dy\,dx \nonumber \\ &=& \, \eta_1 \, \int_{D_{1}}W_1(x) \cdot E_1^{Inc}(x)\,dx  \notag\\
	 			&+& k^2 \, \eta_1 \eta_2 \int_{D_{1}}W_1(x) \cdot \int_{D_{2}}\Upsilon_k(x, y) \cdot \overset{3}{\PP}\left(E_2\right)(y)\,dy\,dx \, - \, J_{1}(z_{1}),
	 		\end{eqnarray}
	 		where we have used $(\ref{grad-M-1st})$, and we have denoted by $J_{1}(z_{1})$ the following term
	 		\begin{equation}\notag
	 			J_{1}(z_{1}) \, := \, \eta_1 \, \int_{D_{1}}\mathcal{P}(x, z_1) \cdot \overset{3}{\PP}\left(E_1\right)(x)\,dx,
	 		\end{equation}
	 		which can be estimated, using $(\ref{def-eta})$, as 
	 		\begin{equation}\label{EstJ1z1}
	 			\left\vert J_{1}(z_{1}) \right\vert \; \lesssim \; \left\vert \eta_{1} \right\vert \, \left\lVert\mathcal{P}(\cdot, z_1)\right\rVert_{\mathbb{L}^2(D_1)} \, \left\lVert \overset{3}{\PP}\left(E_1 \right)\right\rVert_{\mathbb{L}^2(D_1)} \, = \, \mathcal{O}\left( a^2 \; \left\lVert\overset{3}{\PP}\left(\tilde{E}_1\right)\right\rVert_{\mathbb{L}^2(B_1)} \right).
	 		\end{equation}
	 		By writing 
	 		\begin{equation*}
	 			W_1 \, = \, \overset{1}{\PP}\left(W_1\right) \, + \, \overset{2}{\PP}\left(W_1\right) \, + \, \overset{3}{\PP}\left(W_1\right),
	 		\end{equation*}
                 returning to $(\ref{LS-D1-3})$ and using $(\ref{EstJ1z1})$, we can obtain the following expression
	 		\begin{eqnarray}\label{LS-D1-split}
                    \nonumber
	 			\eta_1 \, \int_{D_{1}}\mathcal{P}(x, z_1) \cdot \overset{1}{\PP}\left(E_1\right)(x)\,dx \, &-& \,k^2\, \eta_1 \, \eta_2  \, \int_{D_{1}}\overset{1}{\PP}\left(W_1\right)(x) \cdot \int_{D_{2}} \Phi_{k}(x, y) \, \overset{1}{\PP}\left(E_2\right)(y)\,dy\,dx \\ &=& \, \eta_1 \, \int_{D_{1}}\overset{1}{\PP}\left(W_1\right)(x) \cdot E_1^{Inc}(x)\,dx\,  \notag\\ 
	 			&+&  k^2\, \eta_1 \, \eta_2  \, \int_{D_{1}}\overset{1}{\PP}\left(W_1\right)(x) \cdot \int_{D_{2}} \Phi_{k}(x, y) \, \overset{3}{\PP}\left(E_2\right)(y)\,dy\,dx \, + \, Error_{1},
	 		\end{eqnarray}
            with 
            \begin{equation}\label{Equa08093}
                Error_{1} \, := \, J_{3} + \, k^2 \, J_{4,2} \, + \, k^2\,J_{4,3} \, + \, k^2\, J_{2,2} \, + \, k^2\, J_{2,3} + \O\left(a^2 \left\lVert \overset{3}{\PP}\left(\tilde{E}_1\right)\right\rVert_{\mathbb{L}^2(B_1)}\right),
            \end{equation}
	 		where we have used the fact that 
                 \begin{eqnarray*}
                  \int_{D_{1}}\overset{2}{\PP}\left(W_1\right)(x) \cdot E_1^{Inc}(x)\, dx \, &=& \, 0, \\
                  \int_{D_{1}}\overset{1}{\PP}\left(W_1\right)(x) \cdot \int_{D_{2}} \underset{y}{\nabla} \, \underset{y}{\nabla} \, \Phi_{k}(x, y) \cdot \overset{3}{\PP}\left(E_2\right)(y)\,dy\,dx \, &=& \, 0,
                 \end{eqnarray*}   
                     and we have denoted by\footnote{The formula of $J_{2,i}$, for $i=2,3$, has been derived by using $(\ref{grad-M-1st})$.}
	 		\begin{eqnarray}\notag
	 			J_{2,i} \, &:=& \, \eta_1 \, \eta_2  \, \left\langle \overset{i}{\PP}\left(W_1\right), N_{D_{2}}^{k}\left(\overset{1}{\PP}\left(E_2\right)\right) \right\rangle_{\mathbb{L}^{2}(D_{1})} ,\quad\mbox{for}\quad i=2, 3.\notag\\
	 			J_3 \, &:=& \, \eta_1 \, \left\langle \overset{3}{\PP}\left(W_1\right),  E_1^{inc}\right\rangle_{\mathbb{L}^{2}(D_{1})}. \notag\\
	 			J_{4, i} \, &:=& \,  \eta_1 \, \eta_2 \, \left\langle \overset{i}{\PP}\left(W_1\right), \left( - \nabla M_{D_{2}}^{k} \, + \, k^{2} \, N_{D_{2}}^{k} \right)\left( \overset{3}{\PP}\left(E_2\right) \right)\right\rangle_{\mathbb{L}^{2}(D_{1})} , \quad\mbox{for}\quad i=2, 3.\notag
	 		\end{eqnarray}
	 	Then, by using Taylor expansions and the formula $(\ref{RefNeeded1})$, we rewrite $J_{2,i}$ as 
   \begin{equation}\label{Equa0932}
       J_{2, i} \, := \, J_{2, i}^{(1)} \, + \, J_{2, i}^{(2)} \, + \, J_{2, i}^{(3)},
   \end{equation}
where 
\begin{eqnarray*}
    J_{2, i}^{(1)}&:=& \eta_1 \eta_2 \int_{D_1} \overset{i}{\PP}(W_1)(x)\cdot\underset{y}{\nabla}(\Phi_k I)(z_1, z_2)\int_{D_2}\mathcal{P}(y, z_2)\, \overset{1}{\PP}(E_2)(y)\,dy\,dx\notag\\
    |J_{2, i}^{(1)}|&\lesssim& \;|\eta_1| \; |\eta_2| \; \left
    \lVert \overset{i}{\PP}(W_1)\right\rVert_{\mathbb{L}^2(D_1)}\cdot\left\lVert\underset{y}{\nabla}(\Phi_k I)(z_1, z_2)\int_{D_2}\mathcal{P}(y, z_2)\, \overset{1}{\PP}(E_2)(y)\,dy\right\rVert_{\mathbb{L}^2(D_1)}\notag\\
    &\lesssim& \;a^5 \; d^{-2}\;  \left\lVert \overset{i}{\PP}(\tilde{W}_1)\right\rVert_{\mathbb{L}^2(B_1)}\left\lVert\overset{1}{\PP}(\tilde{E}_2)(y)\right\rVert_{\mathbb{L}^2(B_2)}
\end{eqnarray*}
   and 
	 			\begin{eqnarray}\label{es-J21-1}
                \nonumber
	 			J_{2, i}^{(2)}&:=& \eta_1 \eta_2 \int_{D_{1}}\overset{i}{\PP}\left(W_1\right)(x) \cdot \Bigg[ \\ \nonumber && \int_{D_{2}}\int_0^1 (1-t) (y-z_2)^\perp\underset{x}{\nabla} \underset{x}{\nabla}\Phi_{k}(z_1, z_2+t(y-z_2)) \cdot (y-z_2)\,dt\cdot \overset{1}{\PP}\left(E_2\right)(y)\,dy \notag\\
                &-&  \int_{D_{2}}\int_0^1 \underset{x}{\nabla} \underset{x}{\nabla}\Phi_{k}(z_1+t(x-z_1),z_2) \cdot (x-z_1)\,dt\cdot (y-z_2)\cdot\overset{1}{\PP}\left(E_2\right)(y)\,dy\Bigg] dx,\notag\\
	 			\left\vert J_{2, i}^{(2)} \right\vert \, 
	 			& \overset{(\ref{def-eta})}{\lesssim} &     \, a^6 \,  d^{-3} \, \left\lVert\overset{i}{\PP}\left(\tilde{W}_1\right)\right\rVert_{\mathbb{L}^2(B_1)} \, \left\lVert\overset{1}{\PP}\left(\tilde{E}_2 \right)\right\rVert_{\mathbb{L}^2(B_2)} ,
	 			\end{eqnarray}
	 			and 
	 			\begin{eqnarray}\label{es-J21-2}
                \nonumber
	 			&&	J_{2, i}^{(3)} \\ \nonumber &:=&\eta_1\eta_2 \int_{D_{1}} \overset{i}{\PP}\left(W_1\right)(x) \cdot \int_{D_{2}}\int_0^1 (1-t)(y-z_2)^\perp \cdot \int_0^1 \underset{x}{\nabla}\left(\underset{y}{\nabla} \underset{y}{\nabla} \Phi_k\right)(z_1+s(x-z_1), z_2+t(y-z_2))\notag\\
	 				&\cdot& \mathcal{P}(x, z_1)\,ds\cdot (y-z_2)\,dt\cdot\overset{1}{\PP}\left(E_2\right)\,dy\,dx.\notag\\ 
	 				\left\vert J_{2, i}^{(3)} \right\vert & \lesssim & \nonumber \left\vert \eta_1 \right\vert \left\vert \eta_2 \right\vert \, a^3 \, \frac{1}{|z_1-z_2|^4} |D_1|^\frac{1}{2} |D_2|^\frac{1}{2} \left\lVert\overset{i}{\PP}\left(W_1\right)\right\rVert_{\mathbb{L}^2(D_1)} \left\lVert\overset{1}{\PP}\left(E_2\right)\right\rVert_{\mathbb{L}^2(D_2)} \\ &\lesssim&  \, a^7 \,  d^{-4} \, \left\lVert\overset{i}{\PP}\left(\tilde{W}_1\right)\right\rVert_{\mathbb{L}^2(B_1)} \left\lVert\overset{1}{\PP}\left(\tilde{E}_2 \right)\right\rVert_{\mathbb{L}^2(B_2)}.
	 			\end{eqnarray}
	 			Hence, by returning to $(\ref{Equa0932})$, using $(\ref{es-J21-1})$ and $(\ref{es-J21-2})$, and the fact that $t<1$, we deduce  
	 			\begin{equation}\label{es-J2i}
	 				|J_{2, i}|\lesssim a^5 d^{-2} \left\lVert\overset{i}{\PP}\left(\tilde{W}_1\right)\right\rVert_{\mathbb{L}^2(B_1)}\left\lVert\overset{1}{\PP}\left(\tilde{E}_2\right)\right\rVert_{\mathbb{L}^2(B_2)}.
	 			\end{equation}
     For the estimation of $J_3$, using $(\ref{def-eta})$ and the smoothness of the incident field, it is obvious that
	 			\begin{equation}\label{es-J3}
	 				|J_3|\lesssim |\eta_1|\left\lVert\overset{3}{\PP}\left(W_1\right)\right\rVert_{\mathbb{L}^2(D_1)}\left\lVert E^{Inc}_1\right\rVert_{\mathbb{L}^2(D_1)} \, = \, \mathcal{O}\left(a \, \left\lVert\overset{3}{\PP}\left(\tilde{W}_1\right)\right\rVert_{\mathbb{L}^2(B_1)} \right).
	 			\end{equation}
	Now, we focus on the estimation of $J_{4,i}$, for $i = 2, 3$. 		
        \begin{enumerate}
            \item[]
            \item Estimation of $J_{4,2}$. By its definition, we have 
            \begin{eqnarray}\label{J42=0}
            \nonumber
                J_{4, 2} \, &:=& \,  \eta_1 \, \eta_2 \, \left\langle \overset{2}{\PP}\left(W_1\right), \left( - \nabla M_{D_{2}}^{k} \, + \, k^{2} \, N_{D_{2}}^{k} \right)\left( \overset{3}{\PP}\left(E_2\right) \right)\right\rangle_{\mathbb{L}^{2}(D_{1})} \\ 
                &=& \,  \eta_1 \, \eta_2 \, \left\langle \left( - \nabla M_{D_{1}}^{-k} \, + \, k^{2} \, N_{D_{1}}^{-k} \right)\left(\overset{2}{\PP}\left(W_1\right)\right), \overset{3}{\PP}\left(E_2\right) \right\rangle_{\mathbb{L}^{2}(D_{2})} \, \overset{(\ref{grad-M-2nd})}{=} \, 0.
            \end{eqnarray}
            \item[]
            \item Estimation of $J_{4,3}$. By its definition, we have 
            \begin{eqnarray*}
                J_{4, 3} \, &:=& \,  \eta_1 \, \eta_2 \, \left\langle \overset{3}{\PP}\left(W_1\right), \left( - \nabla M_{D_{2}}^{k} \, + \, k^{2} \, N_{D_{2}}^{k} \right)\left( \overset{3}{\PP}\left(E_2\right) \right)\right\rangle_{\mathbb{L}^{2}(D_{1})} \\
                & = & \, - \, \eta_1 \, \eta_2 \, \left\langle \overset{3}{\PP}\left(W_1\right),  \nabla M_{D_{2}}\left( \overset{3}{\PP}\left(E_2\right) \right)\right\rangle_{\mathbb{L}^{2}(D_{1})} \\
                & + & \,  \eta_1 \, \eta_2 \, \left\langle \overset{3}{\PP}\left(W_1\right), \left( - \left( \nabla M_{D_{2}}^{k} \, - \, \nabla M_{D_{2}} \right) \, + \, k^{2} \, N_{D_{2}}^{k} \right)\left( \overset{3}{\PP}\left(E_2\right) \right)\right\rangle_{\mathbb{L}^{2}(D_{1})} \\
                & \underset{(\ref{expansion-Nk})}{\overset{(\ref{expansion-gradMk})}{\simeq}}& \, - \, \eta_1 \, \eta_2 \, \left\langle \overset{3}{\PP}\left(W_1\right),  \nabla M_{D_{2}}\left( \overset{3}{\PP}\left(E_2\right) \right)\right\rangle_{\mathbb{L}^{2}(D_{1})} \, +  \,  \eta_1 \, \eta_2 \, \frac{k^{2}}{2} \, \left\langle \overset{3}{\PP}\left(W_1\right), N_{D_{2}}\left( \overset{3}{\PP}\left(E_2\right) \right)\right\rangle_{\mathbb{L}^{2}(D_{1})} \\
            \end{eqnarray*}
            which, by taking the modulus on its both sides and using $(\ref{def-eta})$, gives us
        \begin{equation*} 
             \left\vert J_{4, 3} \right\vert \,  \lesssim  \, a^{-2} \, \left\Vert \overset{3}{\PP}\left(W_1\right) \right\Vert_{\mathbb{L}^{2}(D_{1})} \,  \left\Vert \overset{3}{\PP}\left(E_2\right)  \right\Vert_{\mathbb{L}^{2}(D_{2})} \, \left[ \left\Vert \nabla M \right\Vert_{\mathcal{L}\left( \mathbb{L}^{2}(D_{2}); \mathbb{L}^{2}(D_{1}) \right)} \, + \, \left\Vert  N \right\Vert_{\mathcal{L}\left( \mathbb{L}^{2}(D_{2}); \mathbb{L}^{2}(D_{1}) \right)}  \right], 
        \end{equation*}     
        hence, by knowing that $\left\Vert  N \right\Vert_{\mathcal{L}\left( \mathbb{L}^{2}(D_{2}); \mathbb{L}^{2}(D_{1}) \right)} \; \lesssim \; \left\Vert \nabla M \right\Vert_{\mathcal{L}\left( \mathbb{L}^{2}(D_{2}); \mathbb{L}^{2}(D_{1}) \right)} \, = \, 1$, we deduce 
             \begin{equation}\label{es-J4i}
             \left\vert J_{4, 3} \right\vert \, \lesssim \, a \, \left\lVert\overset{3}{\PP}\left(\tilde{W}_1\right)\right\rVert_{\mathbb{L}^2(B_1)} \left\lVert\overset{3}{\PP}\left(\tilde{E}_2\right)\right\rVert_{\mathbb{L}^2(B_2)}. 
            \end{equation}
        \end{enumerate}
        Then, by gathering $(\ref{es-J2i}), (\ref{es-J3}), (\ref{J42=0}), (\ref{es-J4i})$ and using Proposition \ref{prop-scoeff}, the error term given by $(\ref{Equa08093})$ becomes
	    \begin{equation}\label{Est-Err-1}
         Error_{1} \, = \, \mathcal{O}\left(a^{4}  \right)  \, + \, \mathcal{O}\left( \, a^4 \,  \left\lVert\overset{3}{\PP}\left(\tilde{E}_2\right)\right\rVert_{\mathbb{L}^2(B_2)}\right) \, + \, \mathcal{O}\left( a^8 \, d^{-2} \, \left\lVert\overset{1}{\PP}\left(\tilde{E}_2\right)\right\rVert_{\mathbb{L}^2(B_2)} \right) \, + \, \mathcal{O}\left(a^2 \left\lVert \overset{3}{\PP}\left(\tilde{E}_1\right)\right\rVert_{\mathbb{L}^2(B_1)}\right).
    \end{equation} 			
	 Besides, since $\overset{1}{\PP}\left(W_1\right)\in \mathbb{H}_0(\div=0)$, we have 
	 \begin{equation}\label{P1W1}
	 	\overset{1}{\PP}\left(W_1\right)=Curl (A_1)\quad\mbox{with}\quad \div(A_1)=0, \, \nu\times A_1=0. 
	 \end{equation}
  Then, by using $(\ref{P1W1})$ and the notations given by $(\ref{def-F_j})$, we rewrite $(\ref{LS-D1-split})$ as
	 		\begin{eqnarray}\label{LS-D1-change}
	 			\eta_1 \, \int_{D_{1}}\mathcal{P}(x, z_1)\cdot Curl (F_1)(x)\,dx \, &-& \, k^2 \, \eta_1 \, \eta_2  \, \int_{D_{1}}Curl (A_1)(x) \cdot \int_{D_{2}} \Phi_{k}(x, y) \, Curl (F_2)(y)\,dy\,dx\notag\\ \nonumber
	 			&=& \eta_1\int_{D_{1}}Curl(A_1)(x)\cdot E^{Inc}_1(x)\,dx \\ &+& \, k^2 \, \eta_1 \, \eta_2 \,  \, \int_{D_{1}} Curl(A_1)(x) \cdot \int_{D_{2}} \Phi_{k}(x, y) \; \overset{3}{\PP}\left(E_2\right)(y)\,dy\,dx + Error_1, 
	 		\end{eqnarray}
            with $Error_{1}$ being estimated by $(\ref{Est-Err-1})$. Next, we use integration by parts to rewrite \eqref{P1W1} term by term. 
    \begin{enumerate}
        \item[]
        \item For the first term on the L.H.S there holds, 
	 		\begin{equation}\label{L1}
	 			\int_{D_{1}}\mathcal{P}(x, z_1)\cdot Curl(F_1)(x)\,dx=\int_{D_{1}}Curl(\mathcal{P}(x, z_1))\cdot F_1(x)\,dx=\mathcal{K}\cdot \int_{D_{1}} F_1(x)\,dx,
	 		\end{equation}
	 		where $\mathcal{K}$ is a constant tensor defined by
	 		\begin{equation}\label{def-K}
	 			\mathcal{K}:= \underset{x}{Curl}\left(\mathcal{P}(x, z_1)\right)=\left(\begin{array}{ccc}
	 			0 & 0 & 0 \\ 
	 			0 & 0 & -1 \\ 
	 			0 & 1 & 0 \\ 
	 			0 & 0 & 1 \\ 
	 			0 & 0 & 0 \\ 
	 			1 & 0 & 0 \\ 
	 			0 & -1 & 0 \\ 
	 			1 & 0 & 0 \\ 
	 			0 & 0 & 0
	 			\end{array} 
	 			\right).
	 		\end{equation}
	 		\item[] 
	 		\item For the second term on the L.H.S there holds, 
	 		\begin{eqnarray}\label{L-2}
                    \nonumber
	 			\cdots &:=& \int_{D_{1}}Curl (A_1)(x) \cdot \int_{D_{2}} \Phi_{k}(x, y) \, Curl (F_2)(y)\,dy\,dx \\ \nonumber &=& \, \int_{D_{1}} A_1(x) \cdot \underset{x}{Curl} \left(\int_{D_{2}} \Phi_{k}(x, y) \, Curl\left(F_2\right)(y)\,dy\right)\,dx \\ \nonumber
                      &=& \, \int_{D_{1}} A_1(x) \cdot \underset{x}{Curl} \left(\int_{D_{2}}  \underset{x}{Curl}\left(\Phi_{k}(x, y) \, F_2(y)\right)\,dy\right)\,dx \\
                      &=& k^2\, \int_{D_{1}} A_1(x) \cdot \int_{D_{2}} \Upsilon_k(x, y) \cdot F_2(y) \, dy \, dx. 
	 		\end{eqnarray}
   \item[]
   \item For the first term on the R.H.S there holds,
	 		\begin{equation}\label{R-1}
	 		\int_{D_1}Curl(A_1)(x)\cdot E^{Inc}_1(x)\,dx=\int_{D_{1}}A_1(x)\cdot Curl(E_1^{Inc})(x) \overset{(\ref{EincHinc})}{=} \, i \, k \, \int_{D_{1}} A_1(x) \cdot H_1^{Inc}(x)\,dx.
	 		\end{equation}
   \item[]
   \item For the second term on the R.H.S there holds, 
	 		\begin{eqnarray}\label{R-2}
	 			\int_{D_{1}} Curl(A_1)(x) \cdot \int_{D_{2}} \Phi_{k}(x, y) \, \overset{3}{\PP}\left(E_2 \right)(y)\,dy\,dx&=& \int_{D_1} A_1(x) \cdot \underset{x}{Curl} \left(\int_{D_{2}}\Phi_{k}(x, y) \, \overset{3}{\PP}\left(E_2\right)(y)\,dy\right)\,dx\notag\\
	 			&=&\int_{D_{1}} A_1(x) \cdot \int_{D_{2}}\underset{x}{\nabla}\Phi_{k}(x, y)\times \overset{3}{\PP}\left(E_2\right)(y)\,dy\,dx.
	 		\end{eqnarray}
    \end{enumerate}
	 		Then, using \eqref{L1}, \eqref{L-2},  \eqref{R-1} and \eqref{R-2}, we can rewrite \eqref{LS-D1-change} as 
	 		\begin{eqnarray*}\label{LS-mid1}
	 			 \eta_1 \, \mathcal{K}\cdot\int_{D_{1}} F_1(x)\,dx \, &-& \, k^4 \, \eta_1 \, \eta_2 \,   \int_{D_{1}} A_1(x) \cdot \int_{D_{2}} \Upsilon_k(x, y) \cdot F_2(y)\,dy\,dx\notag\\ \nonumber
	 			&=& i \, \eta_1 \, k \, \int_{D_{1}} A_1(x) \cdot H^{Inc}_1(x)\,dx \\  &+& \, k^2 \, \eta_1 \, \eta_2 \, \int_{D_{1}} A_1(x) \cdot \int_{D_{2}}\underset{x}{\nabla}\Phi_{k}(x, y)\times\overset{3}{\PP}\left(E_2\right)(y)\,dy\,dx+ Error_1.
	 		\end{eqnarray*}
	 		By the use of Taylor expansions for $H_1^{Inc}(\cdot), \underset{x}{\nabla} \Phi_{k}(\cdot, \cdot)$ and $ \Upsilon_{k}(\cdot, \cdot)$, we obtain
	 		\begin{eqnarray}\label{AL-prime1} \nonumber
	 		\eta_1 \, \mathcal{K} \cdot \int_{D_{1}} F_1(x)\,dx \, &-& \, \eta_1 \, \eta_2 \, k^4 \, \int_{D_{1}} A_1(x)\,dx \cdot \Upsilon_k(z_1, z_2) \cdot \int_{D_2}F_2(y)\,dy\notag\\ \nonumber
	 		&=& i \, \eta_1 \, k \, \int_{D_{1}} A_1(x)\,dx \cdot H^{Inc}_1(z_1) \\ \nonumber &+& \, \eta_1 \, \eta_2 \, k^2 \, \int_{D_{1}} A_1(x)\,dx \cdot \underset{x}{\nabla}\Phi_k(z_1, z_2)\times \int_{D_{2}}\overset{3}{\PP}\left(E_2\right)(y)\,dy \\ &+& Error_1 + Error_2,
	 		\end{eqnarray}
	 		where $Error_2$ admits the following estimation
            \begin{eqnarray*}
                Error_{2} \, &=& \, \mathcal{O}\left( a^{5} \, d^{-4} \, \left\Vert \tilde{F_{2}} \right\Vert_{\mathbb{L}^{2}(B_{2})} \, \left\Vert \tilde{A_{1}} \right\Vert_{\mathbb{L}^{2}(B_{1})} \right) \, + \, \mathcal{O}\left( \ a^{2} \, \left\Vert \tilde{A_{1}} \right\Vert_{\mathbb{L}^{2}(B_{1})} \right) \\
                &+& \, \mathcal{O}\left(a^{5} \, d^{-3} \, \left\Vert \overset{3}{\PP}\left(\tilde{E_2}\right) \right\Vert_{\mathbb{L}^{2}(B_{2})} \, \left\Vert \tilde{A_{1}} \right\Vert_{\mathbb{L}^{2}(B_{1})} \right).
            \end{eqnarray*}
            Besides, thanks to Friedrichs Inequality, see \cite[Corollary 3.2]{S-F-Ineq}, we know that 
            \begin{equation}\label{FI}
                \left\Vert \tilde{F_{2}} \right\Vert_{\mathbb{L}^{2}(B_{2})} \; \lesssim \; \left\Vert Curl\left( \tilde{F_{2}} \right) \right\Vert_{\mathbb{L}^{2}(B_{2})} \quad \text{and} \quad  \left\Vert \tilde{A_{1}} \right\Vert_{\mathbb{L}^{2}(B_{1})} \; \lesssim \; \left\Vert Curl\left( \tilde{A_{1}} \right) \right\Vert_{\mathbb{L}^{2}(B_{1})},
            \end{equation}
            which, by scaling back the rotational operator, gives us
            \begin{equation}\label{F2aP1E2}
                \left\Vert \tilde{F_{2}} \right\Vert_{\mathbb{L}^{2}(B_{2})} \;  \lesssim  \; a \, \left\Vert \widetilde{Curl\left( F_{2} \right) }\right\Vert_{\mathbb{L}^{2}(B_{2})} \, \overset{(\ref{def-F_j})}{=} \; a \, \left\Vert \overset{1}{\mathbb{P}}(\widetilde{E}_2)\right\Vert_{\mathbb{L}^{2}(B_{2})},
            \end{equation} 
            and
            \begin{equation*}   
  \left\Vert \tilde{A_{1}} \right\Vert_{\mathbb{L}^{2}(B_{1})} \;  \lesssim  \; a \, \left\Vert \widetilde{Curl\left( A_{1} \right) }\right\Vert_{\mathbb{L}^{2}(B_{1})} \; \overset{(\ref{P1W1})}{=} \; a \; \left\Vert 
	 	\overset{1}{\PP}\left(\tilde{W}_1\right) \right\Vert_{\mathbb{L}^{2}(B_{1})}.
            \end{equation*}
            Hence, we deduce that 
	 		\begin{eqnarray*}
	 			Error_2 \, &=& \, \O\left( a^7 d^{-4}\left\lVert\overset{1}{\PP}\left(\tilde{W}_1\right)\right\rVert_{\mathbb{L}^2(B_1)}\left\lVert\overset{1}{\PP}\left(\tilde{E}_2\right)\right\rVert_{\mathbb{L}^2(B_2)}\right) \, + \, \mathcal{O}\left(  \, a^3 \,  \left\lVert\overset{1}{\PP}\left(\tilde{W}_1\right)\right\rVert_{\mathbb{L}^2(B_1)} \right) \\
	 			&+& \O\left( a^6 d^{-3} \left\lVert\overset{1}{\PP}\left(\tilde{W}_1\right)\right\rVert_{\mathbb{L}^2(B_1)}\left\lVert\overset{3}{\PP}\left(\tilde{E}_2\right)\right\rVert_{\mathbb{L}^2(B_2)}\right),
	 		\end{eqnarray*}
    which, by using Proposition \ref{prop-scoeff}, can be reduced to 
    	 		\begin{equation}\label{Est-Err-2}
	 			Error_2 \, = \, \O\left( a^{8-h} \, d^{-4} \, \left\lVert\overset{1}{\PP}\left(\tilde{E}_2\right)\right\rVert_{\mathbb{L}^2(B_2)}\right) \, + \, \mathcal{O}\left( a^{4-h} \right) 
	 			\, + \, \O\left( a^{7-h} \, d^{-3} \, \left\lVert\overset{3}{\PP}\left(\tilde{E}_2\right)\right\rVert_{\mathbb{L}^2(B_2)}\right).
	 		\end{equation}
In addition, by knowing that 
	 		\begin{equation}\notag
	 			\frac{1}{4}\left(\mathcal{K}^\perp\cdot\mathcal{K}\cdot\mathcal{K}^\perp\right)\cdot\mathcal{K}=I,
	 		\end{equation}
	 		setting  
	 		\begin{equation}\label{def-A1-mu}
	 	 \bm{\mathcal{A}_1^{(1)}}:=\frac{1}{4}\left(\mathcal{K}^\perp\cdot\mathcal{K}\cdot\mathcal{K}^\perp\right)\cdot\int_{D_{1}} A_1(x)\,dx \, \in \, \mathbb{C}^{3\times 3},
	 		\end{equation}
    and using $(\ref{Qj1})$ 
    we rewrite $(\ref{AL-prime1})$, after multiplication of its both sides by $\frac{1}{4}\left(\mathcal{K}^\perp\cdot\mathcal{K}\cdot\mathcal{K}^\perp\right)$, as 
	 		\begin{equation}\label{la-Q1}
	 			Q_{1} \, - \, k^2 \, \eta_1 \, \bm{\mathcal{A}_1^{(1)}} \cdot \left(k^2 \, \Upsilon_{k}(z_1, z_2)\cdot Q_{2} +\underset{x}{\nabla}\Phi_k(z_1, z_2) \times R_{2} \right) \,
	 			= \, i \, k \, \eta_1 \, \bm{\mathcal{A}_1^{(1)}} \cdot H^{Inc}_1 (z_1) \, + \, Error_1 \, + \, Error_2.
	 		\end{equation}
	 		Moreover, following a similar argument to \cite[Appendix A.1, Formula (A.30)]{CGS}, we can know that
	 	\begin{equation}\label{express-A1-mu}
	 	\bm{\mathcal{A}_1^{(1)}} \; = \; \frac{a^{5-h}}{\pm \, c_0} \; {\bf P}_{0, 1}^{(1)} \;  + \; \mathcal{O}\left(a^5\right),
	 		\end{equation}
    where ${\bf P}_{0, 1}^{(1)}$ is the tensor given by 
    \begin{equation}\notag
	 			{\bf P}_{0, 1}^{(1)} \, := \, \int_{B_1}\phi_{n_{0}, B_{1}}(y)\,dy \, \otimes \, \int_{B_1}\phi_{n_{0}, B_{1}}(y)\,dy. 
	 		\end{equation}
    Then, by plugging $(\ref{express-A1-mu})$ into $(\ref{la-Q1})$, we deduce that 
    \begin{equation}\label{Equa1024}
	  Q_{1} \, - \, k^2 \, \eta_1 \, \frac{a^{5-h}}{\pm \, c_0} \; {\bf P}_{0, 1}^{(1)} \cdot \left(k^2 \, \Upsilon_{k}(z_1, z_2)\cdot Q_{2} +\underset{x}{\nabla}\Phi_k(z_1, z_2) \times R_{2} \right) \, = \, i \, k \, \eta_1 \, \frac{a^{5-h}}{\pm \, c_0} \; {\bf P}_{0, 1}^{(1)} \cdot H^{Inc}_1 (z_1) \, + Error_{3}, 
   \end{equation}
   where 
   \begin{equation*}
       Error_{3} \, := \, \mathcal{O}\left(k^{2} \, \eta_{1} \, a^{5} \, \left(k^2 \, \Upsilon_{k}(z_{1},z_{2}) \cdot Q_{2} \, + \, \nabla \Phi_{k}(z_{1},z_{2}) \times R_{2} \right)  \right) \, + \, \mathcal{O}\left(k \, \eta_{1} \, a^{5} \right) \, + \, Error_1 \, + \, Error_2.
   \end{equation*}
   To estimate the above term, we start by taking the modulus on its both sides and we use $(\ref{def-eta}), (\ref{Qj1})$ and $(\ref{F2aP1E2})$, to obtain
      \begin{eqnarray*}
       \left\vert Error_{3} \right\vert \, & \lesssim &    \, a^{7} \, d^{-3} \,  \left\Vert \overset{1}{\PP}\left(\tilde{E}_2\right) \right\Vert_{\mathbb{L}^{2}(B_{2})} \, +  \, a^{6} \, d^{-2} \, \left\Vert \overset{3}{\PP}\left(\tilde{E}_2\right) \right\Vert_{\mathbb{L}^{2}(B_{2})} \, +  \, a^{3} \, + \, \left\vert Error_1 \right\vert \, + \, \left\vert Error_2 \right\vert \\ 
       & \overset{(\ref{Est-Err-1})}{\underset{(\ref{Est-Err-2})}{\lesssim}} & \resizebox{.75\hsize}{!}{$ a^{\min\left(7-3t; 8-h-4t\right)}  \left\Vert \overset{1}{\PP}\left(\tilde{E}_2\right) \right\Vert_{\mathbb{L}^{2}(B_{2})} +  a^{\min(4;7-h-3t)} \, \left\Vert \overset{3}{\PP}\left(\tilde{E}_2\right) \right\Vert_{\mathbb{L}^{2}(B_{2})} + a^{2}  \left\Vert \overset{3}{\PP}\left(\tilde{E}_1\right) \right\Vert_{\mathbb{L}^{2}(B_{2})} +  a^{3} $} 
   \end{eqnarray*}
To finish with the estimation of $Error_{3}$, we use the a-prior estimates presented in Proposition \ref{lem-es-multi}, under the conditions $t \, < \, 1$ and $h \, < \, 1$, we can obtain 
    \begin{equation*}
        Error_{3} \; = \; \mathcal{O}\left(a^{\min(3; 7-2h-3t)} \right).
    \end{equation*} 
    Regarding the estimation of $Error_{3}$ and the scale of $\eta_{1}$, see $(\ref{def-eta})$, the R.H.S of \eqref{Equa1024} is of order $a^{3-h}$ that dominates. This justifies the first equation in \eqref{eq-al-D1}.
      \item[]  		
	 \item The algebraic system with respect to $\overset{3}{\PP}\left(E_1\right)$. \\
      By integrating the equation $(\ref{LS-D1-2})$, over $D_1$, using $(\ref{def-Vm}), (\ref{grad-M-1st})$ and $(\ref{RefNeeded1})$, we obtain 
	 		\begin{eqnarray*}
	 	\eta_1\int_{D_{1}}\overset{3}{\PP}\left(E_1\right)(x)\,dx \, &-& \, k^2 \, \eta_1 \, \eta_2 \, \int_{D_{1}}V_1(x) \cdot \int_{D_{2}}\Upsilon_k(x, y) \cdot \overset{3}{\PP}\left(E_2\right)(y)\,dy\,dx\notag\\
	 	&=&\eta_1 \, \int_{D_{1}}V_1(x) \cdot E_1^{Inc}(x)\,dx \, +  \, k^2 \, \eta_1 \, \eta_2 \, \int_{D_{1}}V_1(x) \cdot \int_{D_{2}}\Phi_{k}(x, y) \, \overset{1}{\PP}\left(E_2\right)(y)\,dy\,dx.
	 		\end{eqnarray*}
	 	Then, by taking the Taylor expansion of $\Upsilon_k(\cdot, \cdot)$ at $(z_1, z_2)$, we obtain
	 		\begin{eqnarray}\label{eq-P31-mid20}
	 			\eta_1 \, \int_{D_{1}}\overset{3}{\PP}\left(E_1\right)(x)\,dx  &-& k^2\, \eta_1 \, \eta_2 \, \int_{D_{1}}V_1(x)\,dx\cdot\Upsilon_k(z_1, z_2) \cdot \int_{D_{2}}\overset{3}{\PP}\left(E_2\right)(y)\,dy\notag\\ \nonumber
	 			&=& \eta_1 \, \int_{D_{1}} V_1(x) \cdot E_1^{Inc}(x)\,dx \, + \, k^2 \, \eta_1 \, \eta_2 \, \int_{D_{1}} V_1(x) \cdot \int_{D_{2}}\Phi_{k}(x, y)\overset{1}{\PP}\left(E_2\right)(y)\,dy\,dx \\ &+& \, k^2 \, J_{6} \, +\, k^2 \, J_{7},
	 		\end{eqnarray}
	 		where
	 		\begin{eqnarray}
	 			J_6 \, &:=& \,  \eta_1\eta_2\int_{D_{1}} V_1(x) \cdot \int_{D_{2}}\int_0^1 \underset{x}{\nabla}(\Upsilon_k)(z_1 + t(x-z_1), y)\cdot\mathcal{P}(x, z_1)\,dt\cdot\overset{3}{\PP}\left(E_2\right)(y)\,dy\,dx,\notag\\
	 			J_7 \, &:=& \, \eta_1\eta_2\int_{D_{1}} V_1(x) \cdot \int_{D_{2}}\int_0^1 \underset{y}{\nabla}(\Upsilon_k)(x, z_2+t(y-z_2))\cdot\mathcal{P}(y, z_2)\,dt\cdot\overset{3}{\PP}\left(E_2\right)(y)\,dy\,dx.\notag
	 		\end{eqnarray}
    Next, we estimate $J_{6}$ and $J_{7}$.
    \begin{enumerate} 
        \item[]
        \item Estimation of $J_{6}$. 
        \begin{eqnarray}\label{es-J6}
	 	\left\vert J_{6} \right\vert \, & \lesssim & \, \left\vert \eta_1 \right\vert \, \left\vert \eta_2 \right\vert \left\lVert V_1\right\rVert_{\mathbb{L}^2(D_1)} \, \left\lVert\int_{D_{2}}\int_0^1 \underset{x}{\nabla}(\Upsilon_k)(z_1 + t(\cdot-z_1), y)\cdot\mathcal{P}(\cdot, z_1)\,dt\cdot\overset{3}{\PP}\left(E_2\right)(y)\,dy\right\rVert_{\mathbb{L}^2(D_1)}\notag\\
       & \overset{(\ref{def-eta})}{\lesssim} & a^5 \, d^{-4} \,  \left\lVert\tilde{V}_1\right\rVert_{\mathbb{L}^2(B_1)} \,  \left\lVert\overset{3}{\PP}\left(\tilde{E}_2\right)\right\rVert_{\mathbb{L}^2(B_2)}.
	 		\end{eqnarray}
        \item[]
        \item Estimation of $J_{7}$. \\
        Similarly, as $J_{7}$'s expression is almost identical to $J_{6}$'s expression, we deduce from $(\ref{es-J6})$, 
	 		\begin{equation}\label{es-J7}
	 			 J_7  \, = \, \mathcal{O}\left( a^5 \, d^{-4} \,  \left\lVert\tilde{V}_1\right\rVert_{\mathbb{L}^2(B_1)} \, \left\lVert\overset{3}{\PP}\left(\tilde{E}_2\right)\right\rVert_{\mathbb{L}^2(B_2)} \right).
	 		\end{equation}
    \end{enumerate}
    In addition to the estimation of $J_{6}$ and $J_{7}$, we need to analyze the second term on the R.H.S of \eqref{eq-P31-mid20}, that we denote by $J_{5}$. For $J_5$, by using the expression \eqref{def-F_j}, there holds
	 		\begin{equation*}\label{J5}
	 			\resizebox{.95\hsize}{!}{$J_{5} \, = \, k^2 \, \eta_1 \, \eta_2 \int_{D_{1}} V_1(x) \cdot \int_{D_{2}}\Phi_{k}(x, y) \, \underset{y}{Curl}\left(F_2\right)(y)\,dy\,dx \, = \,k^2 \,  \eta_1 \, \eta_2 \, \int_{D_{1}}V_1(x) \cdot \int_{D_{2}}\underset{x}{\nabla}\Phi_{k}(x, y)\times F_2(y)\,dy\,dx.$} 
	 		\end{equation*}
	 	 		By further expanding $\nabla \Phi_{k}(\cdot, \cdot)$ at $(z_1, z_2)$, we obtain  
    \begin{equation}\label{J5}
	 			J_{5} \, = \, k^2 \, \eta_1 \, \eta_2 \int_{D_{1}} V_1(x) \, dx \, \cdot \underset{x}{\nabla} \Phi_{k}(z_{1}, z_{2}) \times \int_{D_{2}} \, F_2(y)\,dy\, + \, k^2 \, J_{8} \, + \, k^2 \, J_{9},
	 		\end{equation}
    where 
	 		\begin{eqnarray}
	 			J_8&:=& \eta_1 \eta_2 \int_{D_{1}} V_1(x) \cdot \int_{D_{2}}\int_0^1 \underset{x}{\nabla}\left(\underset{x}{\nabla}\Phi_k\right)\left(z_1 + t(x-z_1), z_2\right)\cdot (x-z_1)\,dt\times F_2(y)\,dy\,dx,\notag\\
	 			J_9&:=& \eta_1 \eta_2 \int_{D_{1}} V_1(x) \cdot \int_{D_{2}}\int_0^1 \underset{y}{\nabla}\left(\underset{x}{\nabla}\Phi_k\right)\left(z_1, z_2+ t(y-z_2)\right)\cdot (y-z_2)\,dt\times F_2(y)\,dy\,dx.\notag
	 		\end{eqnarray}
    Next, we estimate $J_{8}$ and $J_{9}$.
	\begin{enumerate}
	    \item[]
        \item Estimation of $J_{8}$. \\
        \begin{eqnarray}\label{es-J8}
	 			\left\vert J_8 \right\vert & \lesssim & \left\vert \eta_1 \right\vert \, \left\vert \eta_2 \right\vert \left\lVert V_1\right\rVert_{\mathbb{L}^2(D_1)}\, \left\lVert \int_{D_{2}}\int_0^1 \underset{x}{\nabla}\left(\underset{x}{\nabla}\Phi_k\right)\left(z_1 + t(\cdot-z_1), z_2\right)\cdot (\cdot-z_1)\,dt\times F_2(y)\,dy\right\rVert_{\mathbb{L}^2(D_1)}\notag\\ 
	 	  & \overset{(\ref{def-eta})}{\lesssim} & \, a^{2} \, d^{-3} \, \left\lVert V_1\right\rVert_{\mathbb{L}^2(D_1)} \, d^{-3} \, \left\lVert F_2\right\rVert_{\mathbb{L}^2(D_2)} \, = \, \mathcal{O}\left( a^6 \, d^{-3} \, \left\lVert \tilde{V}_1\right\rVert_{\mathbb{L}^2(B_1)} \, \left\lVert\overset{1}{\PP}\left(\tilde{E}_2\right)\right\rVert_{\mathbb{L}^2(B_2)} \right),
	 		\end{eqnarray}
    where the last estimation is justified by the use of Friedrichs Inequality, see $(\ref{FI})$. 
        \item[]
        \item Estimation of $J_{9}$. \\
        Similarly, as $J_{9}$'s expression is almost identical to $J_{8}$'s expression, we deduce, from $(\ref{es-J8})$, 
	 		\begin{equation*}\label{es-J9}
	 		J_9 \, = \, \O\left(a^6 \,  d^{-3} \, \left\lVert \tilde{V}_1\right\rVert_{\mathbb{L}^2(B_1)} \, \left\lVert\overset{1}{\PP}\left(\tilde{E}_2\right)\right\rVert_{\mathbb{L}^2(B_2)}\right).
	 		\end{equation*}
	\end{enumerate} 	
 Hence, the term $J_{5}$, given by $(\ref{J5})$, becomes,  
     \begin{equation}\label{J5=}
	 			J_{5} \, =\, k^2 \, \eta_1 \, \eta_2 \int_{D_{1}} V_1(x) \, dx \, \cdot \underset{x}{\nabla} \Phi_{k}(z_{1}, z_{2}) \times \int_{D_{2}} \, F_2(y)\,dy\, + \, \mathcal{O}\left(a^6 \,  d^{-3} \, \left\lVert \tilde{V}_1\right\rVert_{\mathbb{L}^2(B_1)} \, \left\lVert\overset{1}{\PP}\left(\tilde{E}_2\right)\right\rVert_{\mathbb{L}^2(B_2)}\right).
	 		\end{equation}
    Consequently, by using $(\ref{es-J6}), (\ref{es-J7})$ and $(\ref{J5=})$, the equation $(\ref{eq-P31-mid2})$ takes the following form,  
    \begin{eqnarray}\label{eq-P31-mid2}
    \nonumber
	 \eta_1 \, \int_{D_{1}}\overset{3}{\PP}\left(E_1\right)(x)\,dx \, &-& \, k^2 \, \eta_1 \, \eta_2 \, \bm{\mathcal{A}^{(2)}_1} \cdot \left[ \Upsilon_k(z_1, z_2) \cdot \int_{D_{2}}\overset{3}{\PP}\left(E_2\right)(y)\,dy \, + \, \underset{x}{\nabla} \Phi_{k}(z_{1}, z_{2}) \times \int_{D_{2}} \, F_2(y)\,dy \,\right] \\
        &=&   \, \eta_1 \, \bm{\mathcal{A}^{(2)}_1} \cdot E_1^{Inc}(z_{1}) \, + \, Error^{(2)}_1, 
    \end{eqnarray}
    where, we have used a Taylor expansion for the incident field to obtain,  
    \begin{eqnarray}\label{J10}
    \nonumber
    Error^{(2)}_1 &:=&
         \eta_1\int_{D_{1}}V_1(x) \cdot \int_0^1 \nabla E_1^{Inc}(z_1 + t(x-z_1))\cdot(x-z_1)\,dt\,dx \\ &+& \, \mathcal{O}\left(a^6 \,  d^{-3} \, \left\lVert \tilde{V}_1\right\rVert_{\mathbb{L}^2(B_1)} \, \left\lVert\overset{1}{\PP}\left(\tilde{E}_2\right)\right\rVert_{\mathbb{L}^2(B_2)}\right) \, + \, \mathcal{O}\left( a^5 \, d^{-4} \,  \left\lVert\tilde{V}_1\right\rVert_{\mathbb{L}^2(B_1)} \, \left\lVert\overset{3}{\PP}\left(\tilde{E}_2\right)\right\rVert_{\mathbb{L}^2(B_2)} \right),
    \end{eqnarray}
    and we have denoted by $\bm{\mathcal{A}^{(2)}_1}$ the following notation 
    \begin{equation}\notag
	 	\bm{\mathcal{A}^{(2)}_1} \; := \; \int_{D_{1}} V_1(x)\,dx.
    \end{equation}
    Obviously, the first term on the R.H.S of $(\ref{J10})$, admits the following estimation
    \begin{equation}\label{es-J10}
	 J_{10} \; = \; \eta_1 \, \int_{D_{1}}V_1(x) \cdot \int_0^1 \nabla E_1^{Inc}(z_1 + t(x-z_1))\cdot(x-z_1)\,dt\,dx \; = \; \mathcal{O}\left( 
 \, a^2 \, \left\lVert\tilde{V}_1\right\rVert_{\mathbb{L}^2(B_1)} \right).
	 		\end{equation}

	 		Now, for $\bm{\mathcal{A}^{(2)}_1}$, after scaling from $D_{1}$ to $B_{1}$, the following expansion holds, 
	 	\begin{equation*}
	 		\bm{\mathcal{A}_1^{(2)}} \, = \,  a^3 \int_{B_1} \tilde{V}_1 (x) \,dx\, \overset{(\ref{RefNeeded1})}{=}\, a^3 \sum_n \left\langle \tilde{V}_{1}, e_{n, B_{1}}^{(3)} \right\rangle_{\mathbb{L}^{2}(B_{1})} \otimes \int_{B_1} e_{n, B_{1}}^{(3)}(x)\,dx, 
            \end{equation*}
            which, by using $(\ref{V1-en3})$, becomes,
    \begin{eqnarray*}
	 \bm{\mathcal{A}_1^{(2)}} \, &=& \,  a^3 \, \sum_n \frac{1}{\left( 1 \, + \, \eta_1 \, \lambda_{n}^{(3)}(B_{1})\right)}\int_{B_1} e_{n, B_{1}}^{(3)}(x)\,dx  \otimes\int_{B_1} e_{n, B_{1}}^{(3)}(x)\,dx \\
     &+& \,  a^3 \, \eta_1 \, \sum_n  \frac{}{\left( 1 \, + \, \eta_1 \, \lambda_{n}^{(3)}(B_{1})\right)}\left\langle \tilde{V}_1, \mathcal{S}^{ka}\left(e_{n, B_{1}}^{(3)}\right) \right\rangle_{\mathbb{L}^{2}(B_{1})} \otimes\int_{B_1} e_{n, B_{1}}^{(3)}(x)\,dx,
	 		\end{eqnarray*}
	 		where $\mathcal{S}^{ka}(e_{n,B_{1}}^{(3)})$ is given by \eqref{def-Ska}. Then, thanks to $(\ref{EstSka})$ and $(\ref{es-V1})$, we can estimate the second term on the R.H.S to obtain
	       \begin{equation*}
	        \bm{\mathcal{A}_1^{(2)}} \, \overset{\eqref{es-V12}}{=} \, a^3 \, \sum_n\frac{1}{\left( 1 \, + \, \eta_1 \, \lambda_{n}^{(3)}(B_{1})\right)}\int_{B_1} e_{n,B_{1}}^{(3)}(x)\,dx\otimes \int_{B_1} e_{n,B_{1}}^{(3)}(x)\,dx \, + \, \mathcal{O}\left( k^{2} \, a^{7} \right),
	       \end{equation*}		
	 		and, by multiplying its both sides by $\eta_{1}$ and using $\eta_{1}$'s expression, given by $(\ref{def-eta})$, we obtain    
	 		\begin{eqnarray}\label{Equa1124}
	 	\eta_1 \, \bm{\mathcal{A}_1^{(2)}} \, 
	 			&=& \eta_0 \, a^3 \, \sum_n \frac{1}{\left( \eta_0 \, \lambda_{n}^{(3)}(B_{1}) \, + \, a^2 \right)}\int_{B_1} e_{n,B_{1}}^{(3)}(x)\,dx\otimes \int_{B_1} e_{n,B_{1}}^{(3)}(x)\,dx + \mathcal{O}\left(k^{2} \, a^{5} \right)\notag\\
	 			&=& \, a^3 \, \sum_n \frac{1}{\lambda_{n}^{(3)}(B_{1})}\int_{B_1} e_{n,B_{1}}^{(3)}(x)\,dx\otimes \int_{B_1} e_{n,B_{1}}^{(3)}(x)\,dx\notag\\
	 			&-& a^{3} \, \sum_n \frac{a^{2}}{\lambda_{n}^{(3)}(B_{1}) \, \left(\eta_0 \, \lambda_{n}^{(3)}(B_{1}) \, + \, a^{2} \right)} \, \int_{B_1} e_{n,B_{1}}^{(3)}(x)\,dx\otimes \int_{B_1} e_{n,B_{1}}^{(3)}(x)\,dx +\O\left(k^{2} \, a^{5}\right).
	 		\end{eqnarray}
	 	Clearly, the second term on the R.H.S can be estimated as, 
	 		\begin{eqnarray*}
	 			\left\vert   \sum_n \frac{ a^{3} \, a^{2}}{\lambda_{n}^{(3)}(B_{1}) \, \left(\eta_{0} \, \lambda_{n}^{(3)}(B_{1}) \, + \, a^{2} \right)}\int_{B_1} e_{n,B_{1}}^{(3)}(x)\,dx \otimes \int_{B_1} e_{n,B_{1}}^{(3)}(x)\,dx \right\vert \, &  \lesssim & \, a^5 \, \sum_n \left|\left\langle I , e_{n,B_{1}}^{(3)}\right\rangle_{\mathbb{L}^{2}(B_{1})}\right|^2 \\  &=& \, \mathcal{O}\left(a^5\right).
	 		\end{eqnarray*}
	 	Hence, by plugging the above estimation into $(\ref{Equa1124})$, we deduce 
        \begin{equation}\label{eta1A12=}
	 	\eta_1 \, \bm{\mathcal{A}_1^{(2)}} \, = \,  \, a^3 \, {\bf P}_{0, 1}^{(2)} \, + \, \mathcal{O}\left(a^5\right),
	 		\end{equation}
        where
    	 		\begin{equation}\notag
	 			{\bf P}_{0, 1}^{(2)} \, := \, \sum_n \frac{1}{\lambda_{n}^{(3)}(B_{1})} \, \int_{B_1} e_{n,B_{1}}^{(3)}(x)\,dx \otimes \int_{B_1} e_{n,B_{1}}^{(3)}(x)\,dx.
	 		\end{equation}
    This implies, by returning to $(\ref{eq-P31-mid2})$ and using $(\ref{eta1A12=})$, together with the notations $(\ref{Qj1})$, 
        \begin{equation*}\label{R1-proof}
	 R_{1} \, - \, k^2  \, a^{3} \, {\bf P}_{0, 1}^{(2)} \cdot \left[ \Upsilon_k(z_1, z_2) \cdot R_{2} \, + \, \underset{x}{\nabla} \Phi_{k}(z_{1}, z_{2}) \times Q_{2} \,\right] \, =   \, a^{3} \,  {\bf P}_{0, 1}^{(2)} \cdot E_1^{Inc}(z_{1}) \, + \, Error^{(2)}_{2}, 
    \end{equation*}
    where $Error^{(2)}_{2}$ being given by
    \begin{equation*}
        Error^{(2)}_{2} \, := \, Error^{(2)}_{1} \, + \, \mathcal{O}\left(a^{5} \, \left[ \Upsilon_k(z_1, z_2) \cdot R_{2} \, + \, \underset{x}{\nabla} \Phi_{k}(z_{1}, z_{2}) \times Q_{2} \,\right] \right) \, + \, \mathcal{O}\left(a^{5} \right),
    \end{equation*}
    where $Error^{(2)}_{1}$ is defined by $(\ref{J10})$.
    We need to estimate $Error^{(2)}_{2}$. To do this, by taking the modulus on the both sides we obtain 
    \begin{equation*}
       \left\vert Error^{(2)}_{2} \right\vert \,  \lesssim  \, \left\vert Error^{(2)}_{1} \right\vert \, + \, a^{8} \, \left[ d^{-3} \, \left\Vert \overset{3}{\PP}\left(\tilde{E}_2\right) \right\Vert_{\mathbb{L}^{2}(B_{2})} \,  + \, d^{-2} \, \left\Vert \tilde{F}_2 \right\Vert_{\mathbb{L}^{2}(B_{2})} \,\right]  \, + \,  a^{5},
       \end{equation*}
       hence, by using $(\ref{es-J10}), (\ref{J10}), (\ref{F2aP1E2})$ and the a-priori estimates presented by Proposition \ref{lem-es-multi} and Proposition \ref{prop-scoeff}, we deduce  
       \begin{equation*}
         Error^{(2)}_{2}  \, 
        = \, \mathcal{O}\left(a^4 \right) \, + \, \mathcal{O}\left( a^{7-h} \, d^{-4} \right).
    \end{equation*}
            \end{enumerate}
 	\bigskip
 	To investigate the linear algebraic system in $D_2$, we multiply each side of $(\ref{LS-es})$ by $\eta_2$, and take the inverse of the operator $\TT_{k ,2}$, to get
    
	 	\begin{equation}\label{LS-D2-3}
	 	\eta_2 \, E_2(x) \, - \, k^2\, \eta_2 \, \eta_1 \, \TT_{k, 2}^{-1}\int_{D_{1}}\Upsilon_{k}(x, y) \cdot E_1(y)\,dy \, = \, \eta_2 \, \TT_{k, 2}^{-1}\left(E_2^{Inc}\right)(x), \quad x \in D_{2}.
	 	\end{equation}
	\begin{enumerate}
	    \item[]
        \item The algebraic system with respect to $\overset{1}{\PP}\left(E_2\right)$.

	 	Then, by multiplying the both sides of \eqref{LS-D2-3} by $\overset{1}{\PP}(\mathcal{P}(x, z_2))$, with $x \in D_{2}$, integrating over $D_2$ and using $(\ref{AI1} )$, we get  	 
	 	\begin{equation*}\label{eq-D2-mid2}
	 	\eta_2 \, \int_{D_{2}}\overset{1}{\PP}\left(\mathcal{P}(x, z_2)\right) \cdot E_2(x)\,dx \, - \, k^2 \, \eta_2 \, \eta_1 \, \int_{D_{2}}W_2(x) \cdot \int_{D_{1}}\Upsilon_{k}(x, y) \cdot  E_1(y)\,dy\,dx \, = \, \eta_2 \, \int_{D_{2}}W_2(x) \cdot E_2^{Inc}(x)\,dx,
	 	\end{equation*}
   which can be rewritten as
	 	\begin{eqnarray}\label{LS-D2-P1-main}
	 		\eta_2  \int_{D_{2}}\mathcal{P}(x, z_2) \cdot \overset{1}{\PP}\left(E_2\right)(x) dx &-&  k^2 \eta_2  \eta_1  \int_{D_{2}}\overset{1}{\PP}\left(W_2\right)(x) \cdot \int_{D_{1}} \Phi_{k}(x, y)  \overset{1}{\PP}\left(E_1\right)(y) dy dx \notag\\ \nonumber
	 		&=& \eta_2\int_{D_{2}}\overset{1}{\PP}\left(W_2\right)(x) \cdot E_2^{Inc}(x)dx \\ \nonumber &+&  k^2 \eta_2 \eta_1  \int_{D_{2}}\overset{1}{\PP}\left(W_2\right)(x) \cdot \int_{D_{1}}\Phi_{k}(x, y)  \overset{3}{\PP}\left(E_1\right)(y) dy dx \\ &+&  k^2  L_{1}  +  L_{2}  +  k^2  L_{3},  
	 	\end{eqnarray}
   where we have used $(\ref{grad-M-1st}), (\ref{prop-es-W})$ and the fact that  
	 	\begin{equation}\notag
	 		\int_{D_{2}} \overset{2}{\PP}\left(W_2\right)(x) \cdot E_2^{inc}(x)\,dx \, = \, \int_{D_{2}}\overset{1}{\PP}\left(\mathcal{P}(x, z_2)\right) \cdot \overset{3}{\PP}\left(E_2\right)(x)\,dx \, = \, 0.
	 	\end{equation}
	Besides we have used the following notations,
	 	\begin{eqnarray}\notag
	 		L_{1} \, &:=& \, \eta_2 \, \eta_1 \, \int_{D_{2}}\overset{3}{\PP}\left(W_2\right)(x) \cdot \int_{D_{1}} \Phi_{k}(x, y) \, \overset{1}{\PP}\left(E_1\right)(y)\,dy\,dx\notag\\
	 		L_2 \, &:=& \,  \eta_2 \, \int_{D_{2}}\overset{3}{\PP}\left(W_2\right)(x) \cdot E_2^{Inc}(x)\,dx\notag\\
	 		L_{3} \, &:=& \,  \eta_2 \, \eta_1 \, \int_{D_{2}}\overset{3}{\PP}\left(W_2\right)(x) \cdot \int_{D_{1}}\Upsilon_{k}(x, y) \cdot \overset{3}{\PP}\left(E_1\right)(y)\,dy\,dx.\notag
	 	\end{eqnarray}
   Next, we estimate $L_{1}, L_{2}$ and $L_{3}$. 
   \begin{enumerate}
       \item[]
       \item Estimation of $L_{1}$. \\
	 By Taylor expansion for $\Phi_{k}(x, \cdot)$ and the relation $(\ref{RefNeeded1})$, we obtain 
     \begin{eqnarray}\label{MLEqua0919}
     \nonumber
         L_{1} \,	&=& \, \eta_2 \, \eta_1 \,  \int_{D_{2}}\overset{3}{\PP}\left(W_2\right)(x) \cdot \int_{D_{1}} \int_{0}^{1} \underset{y}{\nabla}\Phi_{k}(x, z_1+t(y-z_{1})) \cdot (y - z_1) \, dt \, \overset{1}{\PP}\left(E_1\right)(y)\,dy\,dx \\
        \left\vert L_{1} \right\vert \, & \overset{(\ref{def-eta})}{\lesssim} & \,  a^5 \, d^{-2} \,  \left\lVert\overset{3}{\PP}\left(\tilde{W}_2\right)\right\rVert_{\mathbb{L}^2(B_2)}\left\lVert\overset{1}{\PP}\left(\tilde{E}_1\right)\right\rVert_{\mathbb{L}^2(B_1)} \, \underset{(\ref{max-P1P3})}{\overset{(\ref{prop-es-W})}{=}} \, \mathcal{O}\left( a^{11-2h} d^{-2}\right).
     \end{eqnarray}
    \item[]
    \item Estimation of $L_{2}$. \\  
    For $L_2$, we directly have
	 	\begin{equation}\label{es-L2}
	 	\left\vert L_{2} \right\vert \, \leq \, \left\vert \eta_2 \right\vert \, \left\lVert\overset{3}{\PP}\left(W_2\right)\right\rVert_{\mathbb{L}^2(D_2)} \, \left\lVert E_2^{Inc}\right\rVert_{\mathbb{L}^2(D_2)} \, \overset{(\ref{def-eta})}{=} \, \mathcal{O}\left( a^3 \, \left\lVert\overset{3}{\PP}\left(\tilde{W}_2\right)\right\rVert_{\mathbb{L}^2(B_2)} \right)  \, \overset{(\ref{prop-es-W})}{=} \, \mathcal{O}\left( a^{8-h} \right).
	 	\end{equation}
   \item[]
   \item Estimation of $L_{3}$. \\ 
       For $L_{3}$, we directly get by using the Calder\'{o}n-Zygmund inequality that  
	 	\begin{equation}\label{es-L3i}
	 	 \left\vert L_{3} \right\vert \,  \lesssim \, a\,  \left\lVert\overset{3}{\PP}\left(\tilde{W}_2\right)\right\rVert_{\mathbb{L}^2(B_2)} \, \left\lVert\overset{3}{\PP}\left(\tilde{E}_1 \right)\right\rVert_{\mathbb{L}^2(B_1)}  \, \underset{(\ref{max-P1P3})}{\overset{(\ref{prop-es-W})}{=}} \, \mathcal{O}\left( a^{8-h} \right).
	 	\end{equation}	
    \end{enumerate}	
	 	Combining $(\ref{MLEqua0919}), (\ref{es-L2})$ and \eqref{es-L3i}, for $h<1$, it allows us to write \eqref{LS-D2-P1-main} as
	 	\begin{eqnarray}\label{mid-D2-1}
	 		\eta_2 \, \int_{D_{2}}\mathcal{P}(x, z_2) \cdot \overset{1}{\PP}\left(E_2\right)(x)\,dx \, &-& k^2\, \eta_2 \, \eta_1 \,  \int_{D_{2}}\overset{1}{\PP}\left(W_2\right)(x) \cdot \int_{D_{1}}\Phi_{k}(x, y) \, \overset{1}{\PP}\left(E_1\right)(y)\,dy\,dx\notag\\ \nonumber
	 		&=& \eta_2\int_{D_{2}}\overset{1}{\PP}\left(W_2\right)(x) \cdot  E_2^{Inc}(x)\,dx \\ &+& k^2\, \eta_2 \, \eta_1 \, \int_{D_{2}}\overset{1}{\PP}\left(W_2\right)(x) \cdot \int_{D_{1}}\Phi_{k}(x, y) \, \overset{3}{\PP}\left(E_1\right)(y)\,dy\,dx \, + \, \mathcal{O}\left( a^{8-h}\right).
	 	\end{eqnarray}
   For the above equation, we need to analyze the terms containing $\overset{1}{\PP}\left(W_2\right)$. To do this, we recall $(\ref{def-F_j})$ that we have  
	 	\begin{equation}\label{W2A2}
	 		\overset{1}{\PP}\left(W_2\right) \, = \, Curl\left( A_2 \right), \quad \mbox{with}\quad \div \left( A_2 \right) \, = \, 0 \quad\mbox{and}\quad \nu\times A_2=0. 
	 	\end{equation}
   Then, by using $(\ref{W2A2})$ and $(\ref{def-F_j})$ in $(\ref{mid-D2-1})$, with help of integration by parts, we end up with the following equation, 
	 	\begin{eqnarray}\label{mid3-D2-P1}
	  \eta_2 \, \mathcal{K}  \cdot  \int_{D_{2}}F_2(x)\,dx \, &-& \, k^4 \, \eta_2 \, \eta_1 \, \int_{D_{2}}A_2(x) \cdot \int_{D_{1}}\Upsilon_{k}(x, y) \cdot F_1(y)\,dy\,dx\notag\\  \nonumber
	 		&=& \, i \, k \, \eta_2 \, \int_{D_{2}} A_2(x) \cdot  H^{Inc}_2(x)\,dx \\ &+& \, k^2 \, \eta_2 \, \eta_1 \, \int_{D_{2}}A_2(x) \cdot \int_{D_{1}}\underset{x}{\nabla}\Phi_{k}(x, y)\times\overset{3}{\PP}\left(E_1\right)(y)\,dy + \mathcal{O}\left( a^{ 8-h}\right), 
	 	\end{eqnarray}
   where $\mathcal{K}$ is the matrix given by $(\ref{def-K})$. Now, by expanding $\Upsilon_{k}(\cdot, \cdot)$ and $\nabla\Phi_{k}(\cdot, \cdot)$, at $(z_2, z_1)$, we derive from \eqref{mid3-D2-P1} the following equation
	 	\begin{eqnarray}\label{mid4-D2-P1}
	 	\eta_2 \,	\mathcal{K}\cdot\int_{D_{2}} F_2(x)\,dx &-& \, k^4 \, \eta_1 \, \eta_2 \, \int_{D_{2}} A_2(x)\,dx\cdot\Upsilon_{k}(z_2, z_1)\cdot\int_{D_{1}}F_1(y)\,dy\notag\\ \nonumber
	 		&=& \, i \, k \, \eta_2 \, \int_{D_{2}} A_2(x)\,dx\cdot H^{Inc}_2(z_2) \\ \nonumber &+& \, k^2 \, \eta_1 \, \eta_2 \, \int_{D_{2}}A_2(x)\,dx\cdot\underset{x}{\nabla}\Phi_{k}(z_2, z_1)\times\int_{D_{1}}\overset{3}{\PP}\left(E_1\right)(y)\,dy \\ &+& \, L_{4} \, + \, L_{5} \, + \, L_{6} \, + \, L_{7} \, + \, L_{8} \, + \, \mathcal{O}\left( a^{8-h}\right).  
	 	\end{eqnarray}
   Next, we set and we estimate the terms $L_{4}, \cdots, L_{8}$.
   \begin{enumerate}
       \item[] 
       \item Estimation of $L_{4}$.
	 	\begin{eqnarray*}
	 		L_4 \, &:=& \, k^4 \, \eta_1 \, \eta_2 \, \int_{D_{2}}A_2(x) \cdot \int_{D_{1}}\int_0^1 \underset{x}{\nabla}(\Upsilon_k)(z_2+t(x-z_2), z_1) \cdot \mathcal{P}(x, z_2)\,dt\cdot F_1(y)\,dy\,dx,\notag\\
	 		\left\vert L_4 \right\vert &\lesssim& \left\vert \eta_1 \right\vert \left\vert \eta_2 \right\vert \left\lVert A_2\right\rVert_{\mathbb{L}^2(D_2)}\left\lVert \int_{D_{1}}\int_0^1 \underset{x}{\nabla}(\Upsilon_k)(z_2+t(\cdot-z_2), z_1) \cdot \mathcal{P}(x, z_2)\,dt\cdot F_1(y)\,dy \right\rVert_{\mathbb{L}^2(D_2)}\notag\\
	 		&\lesssim& a^2 d^{-4} \left\lVert A_2\right\rVert_{\mathbb{L}^2(D_2)}\left\lVert F_1\right\rVert_{\mathbb{L}^2(D_1)}, 
	 	\end{eqnarray*}
   which, after scaling from $D_{j}$ to $B_{j}$, with $j = 1, 2$, and using the Friedrichs Inequality, we obtain 
   \begin{equation}\label{EstL4}
      L_{4} \, = \, \mathcal{O}\left( a^7 \, d^{-4} \, \left\lVert\overset{1}{\PP}\left(\tilde{W}_2\right)\right\rVert_{\mathbb{L}^2(B_2)}\left\lVert\overset{1}{\PP}\left(\tilde{E}_1\right)\right\rVert_{\mathbb{L}^2(B_1)}\right).
   \end{equation}
   \item[]
   \item  Estimation of $L_{5}$. Following the same argument, we have
	 	\begin{eqnarray}\label{EstL5}
	 		L_5 \, &:=& \, k^4 \, \eta_1 \, \eta_2 \,  \int_{D_{2}}A_2(x) \cdot \int_{D_{1}}\int_0^1 \underset{y}{\nabla}(\Upsilon_k)(z_2, z_1+t(y-z_1)) \cdot \mathcal{P}(y, z_1)\,dt\cdot F_1(y)\,dy\,dx,\notag\\
	 		|L_5|&\lesssim& a^7 d^{-4}\left\lVert\overset{1}{\PP}\left(\tilde{W}_2\right)\right\rVert_{\mathbb{L}^2(B_2)}\left\lVert\overset{1}{\PP}\left(\tilde{E}_1\right)\right\rVert_{\mathbb{L}^2(B_1)}.
	 	\end{eqnarray}
   \item[]
   \item Estimation of $J_{6}$.
	 	\begin{eqnarray}\label{EstL6}
	 		L_6&:=& i \, k \, \eta_2 \,  \int_{D_{2}}A_2(x) \cdot \int_0^1 \nabla H^{Inc}_2(z_2+t(x-z_2)) \cdot (x-z_2)\,dt\,dx,\notag\\
	 		|L_6|&\lesssim& \left\vert \eta_2 \right\vert \, \left\lVert A_2\right\rVert_{\mathbb{L}^2(D_2)}\left\lVert \int_0^1 \nabla H^{Inc}_2(z_2+t(\cdot - z_2)) \cdot (\cdot - z_2)\,dt\right\rVert_{\mathbb{L}^2(D_2)}\notag\\
	 		& \underset{(\ref{def-eta})}{
    \overset{(\ref{EincHinc})}{\lesssim}} &  a^4  \, \left\lVert \tilde{A}_2\right\rVert_{\mathbb{L}^2(B_2)} \; = \; \mathcal{O}\left( a^5 \,  \left\lVert\overset{1}{\PP}\left(\tilde{W}_2\right)\right\rVert_{\mathbb{L}^2(B_2)}\right),
	 	\end{eqnarray}
   where the last estimation is justified by the use of Friedrichs Inequality. 
   \item[]
   \item Estimation of $J_{7}$. 
	 	\begin{eqnarray}\label{EstL7}
	 		L_7 \, &:=& \, k^2 \,  \eta_1 \, \eta_2 \, \int_{D_{2}}A_2(x) \cdot \int_{D_{1}}\int_0^1 \underset{x}{\nabla}(\underset{x}{\nabla} \Phi_{k})(z_2+t(x-z_2), z_1) \cdot (x-z_2)\,dt\times\overset{3}{\PP}\left(E_1\right)(y)\,dy\,dx,\notag\\
	 		&& \qquad \overset{(\ref{def-eta})}{\lesssim} a^2 d^{-3} \left\lVert A_2\right\rVert_{\mathbb{L}^2(D_2)}\left\lVert\overset{3}{\PP}\left(E_1\right)\right\rVert_{\mathbb{L}^2(D_1)}\lesssim a^6 d^{-3} \left\lVert\overset{1}{\PP}\left(\tilde{W}_2\right)\right\rVert_{\mathbb{L}^2(B_2)}\left\lVert\overset{3}{\PP}\left(\tilde{E}_1\right)\right\rVert_{\mathbb{L}^2(B_1)}. 
	 	\end{eqnarray}
   \item[]
   \item Estimation of $L_{8}$.
   \begin{equation*}
       L_8 \, := \, k^2 \, \eta_1 \, \eta_2 \, \int_{D_{2}}A_2(x) \cdot \int_{D_{1}}\int_0^1 \underset{y}{\nabla}(\underset{x}{\nabla} \Phi_{k})(z_2, z_1+t(y-z_1)) \cdot (y-z_1)\,dt\times\overset{3}{\PP}\left(E_1\right)(y)\,dy\,dx.
   \end{equation*}
   Observe that $L_{8}$'s expression is almost identical to $L_{7}$'s expression. Hence, from the above estimation, we deduce that
	 	\begin{equation}\label{EstL8}
	 		L_8 \; = \; \mathcal{O}\left(a^6 \, d^{-3} \,  \left\lVert\overset{1}{\PP}\left(\tilde{W}_2\right)\right\rVert_{\mathbb{L}^2(B_2)}\left\lVert\overset{3}{\PP}\left(\tilde{E}_1\right)\right\rVert_{\mathbb{L}^2(B_1)} \right).	 	\end{equation}
   \end{enumerate}
   Finally, by gathering $(\ref{EstL4}), (\ref{EstL5}), (\ref{EstL6}), (\ref{EstL7})$ and $(\ref{EstL8})$, we deduce that with \eqref{d-a},
   \begin{equation}\label{EstL4+...+L8}
       L_{4} + \cdots + L_{8} \, \underset{(\ref{max-P1P3})}{\overset{(\ref{prop-es-W})}{=}} \, \mathcal{O}\left( a^{\min(9-h-4t;  6)}\right).
   \end{equation}
    By returning to $(\ref{mid4-D2-P1})$, using $(\ref{EstL4+...+L8})$ and $(\ref{Qj1})$, we deduce 
	 	\begin{eqnarray}\label{mid5-D2-P1}
   \nonumber
	 		\mathcal{K} \cdot Q_{2} \, &-& \,  k^4 \, \eta_2 \,  \int_{D_{2}}A_2(x)\,dx \cdot \Upsilon_k(z_2, z_1) \cdot Q_{1} \, - \,  k^2 \, \eta_2 \, \int_{D_{2}}A_2 (x)\,dx \cdot \underset{x}{\nabla}\Phi_{k}(z_2, z_1)\times R_{1}\\
	 		&=& \, i \, k \,  \eta_2 \, \int_{D_{2}}A_2 (x)\,dx \cdot H_2^{Inc}(z_2)  \, + \, \mathcal{O}\left(a^{\min(9-h-4t; 6)}\right).
	 	\end{eqnarray}
	 	Similarly to \eqref{def-A1-mu}, by denoting
	 	\begin{equation}\label{def-A21}
	 		\bm{\mathcal{A}_2^{(1)}}:=\frac{1}{4}\left(\mathcal{K}^\perp\cdot\mathcal{K}\cdot\mathcal{K}^\perp\right)\cdot\int_{D_{2}}A_2(x)\,dx,
	 	\end{equation}
        and multiplying the both sides of
	 	\eqref{mid5-D2-P1} by $\frac{1}{4}\left(\mathcal{K}^\perp\cdot\mathcal{K}\cdot\mathcal{K}^\perp\right)$, we obtain
	 	\begin{equation}\label{mid6}
	 	Q_{2}  -  k^2   \eta_2  \bm{\mathcal{A}_2^{(1)}} \cdot \left[ k^2    \Upsilon_{k}(z_2, z_1) \cdot Q_{1}  +  \underset{x}{\nabla}\Phi_{k}(z_2, z_1) \times R_{1} \right] 
	 		=  i  k  \eta_2  \bm{\mathcal{A}_2^{(1)}} \cdot H^{Inc}_2(z_2)  +  \mathcal{O}\left( a^{\min(9-h-4t; 6)}\right).
	 	\end{equation}
	 	
	 	Now, we need to evaluate the tensor $\bm{\mathcal{A}_2^{(1)}}$ given by $(\ref{def-A21})$. Indeed, by scaling from $D_{2}$ to $B_{2}$ and using the fact that $I \, = -\, Curl\left(  \mathcal{Q}(x) \right)$, where $\mathcal{Q}(x)$ is the matrix defined by
   	\begin{equation}\label{DefQ(x)}
	 	\mathcal{Q}(x) \, := \, \begin{pmatrix}
	 		0 & x_3 & 0 \\ 
	 		0 & 0 & x_1 \\ 
	 		x_2 & 0 & 0
	 		\end{pmatrix},
	\end{equation}
   we obtain
        \begin{equation}\notag
	 		\int_{D_{2}} A_2(x)\,dx \; = -\; a^3 \, \int_{B_2} \tilde{A}_2(x)\cdot Curl \mathcal{Q}(x)\,dx,
	 	\end{equation}
	 	which, by integration by parts and the fact that $Curl\left( \tilde{A}_{2} \right) \, = \, a \, \widetilde{Curl\left( A_{2} \right)}$, implies 
	 	\begin{eqnarray*}
	 		\int_{D_{2}}A_2(x)\,dx \, &=& -\, a^4 \, \int_{B_2} \widetilde{Curl A_2}(x)\cdot \mathcal{Q}(x)\,dx \, \overset{(\ref{W2A2})}{=} -\, a^4 \,  \int_{B_2} \overset{1}{\PP}\left(\tilde{W}_2\right)(x)\cdot \mathcal{Q}(x)\,dx\notag\\
	 		&=& - a^4 \, \sum_n \left\langle \tilde{W}_2, e_{n,B_{2}}^{(1)} \right\rangle_{\mathbb{L}^{2}(B_{2})} \otimes \left(\int_{B_2}e_{n,B_{2}}^{(1)}(x) \cdot \mathcal{Q}(x)\,dx\right)^{Tr},
            \end{eqnarray*}
            By keeping the dominant term of $\left\langle \tilde{W}_2, e_{n,B_{2}}^{(1)} \right\rangle_{\mathbb{L}^{2}(B_{2})}$, see for instance $(\ref{Equa1036})$, and using $(\ref{pre-cond})$, we obtain
            \begin{equation}\label{mid7} 
	 	\int_{D_{2}}A_2(x)\,dx \, =  \, a^5 \, \sum_n \left\langle \mathcal{P}(\cdot, 0), e_{n,B_{2}}^{(1)}\right\rangle_{\mathbb{L}^{2}(B_{2})} \, \otimes \left(\int_{B_2} \phi_{n,B_{2}}(x)\,dx\right)^{Tr}\, + \, \O(a^7),
	 	\end{equation}
   due to the fact that
   \begin{eqnarray*}
       &&a^4 \sum_n \left| k^2 a^2 \eta_2 \left\langle N_{B_2}^{-ka}\left(\overset{1}{\PP}(\tilde{W}_2)+\overset{3}{\PP}(\tilde{W}_2)\right), e_{n,B_{2}}^{(3)}\right\rangle_{\mathbb{L}^{2}(B_{2})} \otimes \int_{B_2} e_{n,B_{2}}^{(1)}(x)\cdot\mathcal{Q}(x)\,dx
       \right|\notag\\
       &\lesssim& a^6 \left(\left\lVert\overset{1}{\PP}(\tilde{W}_2)\right\rVert_{\mathbb{L}^2(B_2)}+\left\lVert\overset{3}{\PP}(\tilde{W}_2)\right\rVert_{\mathbb{L}^2(B_2)}\right)\overset{\eqref{prop-es-W}}{\lesssim} a^7.
   \end{eqnarray*}
	 	Since there holds
	 	\begin{eqnarray*}
	 	\left\langle \mathcal{P}(\cdot, 0), e_{n,B_{2}}^{(1)} \right\rangle_{\mathbb{L}^{2}(B_{2})} \, & \overset{\eqref{pre-cond}}{=} & \, \left\langle \mathcal{P}(\cdot, 0), Curl\left(\phi_{n,B_{2}}\right) \right\rangle_{\mathbb{L}^{2}(B_{2})} \\ &=& \, \left\langle Curl\left(\mathcal{P}(\cdot, 0)\right), \phi_{n,B_{2}}\right\rangle_{\mathbb{L}^{2}(B_{2})} \, \overset{\eqref{def-K}}{=} \, \left\langle \mathcal{K}, \phi_{n,B_{2}} \right\rangle_{\mathbb{L}^{2}(B_{2})},
	 	\end{eqnarray*}
	 	then, \eqref{mid7} indicates,
	 	\begin{equation*}\label{Eqau0640}
	 		\int_{D_{2}}A_2(x)\,dx \, = \, a^5 \, \sum_n  \, \left\langle \mathcal{K}, \phi_{n,B_{2}} \right\rangle_{\mathbb{L}^{2}(B_{2})} \otimes \left(\int_{B_2}\phi_{n,B_{2}}(x)\,dx\right)^{Tr}\,+\, \O(a^7).
	 	\end{equation*}
	 	Hence, $(\ref{def-A21})$ takes the following form, 
	 	\begin{equation}\label{es-A21}
	 \bm{\mathcal{A}_2^{(1)}} \, = \, a^5 \, {\bf P}_{0, 2}^{(1)}\,+\, \O(a^7),
	 	\end{equation}
   where the tensor ${\bf P}_{0, 2}^{(1)}$ is given by 
   \begin{equation*}
      {\bf P}_{0, 2}^{(1)} \, := \,  \sum_n  \, \int_{B_{2}} \phi_{n,B_{2}}(x) \, dx \, \otimes \int_{B_{2}} \phi_{n,B_{2}}(x) \, dx.
   \end{equation*}
   Then, by plugging $(\ref{es-A21})$ into $(\ref{mid6})$, we get 
   \begin{eqnarray}\label{Equa0748}
   \nonumber
	 	Q_{2} \, &-& \, k^2 \,  \eta_2 \,  a^5 \, {\bf P}_{0, 2}^{(1)}   \cdot \left[ k^2 \,  \Upsilon_{k}(z_2, z_1) \cdot Q_{1} \, + \, \underset{x}{\nabla}\Phi_{k}(z_2, z_1) \times R_{1} \right] 
	 		= \, i \, k \, \eta_2 \,  a^5 \, {\bf P}_{0, 2}^{(1)}  \cdot H^{Inc}_2(z_2) \\ &+&   \, \mathcal{O}\left(k^2 \, a^{7} \, \left[ k^2 \, \Upsilon_{k}(z_2, z_1) \cdot Q_{1} \, + \, \underset{x}{\nabla}\Phi_{k}(z_2, z_1) \times R_{1} \right] \right) \,  + \, \mathcal{O}\left(a^{\min(9-h-4t;  6)}\right).
	 	\end{eqnarray}
   We need to estimate the second term on the R.H.S of the above equation.  
   \begin{eqnarray*}
       \left\vert \cdots \right\vert \, &:=& \, \left\vert k^2 \, a^{7} \, \left[ k^2 \,  \Upsilon_{k}(z_2, z_1) \cdot Q_{1} \, + \, \underset{x}{\nabla}\Phi_{k}(z_2, z_1) \times R_{1} \right] \right\vert \\
       & \lesssim &  \, a^{7} \, \left[   d^{-3} \, \left\vert Q_{1} \right\vert \, + \, d^{-2} \, \left\vert R_{1} \right\vert \right]  \,
        \underset{(\ref{def-eta})}{\overset{(\ref{Qj1})}{\lesssim}}  \, a^{8} \, \left[   d^{-3} \, \left\Vert \tilde{F}_{1} \right\Vert_{\mathbb{L}^{2}(B_{1})} \, + \, d^{-2} \, \left\Vert \overset{3}{\PP}\left(\tilde{E}_1\right) \right\Vert_{\mathbb{L}^{2}(B_{1})} \right], 
   \end{eqnarray*}
   which, by using Friedrichs Inequality, gives us
   \begin{equation*}
       \left\vert \cdots \right\vert \, \lesssim \, \, a^{8} \, \left[   d^{-3} \, a \, \left\Vert \overset{1}{\PP}\left(\tilde{E}_1\right) \right\Vert_{\mathbb{L}^{2}(B_{1})} \, + \, d^{-2} \, \left\Vert \overset{3}{\PP}\left(\tilde{E}_1\right) \right\Vert_{\mathbb{L}^{2}(B_{1})} \right] \, \overset{(\ref{max-P1P3})}{=} \, \mathcal{O}\left( a^{10-h-3t}\right). 
   \end{equation*}
   Then, by plugging the above estimation into $(\ref{Equa0748})$, we obtain 
      \begin{eqnarray*}
   \nonumber
	 	Q_{2} \, &-& \, k^2 \,  \eta_2 \,  a^5 \, {\bf P}_{0, 2}^{(1)}   \cdot \left[ k^2 \, \Upsilon_{k}(z_2, z_1) \cdot Q_{1} \, + \, \underset{x}{\nabla}\Phi_{k}(z_2, z_1) \times R_{1} \right] \\
	 		&=& \, i \, k \, \eta_2 \,  a^5 \, {\bf P}_{0, 2}^{(1)}  \cdot H^{Inc}_2(z_2) + \, \mathcal{O}\left( a^{\min(9-h-4t; 6)}\right).
	 	\end{eqnarray*}
The R.H.S of the above equation makes sens under the condition 
  \begin{equation*}\label{condition-t,h}
      4 \, - \, h \, - \, 4t \, > \, 0.
  \end{equation*}
	
	 	\item The algebraic system with respect to $\overset{3}{\PP}\left(E_2\right)$.  \\
   By integrating $(\ref{LS-D2-3})$ over $D_2$ and using the definition of 
   $V_2$, given by $(\ref{def-Vm})$, we deduce that
	 	\begin{eqnarray*}\label{LS-P32-2}
	 		\eta_2 \, \int_{D_{2}}\overset{3}{\PP}\left(E_2\right)(x)\,dx \, &-& k^2 \, \eta_2 \, \eta_1 \, \int_{D_{2}} V_2(x) \cdot \int_{D_{1}} \Upsilon_k(x, y) \cdot \overset{3}{\PP}\left(E_1\right)(y)\,dy\,dx\notag\\
	 		&=&\eta_2 \, \int_{D_2}V_2(x) \cdot E_2^{Inc}(x)\,dx \, + \, k^2 \, \eta_1 \, \eta_2 \,\int_{D_2}V_2(x)\int_{D_{1}}\underset{x}{\nabla}\Phi_{k}(x, y)\times F_1(y)\,dy\,dx,
	 	\end{eqnarray*}
	 	where, on the R.H.S, we have used the fact that  
	 	\begin{eqnarray*}\label{mid1-P32}
   \nonumber
	 	\int_{D_{1}}\Upsilon_{k}(x, y) \cdot \overset{1}{\PP}\left(E_1\right)(y)\,dy \, \overset{(\ref{grad-M-1st})}{=} \,  \int_{D_{1}}\Phi_{k}(x, y)\overset{1}{\PP}\left(E_1\right)(y)\,dy \, & 
	 \overset{(\ref{def-F_j})}{=} & \int_{D_{1}}\Phi_{k}(x, y) \, Curl\left(F_1\right)(y)\,dy \\ &=& \int_{D_{1}}\underset{x}{\nabla}\Phi_{k}(x, y)\times F_1(y)\,dy.
	 	\end{eqnarray*}
   Hence, by taking the Taylor expansion of $\Upsilon_{k}(\cdot, \cdot)$ and $\Phi_{k}(\cdot, \cdot)$, at $(z_2, z_1)$, we have 
	 	\begin{eqnarray}\label{LS-P32-expan}
	 		\eta_2 \, \int_{D_{2}}\overset{3}{\PP}\left(E_2\right)(x)\,dx \, &-& \, k^2 \, \eta_2 \, \eta_1 \,  \int_{D_{2}}V_2(x)\,dx\cdot\Upsilon_{k}(z_2, z_1) \cdot \int_{D_{1}}\overset{3}{\PP}\left(E_1\right)(y)\,dy\notag\\
	 		&=& \eta_2\int_{D_{2}}V_2(x)\,dx\cdot E_2^{Inc} (z_2) \, + \,  k^2 \, \eta_1\eta_2\int_{D_{2}}V_2(x)\,dx\cdot\underset{x}{\nabla}\Phi_{k}(z_2, z_1)\times\int_{D_{1}} F_1(y)\,dy\notag\\
            &+& L_{9} + L_{10} + L_{11} + L_{12} + L_{13},
	 	\end{eqnarray}
   where $L_{9}, \cdots, L_{13}$ are defined and evaluated as follows.
   \begin{enumerate}
       \item[]
       \item Estimation of $L_{9}$. 
       \begin{eqnarray}\label{Est-L9}
       \nonumber
           	L_9 \, &:=& \, k^2 \,  \eta_2 \, \eta_1 \, \int_{D_{2}} V_2(x) \cdot \int_{D_{1}} \int_0^1 \underset{x}{\nabla} (\Upsilon_k)(z_2+t(x-z_2), z_1) \cdot \mathcal{P}(x, z_2)\,dt\cdot\overset{3}{\PP}\left(E_1\right)(y)\,dy\,dx \\
                 L_9  \, & \overset{(\ref{def-eta})}{=} & \, \mathcal{O}\left( a^5 d^{-4} \left\lVert\tilde{V}_2\right\rVert_{\mathbb{L}^2(B_2)}\left\lVert\overset{3}{\PP}\left(\tilde{E}_1\right)\right\rVert_{\mathbb{L}^2(B_1)} \right). 
       \end{eqnarray}
       \item[]
       \item Estimation of $L_{10}$. 
	 	\begin{equation*}
		L_{10} \, := \, k^2 \, \eta_2 \, \eta_1 \, \int_{D_{2}} V_2(x) \cdot \int_{D_{1}} \int_0^1 \underset{y}{\nabla} (\Upsilon_k)(z_2, z_1+t(y-z_1)) \cdot \mathcal{P}(y, z_1)\,dt\cdot\overset{3}{\PP}\left(E_1\right)(y)\,dy\,dx.
	 	\end{equation*}
	 	It is easy to see that $L_{10}$'s expression is almost identical to $L_{9}$'s expression. Hence, from $(\ref{Est-L9})$, we get  
        \begin{equation}\label{Est-L10}
            L_{10}  \,=\, \mathcal{O}\left( a^5 d^{-4} \left\lVert\tilde{V}_2\right\rVert_{\mathbb{L}^2(B_2)}\left\lVert\overset{3}{\PP}\left(\tilde{E}_1\right)\right\rVert_{\mathbb{L}^2(B_1)} \right).
        \end{equation}
  \item Estimation of $L_{11}$. 
  \begin{equation}\label{Est-L11}
      L_{11} \, := \, \eta_2\int_{D_{2}}V_2(x)\int_0^1 \underset{x}{\nabla} E_2^{Inc}(z_2+t(x-z_2)) \cdot (x-z_2)\,dt\,dx 
            \underset{(\ref{def-eta})}{\overset{(\ref{EincHinc})}{=}}  \mathcal{O}\left( a^4 \, \left\lVert \tilde{V}_2\right\rVert_{\mathbb{L}^2(B_2)} \right).
  \end{equation}
  \item[]
  \item Estimation of $L_{12}$. 
  \begin{eqnarray}\label{Est-L12}
  \nonumber
      L_{12} \, &:=& \, k^2 \, \eta_1 \, \eta_2 \,  \int_{D_{2}}V_2(x) \cdot \int_{D_{1}}\int_0^1 \underset{x}{\nabla}\left(\underset{x}{\nabla}\Phi_{k}\right)(z_2 + t(x-z_2), z_1) \cdot (x-z_2)\,dt \times F_1(y)\,dy\,dx \\
      & \overset{(\ref{def-eta})}{\lesssim} & \, a^{5} \, d^{-3} \, \left\Vert \tilde{V}_{2} \right\Vert_{\mathbb{L}^{2}(B_{2})} \, \left\Vert \tilde{F}_{1} \right\Vert_{\mathbb{L}^{2}(B_{1})} \, = \, \mathcal{O}\left(  a^6 \, d^{-3} \, \left\lVert \tilde{V}_2\right\rVert_{\mathbb{L}^2(B_2)}\left\lVert\overset{1}{\PP}\left(\tilde{E}_1\right)\right\rVert_{\mathbb{L}^2(B_1)}\right),  
  \end{eqnarray}
  where the last estimation is a consequence of the Friederichs Inequality. 
  \item[]
  \item Estimation of $L_{13}$.
  \begin{equation*}
	 L_{13} \, := \, k^2 \, \eta_1 \, \eta_2 \,  \int_{D_{2}}V_2(x) \cdot \int_{D_{1}}\int_0^1 \underset{y}{\nabla}\left(\underset{x}{\nabla}\Phi_{k}\right)(z_2, z_1+t(y-z_1)) \cdot (y-z_1)\,dt\times F_1(y)\,dy\,dx.
   \end{equation*}
   Observe that $L_{13}$'s expression is almost identical to $L_{12}$'s expression. Hence, from $(\ref{Est-L12})$, we deduce that 
   \begin{equation}\label{Est-L13}
      L_{13}  \, = \, \mathcal{O}\left( a^6 \, d^{-3} \, \left\lVert \tilde{V}_2\right\rVert_{\mathbb{L}^2(B_2)}\left\lVert\overset{1}{\PP}\left(\tilde{E}_1\right)\right\rVert_{\mathbb{L}^2(B_1)}\right).  
  \end{equation}
   \end{enumerate}
   Finally, by gathering $(\ref{Est-L9}), (\ref{Est-L10}), (\ref{Est-L11}),  (\ref{Est-L12}), (\ref{Est-L13})$, and using $(\ref{es-V12})$ and $(\ref{max-P1P3})$, we deduce 
   \begin{equation}\label{Est-L9+...+L13}
       L_{9} + \cdots + L_{13} \, = \, \mathcal{O}\left( a^{4-h} \right) \, + \, \mathcal{O}\left( a^{7-3t-2h}\right)+ \mathcal{O}\left(a^{7-h-4t}\right). 
   \end{equation}
   Hence, by returning to $(\ref{LS-P32-expan})$, using the estimation $(\ref{Est-L9+...+L13})$ and the notations $(\ref{Qj1})$, with 
	 	\begin{equation}\label{A22V2} 		\bm{\mathcal{A}_2^{(2)}} \, := \, \int_{D_{2}} V_2(x)\,dx,
	 	\end{equation}
   we obtain 
   \begin{equation}\label{Equa1336}
	R_{2} \, - \, k^2 \, \eta_2  \,  \bm{\mathcal{A}_2^{(2)}} \cdot \left[ \Upsilon_{k}(z_2, z_1) \cdot R_{1}  \, + \,\underset{x}{\nabla}\Phi_{k}(z_2, z_1)\times Q_{1} \right] \, = \, \eta_2 \, \bm{\mathcal{A}_2^{(2)}} \cdot E_2^{Inc} (z_2) \, + \, \mathcal{O}\left(a^{\min(4-h; 7-3t-2h; 7-h-4t} \right). 
    \end{equation}        
  We need to investigate $\bm{\mathcal{A}_2^{(2)}}$. To do this, in $(\ref{A22V2})$, by scaling from $D_{2}$ to $B_{2}$, we expand $\tilde{V}_{2}$ into the basis $\left( e_{n,B_{2}}^{(3)} \right)_{n \in \mathbb{N}}$, with $(\ref{V2-project})$ to get
	 	\begin{eqnarray*}
	 		\bm{\mathcal{A}_2^{(2)}} 
    \, &=& \, a^3 \, \sum_n \frac{1}{\left( 1 \, + \, \eta_2 \, \lambda_{n}^{(3)}(B_{2}) \right)} \int_{B_2}e_{n,B_{2}}^{(3)}(x)\,dx  \otimes\int_{B_2}e_{n,B_{2}}^{(3)}(x)\,dx \\
    &+& \, a^3 \, \sum_n \frac{\eta_2}{\left( 1 \, + \, \eta_2 \, \lambda_{n}^{(3)}(B_{2})\right)} \, \left\langle \tilde{V}_2, \mathcal{S}^{ka}\left(e_{n,B_{2}}^{(3)}\right) \right\rangle_{\mathbb{L}^{2}(B_{2})}  \otimes\int_{B_2}e_{n,B_{2}}^{(3)}(x)\,dx,
	 	\end{eqnarray*}
	 	with $\mathcal{S}^{ka}(e_{n,B_{2}}^{(3)})$ given by \eqref{def-Ska}. Next, we estimate the second term on the R.H.S. 
   \begin{eqnarray*}
       \left\vert \cdots \right\vert \; &:=& \; \left\vert a^3 \, \sum_n  \, \frac{\eta_2}{\left( 1 \, + \, \eta_2 \, \lambda_{n}^{(3)}(B_{2})\right)} \, \left\langle \tilde{V}_2, \mathcal{S}^{ka}\left(e_{n,B_{2}}^{(3)}\right) \right\rangle_{\mathbb{L}^{2}(B_{2})}  \otimes\int_{B_2}e_{n,B_{2}}^{(3)}(x)\,dx  \right\vert \\
       &\overset{(\ref{condition-on-k})}{ \lesssim }& \;  a^{3-h} \, \sum_n  \, \left\vert \left\langle \tilde{V}_2, \mathcal{S}^{ka}\left(e_{n,B_{2}}^{(3)}\right) \right\rangle_{\mathbb{L}^{2}(B_{2})} \right\vert \, \left\vert \left\langle I, e_{n,B_{2}}^{(3)}\right\rangle_{\mathbb{L}^{2}(B_{2})}  \right\vert \\ & \lesssim & \, a^{3-h} \; \left\Vert \mathcal{S}^{-ka}\left( \tilde{V}_2 \right) \right\Vert_{\mathbb{L}^{2}(B_{2})} \, \overset{(\ref{EstSka})}{\lesssim} \, a^{5-h} \; \left\Vert \tilde{V}_2  \right\Vert_{\mathbb{L}^{2}(B_{2})} \, \overset{(\ref{es-V12})}{=} \left( a^{5-2h} \right). 
   \end{eqnarray*}
  Consequently, 
	 	\begin{eqnarray*}
	 		\bm{\mathcal{A}_2^{(2)}} \, &=& \, a^3 \, \sum_n \frac{1}{\left( 1 \, + \, \eta_2 \, \lambda_{n}^{(3)}(B_{2})\right)} \, \left(\int_{B_2}e_{n,B_{2}}^{(3)}(x)\,dx\right)\otimes \left(\int_{B_2} e_{n,B_{2}}^{(3)}(x)\,dx\right) \, + \, \O\left(a^{5-2h}\right)\notag\\ \nonumber
	 		&\overset{(\ref{condition-on-k})}{=}& \pm \, \frac{a^{3-h}}{d_0 }\left(\int_{B_2}e_{n_{*},B_{2}}^{(3)}(x)\,dx\right)\otimes \left(\int_{B_2} e_{n_{*},B_{2}}^{(3)}(x)\,dx\right) \\ &+& a^3 \,  \sum_{n\neq n_{*}}\frac{1}{\left( 1 \, + \, \eta_2 \, \lambda_{n}^{(3)}(B_{2})\right)} \left(\int_{B_2}e_{n,B_{2}}^{(3)}(x)\,dx\right)\otimes \left(\int_{B_2} e_{n,B_{2}}^{(3)}(x)\,dx\right)\, + \, \O\left( a^{5-2h}\right).
	 	\end{eqnarray*}
            Obviously, regarding the second term, there holds
	 	\begin{eqnarray*}
	 	\left\vert \cdots \right\vert \, &:=& \,	\left| a^3 \sum_{n\neq n_{*}}\frac{1}{\left( 1 \, + \, \eta_2 \, \lambda_{n}^{(3)}(B_{2})\right)} \, \left(\int_{B_2}e_{n,B_{2}}^{(3)}(x)\,dx\right)\otimes \left(\int_{B_2} e_{n,B_{2}}^{(3)}(x)\,dx\right) \right| \\ & \lesssim & \, a^3 \, \sum_{n\neq n_{*}} \left| \left\langle e_{n,B_{2}}^{(3)}, I \right\rangle_{\mathbb{L}^{2}(B_{2})}\right|^2 \, = \, \mathcal{O}\left( a^3 \right).
	 	\end{eqnarray*}
	 	Hence, the tensor $\bm{\mathcal{A}_2^{(2)}}$ becomes, 
	 	\begin{equation}\label{A22P022}
	 	\bm{\mathcal{A}_2^{(2)}} \, = \, \frac{a^{3-h}}{\pm d_0} \, {\bf P}_{0, 2}^{(2)} \, + \, \mathcal{O}\left( a^{3}\right),\quad\mbox{for}\quad h<1,
	 	\end{equation} 
   where the tensor ${\bf P}_{0, 2}^{(2)}$ is given by 
   \begin{equation}\notag
	 		{\bf P}_{0, 2}^{(2)} \, := \, \int_{B_2}e_{n_{*},B_{2}}^{(3)}(x)\,dx \; \otimes \; \int_{B_2} e_{n_{*},B_{2}}^{(3)}(x)\,dx.
	 	\end{equation}
  Then, by plugging $(\ref{A22P022})$ into $(\ref{Equa1336})$, we deduce that 
  \begin{equation*}\label{R2-proof}
	R_{2} \, - \, k^2 \, \eta_2  \, \frac{a^{3-h}}{\pm d_{0}} \, {\bf P}_{0, 2}^{(2)}  \cdot \left[ \Upsilon_{k}(z_2, z_1) \cdot R_{1}  \, + \,\underset{x}{\nabla}\Phi_{k}(z_2, z_1)\times Q_{1} \right] \, = \, \eta_2 \, \frac{a^{3-h}}{\pm d_{0}} \, {\bf P}_{0, 2}^{(2)} \cdot E_2^{Inc} (z_2) \, + \, \widetilde{Error^{**}},
    \end{equation*} 
    where, by using $(\ref{def-eta}), (\ref{max-P1P3})$ and the Friedrichs Inequality, we have that
    \begin{equation}\notag
         \widetilde{Error^{**}} \, := \, 
             \mathcal{O}\left(a^{\min\left(4-h; 7-3t-2h; 7-h-4t; 6-3t \right)} \right).
    \end{equation}
	 	\end{enumerate}
     The proof of Proposition \ref{prop-la-1} is now complete.
        
 \section{Appendix}\label{Appendix}
\subsection{Lorentz model}
Here we show that the conditions $(\ref{def-eta})$ and $(\ref{condition-on-k})$ can be derived from the Lorentz model by choosing appropriate incident frequency $k$. Indeed, recall the Lorentz model for the relative permittivity that
\begin{equation}\label{Lorentz model}
\epsilon_r=1+\dfrac{k_\mathrm{p}^2}{k_0^2-k^2-ik\xi},
\end{equation}
where $k_\mathrm{p}$ is the plasmonic \mk{frequency}, $k_0$ is the undamped \mk{frequency} resonance of the background and $\xi$ is the damping frequency with $\xi \, \ll \, 1$. In $D_m$, with $m=1,2$, we respectively denote $\xi_{m}$ as the damping frequencies, $k_{0, m}$ as the undamped frequencies, and  $k_{p, m}$ as the according plasmonic frequencies. Let
\begin{equation*}
	k_{0, 1} \, := \, k_0, \quad k_{p,1} \, = \, k_p, \quad \mbox{in} \; D_1,
\end{equation*}
and $k_{0, 2}$, $k_{p, 2}$ such that
\begin{equation}\label{assump2-k}
	k_{0, 2}^{2} \, < \, k^2_{0,1} \, < \, k_{0, 2}^{2} \, + \, k_{p, 2}^{2} \quad \mbox{in} \; D_2.
\end{equation}
\begin{enumerate}
     \item[]
     \item Regarding the dielectric resonance.
If the frequency $k$ is chosen to be real and $k^2$ close to the undamped resonance frequency $k^2_{0,1}$ in the format that
\begin{equation}\label{assumo-k-k0}
k^2 \, - \, k_{0,1}^2 \, = \, - \, \frac{k_{p,1}^{2} \, a^2 \,  \lambda_{n_0}^{(1)}(B_{1}) \, k_{0,1}^2}{\left(1 \mp \Re(c_{0}) \, a^{h} \right) \pm \dfrac{\Im^{2}(c_{0}) \, a^{h}}{\left(1 \mp \Re(c_{0}) \, a^{h} \right)} - k_{p,1}^{2} \, a^{2} \, \lambda_{n_0}^{(1)}(B_{1})}  = \, - \, k_{p,1}^{2} \, a^2 \,  \lambda_{n_0}^{(1)}(B_{1}) \, k_{0,1}^2 \, \left[ 1 \, + \, \mathcal{O}(a^{h}) \right],
\end{equation}
and 
\begin{equation*}
k \; \xi_1 \, = \, \pm \, \frac{\Im(c_{0}) \, a^{2} \, (k^{2} - k^{2}_{0,1})}{\left(1 \mp \Re(c_{0}) \, a^{h} \right)} \, = \, \mp \, \Im(c_{0}) \, a^{4} \, \lambda_{n_0}^{(1)}(B_{1}) \, k_{0,1}^2 \, k_{p,1}^{2} \left[ 1 \, + \, \mathcal{O}(a^{h}) \right],
\end{equation*}
where $\lambda_{n_{0}}^{(1)}(B_{1})$ is an eigenvalue of the Newtonian potential operator $N_{B_{1}}(\cdot )$. Then, there holds, 
\begin{equation}\notag
\Re\Big(\eta_1\Big) \, = \, a^{-2} \, \Big(\lambda_{n_0}^{(1)}(B_{1}) \, k_{0,1}^2 \Big)^{-1}\left(1 \, + \, \mathcal{O}(a^h)\right) \quad \text{and} \quad 
\Im\Big(\eta_1\Big) \, = \, \pm \, \Im\Big( c_{0} \Big) \, \Big( \lambda_{n_0}^{(1)}(B_{1}) \, k_{0,1}^2 \Big)^{-1} + \mathcal{O}\left( a^{4} \right),
\end{equation}
which implies that $D_1$ behaves as a dielectric nano-particle. Furthermore, by choosing $k$ satisfying $(\ref{assumo-k-k0})$, we derive that 
\begin{equation*}
    1 \, - \, k^{2} \, \eta_{1} \, a^{2} \, \lambda_{n_0}^{(1)}(B_{1}) \, = \, \pm \, c_{0} \, a^{h}, 
\end{equation*}
see $(\ref{condition-on-k})$. For more details we refer the readers to \cite[Remark 1.1 \& Remark 1.2]{CGS}.
 \item[]
 \item Regarding the plasmonic resonance. Using the introduced notations, we start by recalling the Lorentz model for the permittivity related to the nano-particle $D_2$, given by $(\ref{Lorentz model})$,
 \begin{equation*}
     \epsilon_{r}^{(2)} \, = \, 1 \, + \, \frac{k_{p,2}^2}{k_{0,2}^2-k^2-i k \xi_2} \,  \overset{(\ref{assumo-k-k0})}{=}  \, 1+\frac{k_{p,2}^2}{k_{0,2}^2-\left(k_{0,1}^2 \, - \, k_{p,1}^{2} \, a^2 \, \lambda_{n_{0}}^{(1)}(B_{1}) \, k_{0,1}^{2} \, \left[ 1 \, + \, \mathcal{O}(a^h) \right]\right) \, - \, i \, k \, \xi_2}. 
 \end{equation*}
 Hence, 
 \begin{equation*}
     Re\left( \epsilon_{r}^{(2)} \right) \, =  \, 1 \, + \, \, \frac{ k_{p,2}^2 \left[ k_{0,2}^2 \, - \, k_{0,1}^2 \, + \, k_{p,1}^{2} \, a^2 \, \lambda_{n_{0}}^{(1)}(B_{1}) \, k_{0,1}^{2} \, \left[ 1 \, + \, \mathcal{O}(a^h) \right] \right]}{\left[ k_{0,2}^2-k_{0,1}^2 \, + \, k_{p,1}^{2} \, a^2 \, \lambda_{n_{0}}^{(1)}(B_{1}) \, k_{0,1}^{2} \, \left[ 1 \, + \, \mathcal{O}(a^h) \right] \right]^{2} \, + \, \left( k \, \xi_2 \right)^{2}}.
 \end{equation*}
Besides, by using  $(\ref{assump2-k})$ and the fact that $\xi_{2} \, \ll \, 1$, we can deduce that 
\begin{equation*}
    Re\left( \epsilon_{r}^{(2)} \right) \, <  \, 0.
\end{equation*}
Then $D_2$ behaves like a plasmonic nano-particle. Furthermore, by choosing $k$ such that 
\begin{equation}\label{Equa0551}
    k^{2} \, = \, k^{2}_{0,2} \, + \, \frac{\lambda_{n_{\star}}^{(3)}(B_{2}) \, k^{2}_{p,2} \, \left( 1 \, \mp \, a^{h} \, Re\left( d_{0} \right) \right)}{\left\vert 1 \, \mp \, a^{h} \,  d_{0}  \right\vert^{2}} , 
\end{equation}
and letting $\xi_{2}$ to be given by 
\begin{equation*}
    \xi_{2} \, = \, \pm \, \frac{a^{h} \, Im\left( d_{0}\right) \, \lambda_{n_{\star}}^{(3)}(B_{2}) \, k_{p,2}^{2}}{\left\vert 1 \, \mp \, a^{h} \,  d_{0}  \right\vert \, \sqrt{k^{2}_{0,2} \, \left\vert 1 \, \mp \, a^{h} \,  d_{0}  \right\vert^{2} \, + \, \lambda_{n_{\star}}^{(3)}(B_{2}) \, k^{2}_{p,2} \, \left\vert 1 \, \mp \, a^{h} \,  d_{0}  \right\vert}},
\end{equation*}
we deduce that 
\begin{equation*}
  1 \, + \, \eta_{2} \, \lambda_{n_{\star}}^{(3)}(B_{2}) \, = \, \pm \, d_{0} \, a^{h},
\end{equation*}
see $(\ref{condition-on-k})$.
\end{enumerate}
\medskip

\noindent According to the analysis above, there is a frequency $k$ that can be considered as a dielectric resonance and a plasmonic resonance under certain conditions. Indeed, by equating $(\ref{assumo-k-k0})$ and $(\ref{Equa0551})$, we obtain
    \begin{equation*}
k_{0,1}^2 \left[ 1 \, - \, \frac{k_{p,1}^{2} \, a^2 \,  \lambda_{n_0}^{(1)}(B_{1})}{\left(1 \mp \Re(c_{0}) \, a^{h} \right) \pm \dfrac{\Im^{2}(c_{0}) \, a^{h}}{\left(1 \mp \Re(c_{0}) \, a^{h} \right)} - k_{p,1}^{2} \, a^{2} \, \lambda_{n_0}^{(1)}(B_{1})}   \right] \, = \, k^{2}_{0,2} \, + \, \frac{\lambda_{n_{\star}}^{(3)}(B_{2}) \, k^{2}_{p,2} \, \left( 1 \, \mp \, a^{h} \, Re\left( d_{0} \right) \right)}{\left\vert 1 \, \mp \, a^{h} \,  d_{0}  \right\vert^{2}},
\end{equation*}
which, by keeping only the dominant terms, gives us the following condition
\begin{equation}\label{Cdt;D=P}
    k_{0,1}^2 \, = \, k_{0,2}^2 \, + \, \lambda_{n_{\star}}^{(3)}(B_{2}) \, k_{p,2}^2,
\end{equation}
establishing a correlation between the Lorentz model parameters associated with the dielectric nano-particle $D_{1}$, i.e. $k_{0,1}$, and the parameters associated with the plasmonic nano-particle $D_{2}$, i.e. $k_{0,2}$ and $k_{p,2}$. Hence, under the condition $(\ref{Cdt;D=P})$, and the smallness assumption on both $\xi_{1}$ and $\xi_{2}$, we deduce that the dielectric resonance equals the plasmonic resonance up to an additive small error term. To put it simply, a hybrid dielectric-plasmonic dimer can have a common resonance.

\subsection{Justification of $(\ref{SL1})-(\ref{SL4})$.}\label{AppRemark}
The calculation has been divided into four parts. 
\begin{enumerate}
    \item[] 
    \item Computation of ${\bf P}_{0, 1}^{(1)}$.
    \begin{equation*}
        {\bf P}_{0, 1}^{(1)} \, := \, \int_{B_{1}} \phi_{n_{0},B_{1}}(x) \, dx \, \otimes \, \int_{B_{1}} \phi_{n_{0},B_{1}}(x) \, dx \, = \, \frac{12}{\pi^{3}} \, I_{3},
    \end{equation*}
    where we have assumed that $B_{1}$ is being the unit ball, i.e., $B_{1} \equiv B(0,1)$, and $n_{0} \, = \, 1$, see \cite[Section 3.1, Formula (3.1)]{CGS-Negative-permeability}.
    \item[]
    \item Computation of ${\bf P}_{0, 2}^{(1)}$.
    \begin{equation*}
            {\bf P}_{0, 2}^{(1)} \, := \, \sum_n \int_{B_2}\phi_{n,B_{2}}(x)\,dx \otimes\int_{B_2} \phi_{n,B_{2}}(x)\, dx, 
    \end{equation*}
    which by using the fact that 
    \begin{equation*}
        \int_{B_2}\phi_{n,B_{2}}(x)\,dx \, = \, - \, \int_{B_{2}} Q(x) \cdot e_{n,B_{2}}^{(1)}(x) \, dx,
    \end{equation*}
    we get 
    \begin{eqnarray*}\label{Equa0720}
    {\bf P}_{0, 2}^{(1)} \, & = &
     \sum_n \, \int_{B_{2}} Q(x) \cdot e_{n,B_{2}}^{(1)}(x) \, dx \,  \otimes \, \int_{B_{2}} Q(x) \cdot e_{n,B_{2}}^{(1)}(x) \, dx \\
    &=& \, \sum_n \, \langle Q, e_{n,B_{2}}^{(1)} \rangle_{\mathbb{L}^{2}(B_{2})} \,  \otimes \, \int_{B_{2}} Q(x) \cdot e_{n,B_{2}}^{(1)}(x) \, dx \\
    &=& \, \int_{B_{2}} Q(x) \cdot \sum_n \, \langle Q, e_{n,B_{2}}^{(1)} \rangle_{\mathbb{L}^{2}(B_{2})} \,  \otimes \, e_{n,B_{2}}^{(1)}(x) \, dx \\ &=& \, \int_{B_{2}} Q(x) \cdot \overset{1}{\mathbb{P}}\left(Q\right)(x) \, dx. 
    \end{eqnarray*} 
    \item[] 
    \item Computation of ${\bf P}_{0, 1}^{(2)}$.
    \begin{eqnarray*}
        {\bf P}_{0, 1}^{(2)} \, &:=& \, \sum_n \frac{1}{ \lambda_{n}^{(3)}(B_{1})} \, \int_{B_1} e_{n,B_{1}}^{(3)}(x)\,dx\otimes  \int_{B_1} e_{n,B_{1}}^{(3)}(x)\,dx \\
        &=& \, \sum_n  \, \left\langle I_{3}, e_{n,B_{1}}^{(3)} \right\rangle_{\mathbb{L}^{2}(B_{1})} \otimes  \int_{B_1} \frac{1}{ \lambda_{n}^{(3)}(B_{1})} \, e_{n,B_{1}}^{(3)}(x)\,dx \\
        &\overset{(\ref{B1EigFct})}{=}& \, \sum_n  \, \left\langle I_{3}, e_{n,B_{1}}^{(3)} \right\rangle_{\mathbb{L}^{2}(B_{1})} \otimes  \int_{B_1} \nabla M_{B_{1}}^{-1}\left( e_{n,B_{1}}^{(3)}\right)(x)\,dx \\
        &=& \, \int_{B_1} \, \sum_n  \, \left\langle I_{3}, e_{n,B_{1}}^{(3)} \right\rangle_{\mathbb{L}^{2}(B_{1})} \otimes   \nabla M_{B_{1}}^{-1}\left( e_{n,B_{1}}^{(3)}\right)(x)\,dx \\
         &=& \, \int_{B_1} \, \nabla M_{B_{1}}^{-1}\left( \sum_n  \, \left\langle I_{3}, e_{n,B_{1}}^{(3)} \right\rangle_{\mathbb{L}^{2}(B_{1})} \otimes    e_{n,B_{1}}^{(3)}\right)(x)\,dx \\
         &=& \, \int_{B_1} \, \nabla M_{B_{1}}^{-1}\left( I_{3} \, \chi_{B_{1}} \right)(x)\,dx.
    \end{eqnarray*}
    Besides, in the case where $B_{1}$ is the unit ball, i.e., $B_{1} \equiv B(0,1)$, we know that 
    \begin{equation*}
        \nabla M_{B_{1}}\left( I_{3} \right) \, = \, \frac{1}{3} \, I_{3}, \quad \text{in} \; B_{1},
    \end{equation*}
    see \cite[Formula (1.17)]{GS}. Hence, 
    \begin{equation*}
        {\bf P}_{0, 1}^{(2)} \, = \, \int_{B_1} \, 3 \,  I_{3} \,dx \, = \, 3 \, \left\vert B_{1} \right\vert \, I_{3} \, = \, 4 \, \pi \, I_{3}. 
    \end{equation*}
    \item[]
        \item Computation of ${\bf P}_{0, 2}^{(2)}$.
        \begin{equation*}
                   {\bf P}_{0, 2}^{(2)} \, := \,  \int_{B_2} e_{n_{\star},B_{2}}^{(3)}(x)\,dx\otimes  \int_{B_2} e_{n_{\star},B_{2}}^{(3)}(x)\,dx.
        \end{equation*}
        By letting $B_{2}$ to be the unit ball, i.e. $B_{2} \, = \, B(0,1)$, and the index $n_{\star} \, = \, 1$, we obtain 
    \begin{equation*}
 {\bf P}_{0, 2}^{(2)} \, = \, \frac{4 \, \pi}{27} \, I_{3}, 
\end{equation*}
see for instance \cite[subsection 4.5.3]{ThesisA.G.}.
    \end{enumerate}


\begin{thebibliography}{10}
\bibitem{AhnDyaRae99}
J.~F.~Ahner, V.~V.~Dyakin, V.~Ya.~Raevskii and R.~Ritter,
\newblock{On series solutions of the magnetostatic integral equation, Zh. Vychisl. Mat. Mat. Fiz, number 4, volume 39, pages 630-637, 1999.}
\bibitem{ammari2018mathematical}
H. Ammari, B. Fitzpatrick, H. Kang, M. Ruiz, S. Yu and H. Zhang,
\newblock{Mathematical and computational methods in photonics and phononics, volume 235, American Mathematical Soc, 2018.}
\bibitem{Ammari-Li-Zou-2}
H. Ammari, B. Li and J. Zou,
\newblock{Mathematical analysis of electromagnetic scattering by dielectric nano-particles with high refractive indices, Transactions of the American Mathematical Society, volume 376, no. 01, 39--90, 2023.}
\bibitem{amrouche1998vector}
C. Amrouche, C. Bernardi, M. Dauge and V. Girault, 
\newblock{Vector potentials in three-dimensional non-smooth domains, Mathematical Methods in the Applied Sciences, 21 (9), 1998.}
\bibitem{baffou2009}
G. Baffou, R. Quidant, and C. Girard,
\newblock{Heat generation in plasmonic nanostructures: Influence of morphology, Applied Physics Letters, vol. 94, num. 15, 2009.}
\bibitem{Bohren-Huffmann}
C.F. Bohren and D.R. Huffmann, 
\newblock{Absorption and Scattering of Light by Small Particles,
Wiley–Interscience, New York, 2004.}
\bibitem{CP}
K. Catchpole and A. Polman, 
\newblock{Plasmonic solar cells, Opt. Express  16, 21793-21800 (2008).}


\bibitem{CMS}
X. Cao, A. Mukherjee and M. Sini.
\newblock{Effective Medium Theory for Heat Generation Using Plasmonics: A Parabolic Transmission Problem Driven by the Maxwell System. arXiv:2411.18091}

\bibitem{CGS} 
X. Cao, A. Ghandriche and M. Sini,
\newblock{The electromagnetic waves generated by a cluster of nano-particles with high refractive indices, Journal of the London Mathematical Society, 2023.}
\bibitem{CGS-Optic}
X. Cao, A. Ghandriche and M. Sini,
\newblock{Optical Inversion Using Plasmonic Contrast Agents.
arXiv preprint:2408.13793 (2024).}


\bibitem{CGS-Negative-permeability}
X. Cao, A. Ghandriche and M. Sini,
\newblock{From all-dieletric nanoresonators to extended quasi-static plasmonic resonators. arXiv:2312.15149}



\bibitem{colton2019inverse}
D. Colton and R. Kress,
\newblock{Inverse acoustic and electromagnetic scattering theory, 93, 2019, Springer Nature.}
\bibitem{Dautry-Lions}
R. Dautray and J. L. Lions,
\newblock{Mathematical Analysis and Numerical Methods for Science and Technology Volume 3 Spectral Theory and Applications.}
\newblock{1st ed. Berlin, Heidelberg: Springer Berlin Heidelberg, 2000.}
\bibitem{Dyakin-Rayevskii}
V.V. Dyakin and V.Ya. Rayevskii,
\newblock{Investigation of an equation of electrophysics, U.S.S.R Computational Mathematics and Mathematical Physics, Volume 30, Number 1, Pages 213-217, 1990.}
\bibitem{FZS}
X. Fan, W. Zheng and D. Singh, 
\newblock{Light scattering and surface plasmons on small spherical particles. Light Sci Appl 3, e179 (2014).}
\bibitem{friedman1980mathematical}
M. J. Friedman,
\newblock{Mathematical study of the nonlinear singular integral magnetic field equation. I. SIAM Journal on Applied Mathematics, 39 (1), 14--20, 1980.}
\bibitem{friedman1981mathematical}
M. J. Friedman,
\newblock{Mathematical study of the nonlinear singular integral magnetic field equation. III. SIAM Journal on Mathematical Analysis, volume 12, number 4, pages 536-540, 1981.}
\bibitem{10.2307/2008286}
M. J. Friedman and J. E. Pasciak,
\newblock{Spectral Properties for the Magnetization Integral Operator, Mathematics of Computation, 43 (168), 447--453, 1984.}
\bibitem{ThesisA.G.}
A. Ghandriche, 
\newblock{Mathematical Analysis of Imaging Modalities Using Bubbles or Nano-particles as Contrast Agents. Johannes Kepler University, Linz, Austria. 2022.}
\bibitem{GS}
A. Ghandriche and M. Sini, 
\newblock{Photo-acoustic inversion using plasmonic contrast agents: The full Maxwell model, Journal of Differential Equations, volume 341, 1--78, 2022.}
\bibitem{GS-OA}
A. Ghandriche and M. Sini, 
\newblock{Simultaneous Reconstruction of Optical and Acoustical Properties in Photoacoustic Imaging Using Plasmonics, SIAM Journal on Applied Mathematics, vol. 83, num. 4, 1738-1765, 2023.}
\bibitem{hao2004electromagnetic}
E. Hao and G. C. Schatz, 
\newblock{Electromagnetic fields around silver nano-particles and dimers, The Journal of chemical physics, volume 120, number 1, pages 357--366, 2004.}
\bibitem{J-J:2016}
S. Jahani and Z. Jacob, 
\newblock{All-dielectric metamaterials, Nat. Nanotechnol. 11 (2016) 23--36.}
\bibitem{K-M-B-K-L:2016}
A.I. Kuznetsov, A.E. Miroshnichenko, M.L. Brongersma, Y.S. Kivshar and B. Luk’yanchuk,
\newblock{Optically resonant dielectric nanostructures, Science, vol. 354, num. 6314, 2016.}
\bibitem{Liu-et-al}
 J. Liu, Z. Meng and J. Zhou.
\newblock{High Electric Field Enhancement Induced by Modal Coupling for a Plasmonic Dimer Array on a Metallic Film.}
\newblock{Photonics 2024, 11, 183. https://doi.org/10.3390/photonics11020183}
\bibitem{maier2007plasmonics}
S. A. Maier,
\newblock{Plasmonics: fundamentals and applications, Springer, 2007.}
\bibitem{Mitrea}
D. Mitrea, M. Mitrea and J. Pipher, 
\newblock{Vector potential theory on nonsmooth domains in $\mathbb{R}^{3}$ and applications to electromagnetic scattering. The Journal of Fourier Analysis and Applications 3, 131--192, 1997.}
\bibitem{N-H-Book}
L. Novotny and B. Hecht,
\newblock{Principles of Nano-Optics. 2nd ed. Cambridge University Press; 2012.}
\bibitem{PD-L-K}
R. Paniagua-Domingueza, B. Luk'yanchuk and A. I.Kuznetsova,
\newblock{Control of scattering by isolated dielectric nanoantennas.
Chapter 3 in 'dielectric Metamaterials Fundamentals, Designs, and Applications' Woodhead Publishing Series in Electronic and Optical Materials, 2020.}
\bibitem{Kiselev-el-al}
A. Kiselev and O. J. F. Martin,
\newblock{Controlling the magnetic and electric responses of dielectric nano-particles via near-field coupling.}
\newblock{Phys. Rev. B 106, 205413 – Published 23 November, 2022.}
\bibitem{Raevskii1994}
V. Ya. Raevskii,
\newblock{Some properties of the operators of potential theory and their application to the investigation of the basic equation of electrostatics and magnetostatics,} 
\newblock{Theoretical and Mathematical Physics, Volume 100, Number 3, pages 1040--1045, 1994.}
\bibitem{S-F-Ineq}
B. Schweizer,
\newblock{On Friedrichs inequality, Helmholtz decomposition, vector potentials, and the div-curl lemma, Springer, 2018.}
\bibitem{TS2018}
D. Tzarouchis and A. Sihvola,
\newblock{Light Scattering by a Dielectric Sphere: Perspectives on the Mie Resonances, Applied Sciences, VOL. 8, NUM. 2, 2018.}
\bibitem{Wang-et-al}
J. Wang, S. Wu,    W. Yang  and  X. Tiana. 
\newblock{Strong anapole–plasmon coupling in dielectric–metallic hybrid nanostructures.}
\newblock{Phys. Chem. Chem. Phys., 26, 23429-23437, 2024.}
\end{thebibliography}
\end{document}